\documentclass[11pt]{amsart}
\usepackage{amsmath,amssymb,amsthm, amscd}
\usepackage{pb-diagram, pb-xy}
\usepackage[all]{xy}
\usepackage{graphicx}

\usepackage{hyperref}
\usepackage{verbatim}
\usepackage{xcolor}

\usepackage{mathtools}
\newcommand{\rev}{\mathrm{rev}}
\newcommand{\CA}{\mathcal{A}}
\newcommand{\vpi}{\varpi}
\newcommand{\ug}[1]{\underset{#1}{>}}
\newcommand{\ugc}[1]{\underset{#1}{\gtrdot}}
\newcommand{\ul}[1]{\underset{#1}{<}}
\newcommand{\ulc}[1]{\underset{#1}{\lessdot}}

\theoremstyle{plain}
\newtheorem{lem}{Lemma}[section]
\newtheorem{prop}[lem]{Proposition}
\newtheorem{conj}[lem]{Conjecture}
\newtheorem{cor}[lem]{Corollary}
\newtheorem{thm}[lem]{Theorem}
\newtheorem*{claim*}{Claim}

\theoremstyle{definition}
\newtheorem{definition}[lem]{Definition}

\theoremstyle{remark}
\newtheorem{remark}[lem]{Remark}
\newtheorem{remarks}[lem]{Remarks}
\newtheorem{ex}[lem]{Example}

\numberwithin{equation}{section}

\newcommand{\af}{\mathrm{af}}

\newcommand{\Iaf}{I_\af}
\newcommand{\Waf}{W_\af}

\newcommand{\Qaf}{Q_\af}

\newcommand{\X}{X}
\newcommand{\Xaf}{X_\af}
\newcommand{\Xafv}{X_\af^\vee}
\newcommand{\Xafz}{X_\af^0}

\newcommand{\Phiaf}{\Phi^\af}
\newcommand{\Phiafp}{\Phi^{\af+}}
\newcommand{\Phiafm}{\Phi^{\af-}}
\newcommand{\casetwo}[4]{\left\{ \begin{array}{ll} #1 &\mbox{if $#2$} \\[2mm] #3 &\mbox{if $#4$}\,. \end{array} \right.}
\newcommand{\casetwoc}[4]{\left\{ \begin{array}{ll} #1 &\mbox{if $#2$} \\[2mm] #3 &\mbox{if $#4$}\,, \end{array} \right.}

\newcommand{\edge}[1]{\xrightarrow{#1}}
\newcommand{\barrl}[1]{\xRightarrow{#1}} 
 
\newcommand{\blarrl}[1]{\xLeftarrow{#1}}
\newcommand{\cl}{\mathop{\rm cl}\nolimits}
\newcommand{\bcllt}[1]{\QLS(#1)}

\newcommand{\Ht}{\mathrm{height}}
\newcommand{\wt}{\mathop{\rm wt}\nolimits}

\newcommand{\la}{\lambda}
\newcommand{\lev}{\mathrm{lev}}

\newcommand{\Fg}{\mathfrak{g}}
\newcommand{\Fh}{\mathfrak{h}}

\newcommand{\BZ}{\mathbb{Z}}
\newcommand{\BQ}{\mathbb{Q}}
\newcommand{\BR}{\mathbb{R}}
\newcommand{\BB}{\mathbb{B}}
\newcommand{\BC}{\mathbb{C}}

\newcommand{\Hom}{\mathop{\rm Hom}\nolimits}

\newcommand{\QLS}{\mathop{\rm QLS}\nolimits}
\newcommand{\QB}{\mathop{\rm QB}\nolimits}
\newcommand{\Deg}{\mathop{\rm Deg}\nolimits}

\newcommand{\ext}{\mathop{\rm ext}\nolimits}

\newcommand{\mcr}[1]{\lfloor #1 \rfloor}

\newcommand{\bQB}[1]{\QB_{#1\lambda}(W^{J})}
\newcommand{\wts}[2]{\wt_{\lambda}(#1 \Rightarrow #2)}

\newcommand{\bzero}{{\bf 0}}

\newcommand{\bp}{{\bf p}}
\newcommand{\bq}{{\bf q}}
\newcommand{\bi}{{\bf i}}
\newcommand{\bdeta}{{\boldsymbol \eta}}

\newcommand{\pair}[2]{\langle #1,\,#2 \rangle}
\newcommand{\Bpair}[2]{\bigl\langle #1,\,#2 \bigr\rangle}

\newcommand{\ti}[1]{\widetilde{#1}}
\newcommand{\ha}[1]{\widehat{#1}}

\newcommand{\ud}[1]{\underline{#1}}

\newcommand{\ve}{\varepsilon}
\newcommand{\vp}{\varphi}

\renewcommand{\l}[1] {l^{A}_{#1}}

\newcommand{\inner}[2]{\langle #1,\,#2 \rangle}
\def \sgn{\mathrm{sgn}}
\def \Bzero{\mathbf{0}}
\def \A{\mathcal{A}}

\begin{document}

\title[A uniform model for KR crystals II. Path models and $P=X$]{A uniform model for Kirillov-Reshetikhin crystals II. Alcove model, path model, and $P=X$}

\author[C.~Lenart]{Cristian Lenart}
\address[Cristian Lenart]{Department of Mathematics and Statistics, State University of New York at Albany, 
1400 Washington Avenue, Albany, NY 12222, U.S.A.}
\email{clenart@albany.edu}
\urladdr{http://www.albany.edu/\~{}lenart/}

\author[S.~Naito]{Satoshi Naito}
\address[Satoshi Naito]{Department of Mathematics, Tokyo Institute of Technology,
2-12-1 Oh-Okayama, Meguro-ku, Tokyo 152-8551, Japan}
\email{naito@math.titech.ac.jp}

\author[D.~Sagaki]{Daisuke Sagaki}
\address[Daisuke Sagaki]{Institute of Mathematics, University of Tsukuba, 
Tsukuba, Ibaraki 305-8571, Japan}
\email{sagaki@math.tsukuba.ac.jp}

\author[A.~Schilling]{Anne Schilling}
\address[Anne Schilling]{Department of Mathematics, University of California, One Shields
Avenue, Davis, CA 95616-8633, U.S.A.}
\email{anne@math.ucdavis.edu}
\urladdr{http://www.math.ucdavis.edu/\~{}anne}

\author[M.~Shimozono]{Mark Shimozono}
\address[Mark Shimozono]{
Department of Mathematics, MC 0151,
460 McBryde Hall, Virginia Tech,
225 Stanger St., 
Blacksburg, VA 24061 USA}
\email{mshimo@vt.edu}



\begin{abstract}
We establish the equality of the specialization $P_\lambda(x;q,0)$ of the Macdonald polynomial at $t=0$ with the 
graded character $X_\la(x;q)$ of a tensor product of ``single-column'' Kirillov-Reshetikhin (KR) modules for untwisted 
affine Lie algebras. This is achieved by constructing two uniform combinatorial models for the crystals associated 
with the mentioned tensor products: the quantum alcove model (which is naturally associated to Macdonald polynomials), 
and the quantum Lakshmibai-Seshadri path model. We provide an explicit affine crystal isomorphism between the two 
models, and realize the energy function in both models.  In particular, this gives the first proof of the positivity of the $t = 0$ 
limit of the symmetric Macdonald polynomial in the untwisted and non-simply-laced cases, when it is expressed as a linear 
combination of the irreducible characters for a finite-dimensional simple Lie subalgebra, as well as a 
representation-theoretic meaning of the coefficients in this expression in terms of degree functions.
\end{abstract}

\maketitle

\section{Introduction}
We prove the equality of the specialization $P_\lambda(x;q,0)$ of the Macdonald polynomial 
$P_\lambda(x;q,t)$ (see \cite{machecke} for the definition of Macdonald polynomials) at $t=0$ with the graded 
character $X_\la(x;q)$ of a tensor product of ``single-column'' Kirillov-Reshetikhin (KR) modules \cite{karrym} for all
untwisted affine Lie algebras. This result follows from another important result in this paper, namely the
construction of the explicit isomorphism between two uniform combinatorial models for the affine crystals associated with the mentioned tensor products. 

The first model has its origins in the work of Naito and Sagaki~\cite{NS03,NS05,NS06, NS-London,NS08}, who realized the tensor products 
of single-column KR crystals in terms of so-called projected level-zero affine Lakshmibai-Seshadri (LS) paths. We provide an explicit description 
of these paths (as quantum LS paths), in terms of the parabolic quantum Bruhat graph~\cite{BFP,Po,LS}, which originated from (small) quantum 
cohomology of partial flag manifolds. This part of our work is based on previous results of ours in \cite{LNSSS}, where we study various properties 
of the parabolic quantum Bruhat graph, including two lifts of it: to the Bruhat order on the affine Weyl group, and to Littelmann's level-zero weight 
poset~\cite{Li}; we also provided a precise characterization of the latter. 

In the second part of the paper, we construct an explicit affine crystal isomorphism between the quantum LS path model and the quantum alcove 
model in \cite{LL}. The latter is intimately related to $P_\lambda(x;q,0)$ by the Ram-Yip formula for Macdonald polynomials \cite{RY}. This leads 
us to our $P_\la=X_\la$ result. 

The context of this project has its origins in Ion's observation~\cite{Ion} that, when the affine simple root $\alpha_0$ is short 
(which includes the duals of untwisted affine root systems), $P_\lambda(x;q,0)$ is an affine Demazure character 
(see~\cite{Sa} for type $A$). On the other hand, Fourier and Littelmann~\cite{FL1} showed that, for simply-laced untwisted 
affine Lie algebras, these Demazure characters are graded characters of tensor products of KR modules, and hence of 
local Weyl modules for current algebras, by using results in~\cite{NS05}. 
Combining~\cite{Ion} and~\cite{FL1}, one deduces the equality $P_\la=X_\la$ in the simply-laced untwisted cases.
Our results go beyond this. Let us note that, in the non-simply-laced untwisted affine types, $X_\la$ is, in general, a 
positive (but not explicitly described)
sum of affine Demazure characters, as proved in \cite[Theorem A]{Na}; this means that the underlying tensor product of KR modules is, in fact, larger than a corresponding Demazure module. Because of this, we cannot directly apply the known
results and methods in the simply-laced untwisted cases to prove the equality $P_\la = X_\la$ for all untwisted cases; 
this is why we need the affine crystal isomorphism between the quantum LS path model and the quantum alcove model.
Here we should mention that the ``folding'' procedure (according to an affine Dynkin diagram
automorphism) does not work for the equality $P_{\lambda}=X_{\lambda}$ in the non-simply-laced 
untwisted affine types, since this procedure only gives us Macdonald polynomials associated to 
twisted affine root systems.

Our work reveals other interesting connections, complementing some recent results in the literature. 
Braverman and Finkelberg~\cite{BF2} have shown that, 
for simply-laced untwisted affine root systems, the characters of 
the duals of certain current algebra modules, called global Weyl modules, 
coincide with the characters $\Psi_\la(x;q)$ of the spaces of global sections of 
line bundles on quasi-maps spaces, which arise in the study of quantum cohomology and quantum $K$-theory of the flag 
manifold; in this case, it is also shown that the function $\Psi_\la(x;q)$ is
equal to $P_\la(x;q,0)$ times an explicit product of geometric series whose ratios are powers of $q$,
and these functions are called $q$-Whittaker functions due to their appearance in the quantum group version of 
the Kostant-Whittaker reduction of Etingof~\cite{E} and Sevostyanov~\cite{Se} 
for the $q$-Toda integrable system. More precisely, the functions
$\Psi_\la(x;q)$ are eigenfunctions of the $q$-Toda difference operators, and their generating function 
yields the $K$-theoretic $J$-function of Givental and Lee~\cite{BF1}. 
Note, however, that in the non-simply-laced untwisted cases, the situation differs considerably: 
indeed, the proof in \cite{BF2} of the equality between $\Psi_\la(x;q)$ and 
$P_{\lambda}(x;q,0)$ times the explicit product, does not carry over; 
this is mainly because $X_\la(x;q)$ is not a single affine Demazure character.
Finally, the quantum alcove model arises in Lenart and Postnikov's conjectural description of the quantum product 
by a divisor in quantum $K$-theory~\cite{LP}. We summarize these connections in Figure~\ref{figure.outline}.

\begin{figure}
\begin{center}
\includegraphics[scale=0.8]{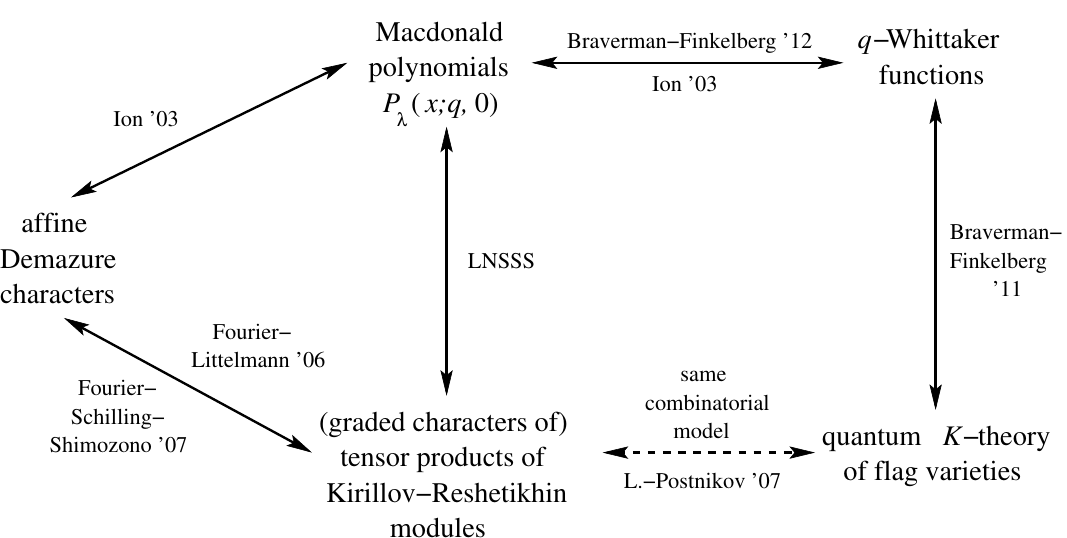}
\end{center}
\caption{Outline \label{figure.outline}}
\end{figure}

Combinatorial models for all nonexceptional KR crystals (not just of column-shape) were given in~\cite{FOS-combinatorics}. 
The quantum LS path model and the quantum alcove model uniformly describe tensor products of column-shape KR crystals, 
for all untwisted affine types. More precisely, these models realize the root operators on the aforementioned tensor products, 
and also give efficient formulas for the corresponding energy function; 
the latter can be viewed as an affine grading on a tensor product of KR crystals~\cite{NS08,ST}, and is used to express one-dimensional 
configuration sums in statistical mechanics~\cite{HKOTT,HKOTY}. Given that, in large rank, certain configuration sums for classical Lie
types were shown to coincide with certain parabolic Lusztig $q$-analogues of weight multiplicity \cite{LOS}, we can also use our energy 
formula to compute the latter. Moreover, by results in~\cite{LOS}, we can calculate the energy on certain tensor products of ``single-row'' KR 
crystals. Another application of the quantum alcove model, which was given in~\cite{LL2}, is a uniform realization of the  {combinatorial 
$R$-matrix} (i.e., the unique affine crystal isomorphism commuting factors in a tensor product of KR crystals).

The quantum LS path model and the quantum alcove model were implemented in the computer algebra system {\sc Sage}~\cite{sage,sagecombinat}. 
Using this implementation, we verified some conjectures related to KR crystals in the exceptional types (except for two Dynkin nodes for type $E_8^{(1)}$); 
these conjectures, which had been previously proved only in the classical types in~\cite{FOS-perfect}, are concerned with the perfectness property of KR 
crystals~\cite{HKOTT}, and with their graded classical decompositions~\cite{HKOTY}. 

There have been several developments related to the work in this paper. Based on our results, an interpretation of the specialization $P_\la(x;q,0)$
above is given in~\cite{ns1,ns2}, in terms of the crystal bases of level-zero extremal weight modules
over quantum affine algebras. On another hand, our work was used in~\cite{CSSW} to provide the 
character of a stable level-one Demazure module associated to type $B_{n}^{(1)}$ as an explicit combination of suitably 
specialized Macdonald polynomials. In addition, our results were used in a crucial way by Chari and Ion in~\cite[Theorem 4.2]{CI} to show 
that Macdonald polynomials at $t=0$ are characters of local Weyl modules for current algebras. Based on this, they prove a 
Bernstein-Gelfand-Gelfand (BGG) reciprocity theorem for the category of representations of a current algebra. In related work, Khoroshkin~\cite{Kho} 
exhibits a categorification of Macdonald polynomials, by realizing them as the Euler characteristic of bigraded characters for certain complexes of 
modules over a current algebra.  This realization simplifies considerably if BGG reciprocity holds (the mentioned complexes become actual 
modules concentrated in homological degree zero). 

The paper is organized as follows. In Sections~\ref{sec:LS} and~\ref{sec:qLS}, we review the affine Lakshmibai-Seshadri (LS) and 
the quantum Lakshmibai-Seshadri path models, respectively. Theorem~\ref{thm:LS=QLS} shows that the set of projected level-zero 
affine LS paths $\BB(\lambda)_{\cl}$ is the same as the set of quantum LS paths $\QLS(\lambda)$, where $\lambda$ is a (level-zero)
dominant integral weight. This fact is also proven in~\cite{LNSSS2} in a somewhat roundabout way, by providing an explicit description of 
the image of a quantum LS path under root operators and showing that the set of quantum LS paths is stable under the action of 
the root operators. (Quantum) LS paths carry a grading by a degree function (which is closely related to the (``right'') energy function
on KR crystals). We provide an explicit formula for the degree function of quantum LS paths in Theorem~\ref{thm:deg-QLS} in terms 
of the parabolic quantum Bruhat graph. For KR crystals, there exist the right and the left energy functions. In Section~\ref{sec:degree}, we 
also relate the left energy function with the degree function using the Lusztig involution. In Section~\ref{section.qalc} the quantum alcove model and 
its crystal structure are defined. In Section~\ref{section.bijection}, we show that there is a bijection 
between the quantum alcove model and the quantum LS path model by exhibiting a forgetful map and its inverse.
We show that up to Kashiwara operators $f_0$ at the end of their strings, there is an affine crystal isomorphism 
between the quantum alcove model and a tensor product of KR crystals. Section~\ref{S:energy} contains 
the main application of this work: by showing that the energy/degree function under the affine crystal isomorphism maps 
to a height function in the quantum alcove model, we show that the graded character of 
a tensor product of single-column KR crystals is equal to a Macdonald polynomial evaluated at $t=0$
(see Corollary~\ref{peqx}). We conclude in Section~\ref{sec:proofs of lemmas} 
with the proof of Lemmas from various sections. In Appendix~\ref{section.perfectness}, we verify the conjectures in \cite{HKOTT,HKOTY} 
mentioned above, in the exceptional types.

We follow the same conventions and notation as in~\cite{LNSSS}.

\subsection*{Acknowledgments}

We would like to thank the Mathematisches Forschungsinstitut Oberwolfach for their support during the Research in Pairs program
and ICERM for hosting the semester program ``Automorphic Forms, Combinatorial Representation Theory and Multiple Dirichlet Series"
in the spring 2013, where some of the ideas were conceived.
We thank Dan Orr for valuable discussions.
We used {\sc Sage}~\cite{sage} and {\sc Sage-combinat}~\cite{sagecombinat} in our work.

C.L. was partially supported by the NSF grants DMS--1101264 and DMS--1362627. 
The hospitality and support of the Max-Planck-Institut f\"ur Mathematik in Bonn, where part of this work was carried out, is also gratefully acknowledged.
S.N. was supported by Grant-in-Aid for Scientific Research (C), No. 24540010, Japan.
D.S. was supported by Grant-in-Aid for Young Scientists (B) No. 23740003, Japan. 
A.S. was partially supported by NSF grants DMS--1001256, OCI--1147247, DMS--1500050, and a grant from the Simons Foundation 
(\#226108 to Anne Schilling).
M.S. was partially supported by the NSF grant DMS--1200804.

%
\section{Lakshmibai-Seshadri paths}
\label{sec:LS}

In this section we review Lakshmibai-Seshadri paths and the corresponding affine crystal
model. As summarized in Theorem~\ref{thm:LScl} and Remark~\ref{rem:LScl}, 
the crystal of level-zero projected LS paths is
isomorphic to tensor products of KR crystals.

%
\subsection{Basic notation}
\label{subsec:notation}
Let $\Fg_{\af}$ be an untwisted affine Lie algebra 
over $\BC$ with Cartan matrix $A=(a_{ij})_{i,\,j \in I_{\af}}$.
The index set $I_{\af}$ of the Dynkin diagram of $\Fg_{\af}$ is  
numbered as in \cite[Section 4.8, Table Aff~1]{Kac}. 
Take the distinguished vertex $0 \in I_{\af}$ as in \cite{Kac}, 
and set $I:=I_{\af} \setminus \{0\}$.
Let 
$\Fh_\af=\bigl(\bigoplus_{j \in I_{\af}} \BC \alpha_{j}^{\vee}\bigr) \oplus \BC d$
denote the Cartan subalgebra of $\Fg_{\af}$, where 
$\bigl\{\alpha_{j}^{\vee}\bigr\}_{j \in I_{\af}} \subset \Fh_\af$ is 
the set of simple coroots, and 
$d \in \Fh_\af$ is the scaling element (or degree operator). 
Also, we denote by 
$\bigl\{\alpha_{j}\bigr\}_{j \in I_{\af}} \subset 
\Fh_{\af}^{\ast}:=\Hom_{\BC}(\Fh_\af,\BC)$ 
the set of simple roots, and by 
$\Lambda_{j} \in \Fh_\af^{\ast}$, $j \in I_{\af}$, 
the fundamental weights; 
note that $\pair{d}{\alpha_{j}}=\delta_{j,0}$ and 
$\pair{d}{\Lambda_{j}}=0$ for $j \in I_{\af}$.
Let $\delta=\sum_{j \in I_{\af}} a_{j}\alpha_{j} \in \Fh_\af^{\ast}$ and 
$c=\sum_{j \in I_{\af}} a^{\vee}_{j} \alpha_{j}^{\vee} \in \Fh_\af$ denote 
the null root and the canonical central element of 
$\Fg_{\af}$, respectively. 
The dual weight lattice $\Xafv$ and the weight lattice $\Xaf$ are defined
as follows:
%
%
\begin{equation} \label{eq:lattices}
\Xafv=
\left(\bigoplus_{j \in I_{\af}} \BZ \alpha_{j}^{\vee}\right) \oplus \BZ d \, 
\subset \Fh_\af
\quad \text{and} \quad 
\Xaf= 
\left(\bigoplus_{j \in I_{\af}} \BZ \Lambda_{j}\right) \oplus 
   \BZ \delta \subset \Fh_\af^{\ast}.
\end{equation}
It is clear that $\Xaf$ contains 
$\Qaf:=\bigoplus_{j \in I_{\af}} \BZ \alpha_{j}$, and that 
$\Xaf \cong \Hom_{\BZ}(\Xafv,\BZ)$. 
We set $\Qaf^{+}:=\sum_{j \in I_{\af}} \BZ_{\ge 0} \alpha_{j}$.
Let $\Fg$ be the classical subalgebra of $\Fg_{\af}$ and denote the finite weight lattice by 
$\X = \bigoplus_{i\in I} \BZ \vpi_i$, where the $\vpi_i$ are the fundamental weights associated to $\Fg$.
We set
\begin{equation*}
Q:=\bigoplus_{j \in I} \BZ \alpha_{j}, \quad 
Q^{+}:=\sum_{j \in I} \BZ_{\ge 0} \alpha_{j}, \quad 
Q^{\vee}:=\bigoplus_{j \in I} \BZ \alpha_{j}^{\vee}, \quad 
Q^{\vee+}:=\sum_{j \in I} \BZ_{\ge 0} \alpha_{j}^{\vee}. 
\end{equation*}

Let $\Waf$ (resp. $W$) be the affine 
(resp. finite) Weyl group with simple reflections $r_i$ for $i\in \Iaf$
(resp. $i\in I$). $\Waf$ acts on $\Xaf$ and $\Xafv$ by
\begin{align*}
  r_i \la &= \la - \pair{\alpha_i^\vee}{\la} \alpha_i \\
  r_i h &= h - \pair{h}{\alpha_i} \alpha^\vee_{i}
\end{align*}
for $i\in \Iaf$, $\la\in \Xaf$, and $h \in \Xafv$.
We denote by $\ell$ the length function on $\Waf$ (resp. $W$), 
and denote by $e$ the identity element of $\Waf$. 
Note that $W_\af \cong W \ltimes Q^\vee$; 
denote by $t_{\xi}$ the image of $\xi \in Q^\vee$ in $W_\af$.

The set of affine real roots (resp. roots)
of $\Fg_\af$ (resp. $\Fg$) are defined by $\Phiaf = \Waf \,\{\alpha_i\mid i\in \Iaf\}$
(resp. $\Phi=W\, \{\alpha_i\mid i\in I\}$).
The set of positive affine real (resp. positive) roots are the
set $\Phiafp = \Phiaf \cap Q_{\af}^{+}$
(resp. $\Phi^+=\Phi\cap Q^{+}$).
We have $\Phiaf=\Phiafp \sqcup \Phiafm$, where $\Phiafm=-\Phiafp$, 
and $\Phi=\Phi^+\sqcup\Phi^-$, where $\Phi^-=-\Phi^+$.
We have $\delta = \alpha_0 + \theta$, where $\theta$ is 
the highest root for $\Fg$, and
\begin{equation*}
\Phiafp = \Phi^+ \sqcup (\Phi+\BZ_{>0}\,\delta).
\end{equation*}
For $\beta \in \Phi^{\af+}$, let $\beta^{\vee}$ denote the coroot of $\beta$, and 
let $r_{\beta} \in W_{\af}$ denote the associated reflection.
Also, we set $\rho = \frac{1}{2}\sum_{\alpha\in\Phi^+} \alpha$. 

The level of a weight $\la\in \Xaf$ is defined by
$\lev(\la) = \pair{c}{\la}$. 
Since the action of $\Waf$ on $\Xaf$ is level-preserving,
the sublattice $\Xafz\subset\Xaf$ of level-zero elements is $\Waf$-stable.
The natural projection $\cl:X_{\af}^{0} \twoheadrightarrow X$ 
has kernel $\BZ\delta$ and sends $\Lambda_{i} - a_{i}^{\vee} \Lambda_{0}$ to 
$\vpi_{i}$ for $i\in I$; also, there is a section $\X\to \Xafz$ given by 
$\vpi_i \mapsto \Lambda_i - a_{i}^{\vee} \Lambda_0$ for $i\in I$.

Let $J$ be a subset of $I$. Denote by $W_J$ 
the parabolic subgroup of $W$ generated by $r_i$ for $i \in J$. 
For $w \in W$, we denote by $\mcr{w} = \mcr{w}^{J}$ 
the minimum-length coset representative in the coset $v W_J$, 
and set $W^{J}:=\bigl\{\mcr{w} \mid w \in W\bigr\}$. 
We set $Q^\vee_J := \bigoplus_{j \in J} \BZ\alpha^\vee_{j}$, 
$Q_J:=\bigoplus_{j \in J} \BZ\alpha_{j}$, and 
$\Phi_J := \Phi \cap Q_{J}$, 
$\Phi_J^{\pm} :=\Phi^{\pm} \cap Q_{J}$; 
note that $\Phi_{J}=\Phi_J^+ \sqcup \Phi_J^-$. 
Also, we set $\rho_J := \frac{1}{2}\sum_{\alpha \in \Phi_J^+} \alpha$. 

Finally, we briefly review the level-zero weight poset. 
Fix a dominant weight $\lambda$ in the finite weight lattice $X$.
We view $X$ as a sublattice of $\Xafz$. 
Let $\Xafz(\lambda)$ be the orbit of $\lambda$ 
under the action of the affine Weyl group $\Waf$.
\begin{definition}[{\cite[Section~4]{Li}}]
\label{definition.level-0 weight poset}
A poset structure is defined on $\Xafz(\lambda)$ 
as the transitive closure of the relation
\begin{equation}\label{lev0anti}
 \mu < r_\beta\mu \quad \iff \quad
 \pair{\beta^\vee}{\mu} >0, 
\end{equation}
where $\beta \in \Phiafp$. This poset is called 
the {\em level-zero weight poset for} $\lambda$.
\end{definition}

%
\subsection{Definition of Lakshmibai-Seshadri paths}
\label{subsec:LS}

In this subsection, 
we fix a dominant integral weight $\lambda \in X$. 
We recall the definition of 
Lakshmibai-Seshadri (LS) paths of shape $\lambda$ from \cite[Section 4]{Li}. 
Let $X_{\af}^{0}(\lambda)$ be the level-zero weight poset for $\lambda$. 
%
%
\begin{definition} \label{dfn:achain}
For $\mu,\,\nu \in X_{\af}^{0}(\lambda)$ with $\nu > \mu$ and 
$b \in \BQ$, a $b$-chain for $(\nu,\mu)$ is, 
by definition, a sequence $\nu=\mu_{0} \gtrdot \mu_{1} \gtrdot
\dots \gtrdot \mu_{m}=\mu$ of covers in $X_{\af}^{0}(\lambda)$ 
such that $b\pair{\beta_{k}^{\vee}}{\mu_{k}} \in \BZ$ 
for all $k=1,\,2,\,\dots,\,m$, 
where $\beta_{k} \in \Phi^{\af+}$ is the corresponding 
positive real root for $\mu_{k-1} \gtrdot \mu_{k}$. 
Here, for $\mu,\,\mu' \in X_{\af}^{0}(\lambda)$, 
the cover $\mu \gtrdot \mu'$ in $X_{\af}^{0}(\lambda)$ 
means that $\mu > \mu'$ in $X_{\af}^{0}(\lambda)$ and that 
there exists no $\nu \in X_{\af}^{0}(\lambda)$ such that $\mu > \nu > \mu'$ 
in $X_{\af}^{0}(\lambda)$.
\end{definition}
%
%
\begin{definition} \label{dfn:LS}
An LS path of shape $\lambda$ is, by definition, 
a pair $\pi=(\ud{\nu}\,;\,\ud{b})$ of a sequence 
$\ud{\nu}:\nu_{1} > \nu_{2} > \cdots > \nu_{s}$ of 
elements in $X_{\af}^{0}(\lambda)$ and a sequence 
$\ud{b}:0=b_{0} < b_{1} < \cdots < b_{s}=1$ of 
rational numbers satisfying the condition that 
there exists a $b_{k}$-chain for $(\nu_{k},\,\nu_{k+1})$ 
for each $k=1,\,2,\,\dots,\,s-1$. 
\end{definition}

Denote by $\BB(\lambda)$ the set of all LS paths of shape $\lambda$.
We identify an element 
\[
\pi=(\nu_{1},\,\nu_{2},\,\dots,\,\nu_{s}\,;\,
b_{0},\,b_{1},\,\dots,\,b_{s}) \in \BB(\lambda)
\]
with the following piecewise-linear, continuous map 
$\pi:[0,1] \rightarrow \BR \otimes_{\BZ} X_{\af}^{0}$:
%
%
\begin{equation} \label{eq:pi}
\pi(t)=\sum_{k=1}^{l-1}
(b_{k}-b_{k-1})\nu_{k}+
(t-b_{l-1})\nu_{l} \quad \text{for \ }
b_{l-1} \le t \le b_{l}, \  1 \le l \le s.
\end{equation}
%
%
\begin{remark} \label{rem:straight}
It follows from the definition of an LS path of shape $\lambda$ 
that $\pi_{\nu}:=(\nu\,;\,0,\,1) \in \BB(\lambda)$ for every 
$\nu \in X_{\af}^{0}(\lambda)$, which corresponds to 
the straight line $\pi_{\nu}(t)=t\nu$, $t \in [0,1]$.
\end{remark}

Recall that $X_{\af}^{0}/\BZ\delta \cong X$. 
Denote by 
\begin{equation*}
\cl:\BR \otimes_{\BZ} X_{\af}^{0} 
\twoheadrightarrow \BR \otimes_{\BZ} X_{\af}^{0}/\BR\delta 
\cong \BR \otimes_{\BZ} X
\end{equation*}
the canonical projection; remark that
$\cl(X_{\af}^{0}(\lambda))=W\lambda \cong W^{J}$ 
(see \cite[Lemma 3.1]{LNSSS}). 
For $\pi \in \BB(\lambda)$, we define $\cl(\pi)$ by: 
$(\cl(\pi))(t)=\cl(\pi(t))$ for $t \in [0,1]$; 
note that $\cl(\pi)$ is a piecewise linear, continuous map 
from $[0,1]$ to $\BR \otimes_{\BZ} X$. Then we set 
\begin{equation*}
\BB(\lambda)_{\cl}:=\bigl\{\cl(\pi) \mid \pi \in \BB(\lambda) \bigr\};
\end{equation*}
an element of this set is called a projected level-zero LS path.

%
\subsection{Crystal structures on 
  \texorpdfstring{$\BB(\lambda)$}{B(lambda)} and 
  \texorpdfstring{$\BB(\lambda)_{\cl}$}{B(lambda)cl}
}
\label{subsec:crystal}
As in the previous subsection, let $\lambda \in X$ be a dominant integral weight. 
We use the following notation:
\begin{equation}\label{tialpha}
\ti{\alpha}_{i}:=
\begin{cases}
\alpha_{i} & \text{if $i \ne 0$}, \\[1.5mm]
-\theta & \text{if $i = 0$},
\end{cases}
\qquad
s_{i}:=
\begin{cases}
r_{i} & \text{if $i \ne 0$}, \\[1.5mm]
r_{\theta} & \text{if $i = 0$},
\end{cases}
\end{equation}
where $\theta$ is the highest root for $\Fg$.

Following \cite{Li}, we give $\BB(\lambda)$ and $\BB(\lambda)_{\cl}$ 
crystal structures with the weight lattices $X_{\af}^{0}$ and $\cl(X_{\af}^{0}) \cong X$,
respectively. Here we focus on the crystal structure 
on $\BB(\lambda)_{\cl}$; for the crystal structure on $\BB(\lambda)$, 
in the argument below, replace $\eta \in \BB(\lambda)_{\cl}$ 
with $\pi \in \BB(\lambda)$, and then replace $\ti{\alpha}_{j} \in \Phi$ and 
$s_{j} \in W$ with $\alpha_{j} \in \Phi^{\af}$ and $r_{j} \in W_{\af}$. 

Let $\eta \in \BB(\lambda)_{\cl}$. 
We see from \cite[Lemma 4.5\,a)]{Li} that 
$\eta(1) \in \cl(X_{\af}^{0}) \cong X$. So we set
\begin{equation*}
\wt (\eta):=\eta(1) \in X. 
\end{equation*}
Next we define root operators $e_{j}$ and $f_{j}$ 
for $j \in I_{\af}=I \sqcup \{0\}$ as follows 
(see \cite[Section 1]{Li}). We set
\begin{equation}
\label{eq:H m}
\begin{array}{l}
H(t)=H^{\eta}_{j}(t):=\pair{\ti{\alpha}_{j}^{\vee}}{\eta(t)}
  \quad \text{for \,} t \in [0,1], \\[3mm]
m=m^{\eta}_{j}
  :=\min\bigl\{H^{\eta}_{j}(t) \mid t \in [0,1]\bigr\}\;.
\end{array}
\end{equation}
It follows from \cite[Lemma 4.5\,d)]{Li} that 
all local minima of $H(t)$ are integers; 
in particular, $m \in \BZ_{\le 0}$. 
If $m = 0$, then $e_{j}\eta:=\bzero$, 
where $\bzero$ is an extra element not 
contained in $\BB(\lambda)_{\cl}$. 
If $m \le -1$, then set
\begin{equation*}
\begin{array}{l}
t_{1}:=\min\bigl\{t \in [0,1] \mid 
       H(t)=m \bigr\}, \\[3mm]
t_{0}:=\max\bigl\{t \in [0,t_{1}] \mid
       H(t) = m+1 \bigr\}.
\end{array}
\end{equation*}
%
%
\begin{remark} \label{rem:ro} \mbox{}
\begin{enumerate}
\item
Recall that all local minima of $H(t)$ are integers 
by \cite[Lemma 4.5\,d)]{Li}. Hence we deduce that 
$H(t)$ is strictly decreasing on $[t_{0},\,t_{1}]$. 
\item
Because $H(t)$ attains the minimum $m$ at $t=t_{1}$, 
it follows immediately that $H(t_{1}+\varepsilon) \ge H(t_{1})$ 
for sufficiently small $\varepsilon > 0$. 
\item
We deduce that $H(t_{0}-\varepsilon) \ge H(t_{0})$ 
for sufficiently small $\varepsilon > 0$. 
Indeed, suppose that $H(t_{0}-\varepsilon) < H(t_{0})$. 
Then the minimum $m'$ of $H(t)$ on $[0,\,t_{0}]$ is 
less than $H(t_{0})=m+1$. Since all local minima of 
$H(t)$ are integers, we obtain $m'=m$. However, this contradicts 
the definition of $t_{1}$; recall that $t_{0} < t_{1}$. 
\end{enumerate}
\end{remark}

Define $e_{j}\eta$ for $j\in I_{\af}$ by:
\begin{equation}\label{defeqls}
(e_{j}\eta)(t)=
\begin{cases}
\eta(t) & \text{if } 0 \le t \le t_{0},  \\[2mm]
\eta(t_{0})+s_{j}(\eta(t)-\eta(t_{0}))
       & \text{if } t_{0} \le t \le t_{1}, \\[2mm]
\eta(t)+\ti{\alpha}_{j} & \text{if } t_{1} \le t \le 1,
\end{cases}
\end{equation}
where $s_{j} \in W$ is 
the reflection with respect to $\ti{\alpha}_{j} \in \Phi$. 
We see from \cite[Corollary 2\,a)]{Li} that 
$e_{j}\eta \in \BB(\lambda)_{\cl}$. 
The definition of $f_{j}\eta \in 
\BB(\lambda)_{\cl} \cup \{\bzero\}$ is similar 
(see also \cite[Section 2.2]{NS08}). 
In addition, for $\eta \in \BB(\lambda)_{\cl}$ and $j \in I_{\af}$, 
we set
\begin{equation}
\label{eq:phi eps}
\ve_{j}(\eta):=
\max\bigl\{n \ge 0 \mid 
 e_{j}^{n}\eta \ne \bzero \bigr\}, \qquad
\vp_{j}(\eta):=
\max\bigl\{n \ge 0 \mid 
 f_{j}^{n}\eta \ne \bzero \bigr\}.
\end{equation}
We see from \cite[Section 2]{Li} that 
the set $\BB(\lambda)_{\cl}$ together with 
the map $\wt:\BB(\lambda)_{\cl} \rightarrow X$, 
the root operators $e_{j}$, $f_{j}$, $j \in I_{\af}$, and 
the maps $\ve_{j},\,\vp_{j}$, $j \in I_{\af}$, becomes 
a crystal with $\cl(X_{\af}^{0}(\lambda)) \cong X$ 
the weight lattice. 

\begin{remark}
It is easily verified that 
\begin{align*}
& \wt(\cl(\pi))=\cl(\wt(\pi))
  \quad \text{for $\pi \in \BB(\lambda)$}, \\
& \cl(e_{j}\pi)=e_{j}\cl(\pi) \quad \text{and} \quad \cl(f_{j}\pi)=f_{j}\cl(\pi)
  \quad \text{for $\pi \in \BB(\lambda)$ and $j \in I_{\af}$}, \\
& \ve_{j}(\cl(\pi))=\ve_{j}(\pi) \quad \text{and} \quad 
  \vp_{j}(\cl(\pi))=\vp_{j}(\pi) \quad 
  \text{for $\pi \in \BB(\lambda)$ and $j \in I_{\af}$}.
\end{align*}
\end{remark}

We know the following theorem from~\cite{NS03, NS05, NS06}.
%
%
\begin{thm} \label{thm:LScl}
\mbox{}
\begin{enumerate}
\item
For each $i \in I$, the crystal $\BB(\vpi_{i})_{\cl}$ is 
isomorphic to the crystal basis of $W(\vpi_{i})$, 
the level-zero fundamental representation of the quantum affine algebra 
$U_{q}'(\Fg_{\af})$ (without the degree operator), 
introduced by Kashiwara \cite{Kas-OnL}.
\item
The crystal graph of $\BB(\lambda)_{\cl}$ is connected. 
\item
Let $\bi=(i_{1},\,i_{2},\,\dots,\,i_{p})$ be 
an arbitrary sequence of elements of $I$ 
(with repetitions allowed), and set 
$\lambda_{\bi}:=\vpi_{i_{1}}+\vpi_{i_{2}}+ \cdots + \vpi_{i_{p}}$. 
Then, there exists an isomorphism 
$\Psi_{\bi}:\BB(\lambda_{\bi})_{\cl} \stackrel{\sim}{\rightarrow} 
\BB(\vpi_{i_{1}})_{\cl} \otimes \BB(\vpi_{i_{2}})_{\cl} \otimes 
\cdots \otimes \BB(\vpi_{i_{p}})_{\cl}$ of crystals. 
\end{enumerate}
\end{thm}
%
%
\begin{remark} \label{rem:LScl}
It is known that 
the fundamental representation $W(\vpi_{i})$ of level-zero
is isomorphic to the Kirillov-Reshetikhin (KR) module 
$W_{1}^{(i)}$ in the sense of \cite[Section 2.3]{HKOTT}
(for the Drinfeld polynomials of $W(\vpi_{i})$, 
see \cite[Remark~3.3]{N}). 
Also we can prove that the crystal basis of 
$W(\vpi_{i}) \cong W_{1}^{(i)}$ is unique, 
up to a nonzero constant multiple 
(see also \cite[Lemma 1.5.3]{NS-CMP}); 
we call this crystal basis a (one-column) KR crystal, 
and denote by $B^{i,1}$. 
By the theorem above, the crystal $\BB(\lambda)_{\cl}$ 
of projected level-zero LS paths of shape $\lambda$ is a model 
for the corresponding tensor product of KR crystals.
\end{remark}

In this paper we use the Kashiwara convention for the tensor product. More
precisely, for two (normal) crystals $B_1$ and $B_2$, the tensor product 
$B_1\otimes B_2$ as a set is the Cartesian product of the two sets. 
For $b=b_1 \otimes b_2\in B_1 \otimes B_2$, the weight function is simply
$\wt(b) = \wt(b_1) + \wt(b_2)$. In the Kashiwara convention the crystal 
operators are given by
\begin{equation*}
	f_i (b_1 \otimes b_2)=
	\begin{cases}
	b_1 \otimes f_i(b_2) & \text{if $\varepsilon _i(b_2) \geq \varphi_i(b_1)$,}\\
	f_i(b_1) \otimes b_2  & \text{otherwise,}
	\end{cases}
\end{equation*}
and similarly for $e_i(b)$, where $\varepsilon_j$ and $\varphi_j$ 
are defined as in \eqref{eq:phi eps}.

%
\section{Quantum Lakshmibai-Seshadri paths}
\label{sec:qLS}

In this section, we introduce quantum Lakshmibai-Seshadri paths, 
which are defined in terms of the parabolic quantum Bruhat graph.
The main result of this section is Theorem~\ref{thm:LS=QLS}, 
which shows that projected level-zero LS paths are QLS paths.
Although this result is proved in \cite{LNSSS2}, 
the proof given there is somewhat roundabout, and heavily 
depends on the connectedness of the affine crystal $\BB(\lambda)_{\cl}$, 
which itself is a deep result; in contrast, the proof given here is much more direct, 
and we need not use root operators.

%
\subsection{The parabolic quantum Bruhat graph}
\label{subsec:PQBG}
The quantum Bruhat graph was first introduced in a paper by Brenti, Fomin and Postnikov~\cite{BFP}
motivated by work of Fomin, Gelfand and Postnikov~\cite{FGP} in type $A$.
It later appeared in connection with the quantum cohomology of flag varieties in a paper by Fulton 
and Woodward~\cite{FW}. 

Let $J$ be a subset of $I$. 
We denote by $\QB(W^J)$ the \emph{parabolic quantum Bruhat graph}. 
Its vertex set is $W^J$. There are two kinds of directed edges.
Both are labeled by some $\alpha\in \Phi^+\setminus \Phi_J^+$.
For $w\in W^J$, there is a directed edge 
$w \overset{\alpha}{\longrightarrow} \mcr{wr_\alpha}$ 
(recall that $\mcr{wr_\alpha}$ denotes 
the minimum-length coset representative in the coset $wr_\alpha W_J$)
if $\alpha \in \Phi^+ \setminus \Phi_J^+$ and 
one of the following holds:
\begin{enumerate}
\item (Bruhat edge) $w\lessdot wr_\alpha$ is a covering relation in Bruhat order, that is, 
$\ell(wr_\alpha)=\ell(w)+1$. (One may deduce that $wr_\alpha \in W^J$.)
\item (Quantum edge) 
\begin{align}\label{E:projqarrow}
\ell(\lfloor wr_\alpha \rfloor)=\ell(w)+1-\pair{\alpha^\vee}{2\rho-2\rho_J}.
\end{align}
\end{enumerate}
We define the \emph{weight} of an edge $w \overset{\alpha}\longrightarrow \mcr{wr_\alpha}$ in the 
parabolic quantum Bruhat graph to be
either $\alpha^\vee$ or $0$, depending on whether it is a quantum edge or not, respectively. 
Then the weight of a directed path $\bp$, denoted by $\wt(\bp) \in Q^{\vee+}$, is 
defined as the sum of the weights of its edges. 
%
%
\subsection{Definition of quantum Lakshmibai-Seshadri paths}
\label{subsec:qLS}

In this subsection, 
we fix a dominant integral weight $\lambda \in X$. Set 
%
%
\begin{equation} \label{eq:J}
J=\bigl\{i \in I \mid \pair{\alpha_{i}^{\vee}}{\lambda}=0\bigr\},
\end{equation}
so that $W_J$ is the stabilizer of $\lambda$.
Given a rational number $b$, we define $\bQB{b}$ to be the subgraph of 
the parabolic quantum Bruhat graph $\QB(W^{J})$ with 
the same vertex set but having only the edges:
\begin{equation}\label{bbord}
x \stackrel{\alpha}{\rightarrow} y \quad \text{with} \quad
\pair{\alpha^{\vee}}{b\lambda} = 
b \pair{\alpha^{\vee}}{\lambda}  \in \BZ;
\end{equation}
note that $\bQB{b} = \QB(W^{J})$ if $b \in \BZ$.
%
%
\begin{definition} \label{dfn:QLS}
A quantum Lakshmibai-Seshadri (QLS) path of shape $\lambda$ is 
a pair $\eta=(\ud{x}\,;\,\ud{b})$ of 
a sequence $\ud{x}\,:\,x_{1},\,x_{2},\,\dots,\,x_{s}$ of 
elements in $W^{J}$ with $x_{k} \ne x_{k+1}$ 
for $1 \le k \le s-1$ and a sequence 
$\ud{b}\,:\,
 0=b_{0} < b_{1} < \cdots < b_{s}=1$ 
of rational numbers satisfying the condition that 
there exists a directed path 
from $x_{k+1}$ to $x_{k}$ in $\bQB{b_{k}}$ 
for each $1 \le k \le s-1$.
\end{definition}

Denote by $\QLS(\lambda)$ the set of QLS paths of shape $\lambda$. 
We use the notation $x \barrl{b\lambda} y$ to indicate that there exists 
a directed path from $x$ to $y$ in $\QB_{b\lambda}(W^J)$, 
where $J$ is as in \eqref{eq:J}; 
so we can write an element
\[
	\eta=(x_{1},\,x_{2},\,\dots,\,x_{s}\,;\,
	b_{0},\,b_{1},\,\dots,\,b_{s}) 
\]
in $\QLS(\lambda)$ as follows: 
\begin{equation} \label{eq:eta2}
 x_1 \blarrl{b_1\lambda} 
 x_2 \blarrl{b_2\lambda} \cdots \blarrl{b_{s-2}\lambda} 
 x_{s-1} \blarrl{b_{s-1}\lambda} x_s\,.
\end{equation}
Since $W^J$ can be identified with $W\lambda$ under 
the canonical bijection $w\mapsto w\lambda$, we will 
sometimes think of the elements $x_i$ as weights. Moreover, we identify $\eta$  
with the following piecewise-linear, continuous map 
$\eta:[0,1] \rightarrow \BR \otimes_{\BZ} X$:
%
%
\begin{equation} \label{eq:eta}
\eta(t)=\sum_{k=1}^{l-1}
(b_{k}-b_{k-1})x_{k}\lambda + 
(t-b_{l-1})x_{l}\lambda \quad \text{for \ }
b_{l-1} \le t \le b_{l}, \  1 \le l \le s.
\end{equation}
%
%
\begin{remark} \label{rem:straight2}
It follows from the definition of a QLS path of shape $\lambda$ 
that $\eta_{x}:=(x\,;\,0,\,1) \in \QLS(\lambda)$ for every $x \in W^{J}$, 
with $J$ as in \eqref{eq:J},
which corresponds to the straight line $\eta_{x\lambda}(t)=tx\lambda$, $t \in [0,1]$.
We can easily see that $\cl(\pi_{\nu})=\eta_{\cl(\nu)}$ 
for $\nu \in X_{\af}^{0}(\lambda)$; 
recall that $\cl(X_{\af}^{0}(\lambda))=W\lambda \cong W^{J}$. 
\end{remark}

%
\subsection{Relation between LS paths and QLS paths}
\label{subsec:LSQLS}
%
%
We now establish the correspondence between projected level-zero LS paths and QLS paths.
As before, $\lambda\in X$ is a fixed dominant integral weight, 
and $J=\bigl\{i \in I \mid \pair{\alpha_{i}^{\vee}}{\lambda}=0\bigr\}$.
\begin{thm} \label{thm:LS=QLS}
$\BB(\lambda)_{\cl}=\QLS(\lambda)$ as sets of 
piecewise-linear, continuous maps from $[0,1]$ to $\BR \otimes_{\BZ} X$ 
{\rm(}see {\rm Section~\ref{subsec:LS}} and \eqref{eq:eta}{\rm)}.
\end{thm}

In order to prove this theorem, 
we need the following lemma. 
%
%
\begin{lem} \label{lem:sigma}
Let $b \in \BQ$. 

{\rm (1)} Let $\mu,\,\nu \in X_{\af}^{0}(\lambda)$. 
If there exists a $b$-chain for $(\nu,\,\mu)$, 
then there exists a directed path from $\cl(\mu)$ to $\cl(\nu)$ 
in $\bQB{b}$. 

{\rm (2)} Let $w,\,w' \in W^{J}$. 
If there exists a directed path from $w$ to $w'$ in $\bQB{b}$, 
then for each $\mu \in X_{\af}^{0}(\lambda)$ with $\cl(\mu)=w$, 
there exists a $b$-chain for $(\nu,\,\mu)$ 
for some $\nu \in X_{\af}^{0}(\lambda)$ with $\cl(\nu)=w'$.  
\end{lem}

\begin{proof}
(1) It suffices to show the assertion in the case that 
$\nu$ is a cover of $\mu$, i.e., $\mu \lessdot \nu$ 
in $X_{\af}^{0}(\lambda)$. Let $\beta \in \Phi^{\af+}$ be 
such that $r_{\beta}\mu=\nu$ and $b\pair{\beta^{\vee}}{\mu} 
\in \BZ$. Then, $\beta \in \Phi^{+}$ or $\beta \in \delta-\Phi^{+}$ 
(see \cite[Lemma 6.4\,(1)]{LNSSS}). 
Set $w:=\cl(\mu) \in W\lambda \cong W^{J}$. 
If $\beta \in \Phi^{+}$, then it follows from \cite[Theorem 6.5]{LNSSS} 
that $\gamma:=w^{-1}\beta \in 
\Phi^{+} \setminus \Phi^{+}_{J}$ and 
$\cl(\mu)=w \stackrel{\gamma}{\rightarrow} \cl(\nu)$ in $\QB(W^{J})$. 
In addition, we see that 
$b \pair{\gamma^{\vee}}{\lambda} = 
b \pair{\beta^{\vee}}{\mu} \in \BZ$,
which implies that 
$\cl(\mu)=w \stackrel{\gamma}{\rightarrow} \cl(\nu)$
in $\bQB{b}$. 
Similarly, 
if $\beta \in \delta-\Phi^{+}$, then 
it follows from \cite[Theorem 6.5]{LNSSS} 
that $\gamma:=w^{-1}(\beta-\delta) \in 
\Phi^{+} \setminus \Phi^{+}_{J}$ and 
$\cl(\mu)=w \stackrel{\gamma}{\rightarrow} \cl(\nu)$ in $\QB(W^{J})$. 
We see that 
$b \pair{\gamma^{\vee}}{\lambda}=b \pair{\beta^{\vee}-c}{\mu} = 
b \pair{\beta^{\vee}}{\mu} \in \BZ$,
which implies that 
$\cl(\mu)=w \stackrel{\gamma}{\rightarrow} \cl(\nu)$
in $\bQB{b}$. Thus we have proved part (1). 

(2) Fix $\mu \in X_{\af}^{0}(\lambda)$ such that $\cl(\mu)=w$. 
Assume that 
\begin{equation*}
w=x_{0} \stackrel{\gamma_{1}}{\rightarrow} 
  x_{1} \stackrel{\gamma_{2}}{\rightarrow} \cdots 
  \stackrel{\gamma_{m}}{\rightarrow}
x_{m}=w'
\end{equation*}
is a directed path from $w$ to $w'$ in $\bQB{b}$. 
We show the assertion by induction on the length $m$ of 
the directed path above. Assume first that $m=1$; for simplicity of 
notation, we set $\gamma:=\gamma_{1}$. Set
\begin{equation*}
\beta:=
 \begin{cases}
 w\gamma & 
  \text{if $w \stackrel{\gamma}{\rightarrow} w'$ is a Bruhat edge}, \\[1.5mm]
 \delta+w\gamma & 
  \text{if $w \stackrel{\gamma}{\rightarrow} w'$ is a quantum edge}.
 \end{cases}
\end{equation*}
It follows from \cite[Theorem 6.5]{LNSSS} that 
$\beta \in \Phi^{\af+}$ and $\mu \lessdot r_{\beta}\mu=:\nu$. 
Also, we see that $\cl(\nu)=w'$. In addition, 
$b\pair{\beta^{\vee}}{\mu}=b\pair{\gamma^{\vee}}{\lambda} \in \BZ$. 
Thus, $\mu \lessdot \nu$ is a $b$-chain 
for $(\nu,\,\mu)$. Assume that $m \ge 2$. By our induction hypothesis,
there exists a $b$-chain for $(\mu',\,\mu)$ 
for some $\mu' \in X_{\af}^{0}(\lambda)$ with $\cl(\mu')=x_{m-1}$. 
Also, by our induction hypothesis,
there exists a $b$-chain for $(\nu,\,\mu')$ 
for some $\nu \in X_{\af}^{0}(\lambda)$ with $\cl(\nu)=x_{m}=w'$. 
Concatenating these $b$-chains, we obtain a $b$-chain for 
$(\nu,\,\mu)$. Thus we have proved the lemma. 
\end{proof}

\begin{proof}[Proof of Theorem~{\rm \ref{thm:LS=QLS}}]
First, let us show that 
$\BB(\lambda)_{\cl} \subset \QLS(\lambda)$. Let 
%
%
\begin{equation} \label{eq:pi0}
\pi=(\nu_{1},\,\nu_{2},\,\dots,\,\nu_{s-1},\,\nu_{s}\,;\,
b_{0},\,b_{1},\,b_{2},\,\dots,\,
b_{s-1},\,b_{s}) \in \BB(\lambda).
\end{equation}
We show $\cl(\pi) \in \QLS(\lambda)$ by induction on $s$. 
If $s=1$, then the assertion is obvious by Remark~\ref{rem:straight2}. 
Assume that $s > 1$. 
Set 
%
%
\begin{equation} \label{eq:pip}
\pi':=(\nu_{2},\,\dots,\,\nu_{s-1},\,\nu_{s}\,;\,
b_{0},\,b_{2},\,\dots,\,b_{s-1},\,b_{s}).
\end{equation}
Then we see that $\pi' \in \BB(\lambda)$, and hence 
$\cl(\pi') \in \QLS(\lambda)$ by our induction hypothesis. 
Write $\cl(\pi')$ as: 
%
%
\begin{equation} \label{eq:clpip}
\cl(\pi'):=(y_{1},\,y_{2},\,\dots,\,y_{u}\,;\,
c_{0},\,c_{1},\,\dots,\,c_{u-1},\,c_{u})
\end{equation}
for some $y_{1},\,y_{2},\,\dots,\,y_{u} \in W^{J}$ and 
$0=c_{0} < c_{1} < \cdots < c_{u-1} < c_{u}=1$. 
Here we claim that $0 < b_{1} < b_{2} \le c_{1}$ and $y_{1}=\cl(\nu_{2})$; notice that 
the inequalities $0 < b_1 < b_2$ are obvious by the definition of LS paths. 
We show that $b_{2} \le c_{1}$ and $y_{1}=\cl(\nu_{2})$. 
By \eqref{eq:pi} and \eqref{eq:pip}, we have 
\begin{equation*}
\pi'(t) = 
 \begin{cases} 
 t\nu_2 & \text{for $t \in [0,b_2]$}, \\
 b_2 \nu_2 + (t-b_2) \nu_3 & \text{for $t \in [b_2,b_3]$}.
 \end{cases}
\end{equation*}
Hence we have
\begin{equation} \label{eq:clpi1}
\cl(\pi')(t) = 
 \begin{cases} 
 t \cl(\nu_2) & \text{for $t \in [0,b_2]$}, \\
 b_2 \cl(\nu_2) + (t-b_2) \cl(\nu_3) & \text{for $t \in [b_2,b_3]$};
 \end{cases}
\end{equation}
note that $\cl(\nu_{2})$  and $\cl(\nu_3)$ are direction vectors of 
$\cl(\pi')$ for the intervals $[0,\,b_{2}]$ and $[b_{2},\,b_{3}]$, 
respectively. Therefore, 
\begin{enumerate}
\item[(a)] 
 if $\cl(\nu_2) \ne \cl(\nu_3)$, then the first turning point 
 of $\cl(\pi')$ is equal to $b_2$,
\item[(b)] 
 if $\cl(\nu_2) = \cl(\nu_3)$, then the first turning point 
 of $\cl(\pi')$ is greater than $b_2$; 
\end{enumerate}
in particular, the first turning point of $\cl(\pi')$ 
is greater than or equal to $b_{2}$. Next, by \eqref{eq:clpip}, 
we have
\begin{equation} \label{eq:clpi2}
(\cl(\pi'))(t) = 
 \begin{cases}
 t (y_1\lambda) & \text{for $t \in [0,c_1]$}, \\
 c_1 (y_1\lambda) + (t-c_1) (y_2\lambda) & \text{for $t \in [c_1,c_2]$}. 
 \end{cases}
\end{equation}
Since $y_1 \in W^{J}$ is not equal to $y_2 \in W^{J}$ in \eqref{eq:clpip}, 
we have $y_{1}\lambda \ne y_{2}\lambda$, 
which implies that $c_{1}$ is nothing but the first turning point of $\cl(\pi')$. 
Because the first turning point of $\cl(\pi')$ 
is greater than or equal to $b_{2}$ as seen above, we obtain $b_{2} \le c_{2}$. 
Moreover, by comparing the first direction vector of $\cl(\pi')$ in \eqref{eq:clpi1} 
and that in \eqref{eq:clpi2}, we obtain $y_{1}=\cl(\nu_{2})$, as desired. 

If $\cl(\nu_{1})=\cl(\nu_{2})$, then it follows immediately that $\cl(\pi)=\cl(\pi')$, 
and hence $\cl(\pi) \in \QLS(\lambda)$. Assume that 
$\cl(\nu_{1}) \ne \cl(\nu_{2})=y_{1}$; 
set $x_{1}:=\cl(\nu_{1}) \in W\lambda \cong W^{J}$.
Because there exists a $b_{1}$-chain for 
$(\nu_{1},\,\nu_{2})$ by the definition of an LS path, 
we deduce from Lemma~\ref{lem:sigma}\,(1) 
that there exists a directed path from 
$y_{1}=\cl(\nu_{2})$ to $x_{1}=\cl(\nu_{1})$ 
in $\bQB{b_{1}}$.
Therefore, we see that 
\begin{equation*}
(x_{1},\,y_{1},\,y_{2},\,\dots,\,y_{u}\,;\,
c_{0},\,b_{1},\,c_{1},\,\dots,\,c_{u-1},\,c_{u}) 
\end{equation*}
is a QLS path of shape $\lambda$, which is identical to $\cl(\pi)$. 
Thus we obtain $\cl(\pi) \in \QLS(\lambda)$, 
as desired. 

Next, let us show the opposite inclusion, i.e., 
$\BB(\lambda)_{\cl} \supset \QLS(\lambda)$. Let 
\begin{equation*}
\eta=
(x_{1},\,x_{2},\,\dots,\,x_{s-1},\,x_{s}\,;\,
b_{0},\,b_{1},\,b_{2},\,\dots,\,
b_{s-1},\,b_{s}) \in \QLS(\lambda).
\end{equation*}
We show by induction on $s$ 
that there exists $\pi \in \BB(\lambda)$ such that 
$\cl(\pi)=\eta$. If $s=1$, then the assertion is obvious 
by Remark~\ref{rem:straight2}. Assume that $s > 1$. 
We see that 
\begin{equation*}
\eta':=
(x_{2},\,\dots,\,x_{s-1},\,x_{s}\,;\,
b_{0},\,b_{2},\,\dots,\,
b_{s-1},\,b_{s})
\end{equation*}
is contained in $\QLS(\lambda)$. Hence, 
by our induction hypothesis, there exists $\pi' \in \BB(\lambda)$ 
such that $\cl(\pi')=\eta'$. Write $\pi'$ as: 
\begin{equation*}
\pi'=(\mu_{1},\,\mu_{2},\,\dots,\,\mu_{u}\,;\,
c_{0},\,c_{1},\,\dots,\,c_{u-1},\,c_{u}) 
\end{equation*}
for some $\mu_{1},\,\mu_{2},\,\dots,\,\mu_{u} \in X_{\af}^{0}(\lambda)$ and 
$0=c_{0} < c_{1} < \cdots < c_{u-1} < c_{u}=1$;
we remark that $0 < b_{1} < b_{2} \le c_{1}$ and 
$\cl(\mu_{1})=x_{2}$. 
Because there exists a directed path from $x_{2}=\cl(\mu_{1})$ to 
$x_{1}$ in $\bQB{b_{1}}$, 
it follows from Lemma~\ref{lem:sigma}\,(2)
that there exists a $b_{1}$-chain for $(\nu_{1},\mu_{1})$ 
for some $\nu_{1} \in X_{\af}^{0}(\lambda)$ with $\cl(\nu_{1})=x_{1}$. 
Therefore, we have
\begin{equation*}
\pi:=(\nu_{1},\,\mu_{1},\,\mu_{2},\,\dots,\,\mu_{u}\,;\,
c_{0},\,b_{1},\,c_{1},\,\dots,\,c_{u-1},\,c_{u}) 
\in \BB(\lambda).
\end{equation*}
Also, it is easily seen that $\cl(\pi)=\eta$. Thus we have proved 
the opposite inclusion, thereby completing the proof of the theorem. 
\end{proof}

%
\section{Formula for the degree function}
\label{sec:degree}
Throughout this section, 
we fix a dominant integral weight $\lambda \in X$, and set
$J:=\bigl\{i \in I \mid \pair{\alpha_{i}^{\vee}}{\lambda}=0\bigr\}$. 
We define the degree function on projected level-zero LS paths 
in Section~\ref{subsec:dfn-deg} and recall the relation with the energy
function on KR crystals in Theorem~\ref{thm:ns-deg}. Theorem~\ref{thm:deg-QLS}
is the main result of this section and provides an explicit expression for the degree function
as sums of weights of shortest paths in the parabolic quantum Bruhat graph.

%
\subsection{Weights of directed paths}
\label{subsec:wtdp}
We know the following proposition from \cite[Proposition~8.1]{LNSSS}.
%
%
\begin{prop} \label{prop:weight}
Let $x,\,y \in W^{J}$, where 
$J=\bigl\{ i \in I \mid \pair{\alpha_{i}^{\vee}}{\lambda}=0 \bigr\}$.
Let $\bp$ and $\bq$ be a shortest 
and an arbitrary directed path from $x$ to $y$ in $\QB(W^{J})$, 
respectively. Then there exists $h \in Q^{\vee+}$ such that 
\begin{equation*}
\wt (\bq) - \wt (\bp) \equiv h \mod Q_{J}^{\vee}.
\end{equation*}
In addition, if $\bq$ is also shortest, 
then $\wt (\bq) \equiv \wt (\bp) \mod Q_{J}^{\vee}$. 
\end{prop}

Let $x,\,y \in W^{J}$, where 
$J=\bigl\{ i \in I \mid \pair{\alpha_{i}^{\vee}}{\lambda}=0 \bigr\}$.
By Proposition~\ref{prop:weight}, 
the value $\pair{\wt(\bp)}{\lambda}$ for
a shortest directed path $\bp$ from $x$ to $y$ in $\QB(W^{J})$
does not depend on the choice of 
such a shortest directed path;
we denote this value by $\wts{x}{y}$. 

The following is a corollary to \cite[Lemma 7.7]{LNSSS}; 
for $\ti{\alpha}_{j}$ and $s_{j}$, see \eqref{tialpha}.
%
%
\begin{cor} \label{cor:weight}
Let $w_{1},\,w_{2} \in W^{J}$, and $j \in I_{\af}$. 

{\rm (1)} 
If $\pair{\ti{\alpha}_{j}^{\vee}}{w_{1}\lambda} > 0$ and 
$\pair{\ti{\alpha}_{j}^{\vee}}{w_{2}\lambda} \le 0$, then 
\begin{equation*}
\wts{\mcr{s_{j}w_{1}}}{w_{2}} =
\wts{w_{1}}{w_{2}}-
\delta_{j,\,0}\pair{\ti{\alpha}_{j}^{\vee}}{w_{1}\lambda}. 
\end{equation*}

{\rm (2)} 
If $\pair{\ti{\alpha}_{j}^{\vee}}{w_{1}\lambda} < 0$ and 
$\pair{\ti{\alpha}_{j}^{\vee}}{w_{2}\lambda} < 0$, then 
\begin{equation*}
\wts{ \mcr{s_{j}w_{1}} }{ \mcr{s_{j}w_{2}} } = 
\wts{w_{1}}{w_{2}}-
\delta_{j,\,0}\pair{\ti{\alpha}_{j}^{\vee}}{w_{1}\lambda}+ 
\delta_{j,\,0}\pair{\ti{\alpha}_{j}^{\vee}}{w_{2}\lambda}. 
\end{equation*}

{\rm (3)} 
If $\pair{\ti{\alpha}_{j}^{\vee}}{w_{1}\lambda} \ge 0$ and 
$\pair{\ti{\alpha}_{j}^{\vee}}{w_{2}\lambda} < 0$, then 
\begin{equation*}
\wts{w_{1}}{\mcr{s_{j}w_{2}}} =
\wts{w_{1}}{w_{2}}+
\delta_{j,\,0}\pair{\ti{\alpha}_{j}^{\vee}}{w_{2}\lambda}. 
\end{equation*}

\end{cor}

\begin{proof}
We give a proof only for part (1); 
the proofs for parts (2) and (3) are similar. 
Let $\bp$ be a shortest directed path from $w_{1}$ to $w_{2}$. 
Then it follows from \cite[Lemma 7.7\,(3) and (5)]{LNSSS} that
there exists a shortest directed path $\bp'$ 
from $\mcr{s_{j}x}$ to $w_2$ such that
\begin{equation*}
\wt(\bp')=\wt(\bp)-\delta_{j,\,0}w_{1}^{-1}\ti{\alpha}_{j}^{\vee}.
\end{equation*}
Hence, 
\begin{align*}
\wts{\mcr{s_{j}w_{1}}}{w_{2}} & = 
\pair{\wt(\bp')}{\lambda}=
\pair{\wt(\bp)}{\lambda}-\delta_{j,\,0}
\pair{w_{1}^{-1}\ti{\alpha}_{j}^{\vee}}{\lambda} \\
& = \wts{w_{1}}{w_{2}}-
\delta_{j,\,0}\pair{\ti{\alpha}_{j}^{\vee}}{w_{1}\lambda}.
\end{align*}
Thus we have proved the corollary. 
\end{proof}

%
\subsection{Definition of the degree function}
\label{subsec:dfn-deg}
Let us recall from \cite[Section 3.1]{NS08} the definition of 
the degree function
\begin{equation*}
\Deg=\Deg_{\lambda} : 
   \BB(\lambda)_{\cl} \rightarrow \BZ_{\le 0}.
\end{equation*}
Denote by $\BB_{0}(\lambda)$ the connected component of 
$\BB(\lambda)$ containing the straight line $\pi_{\lambda}=(\lambda\,;\,0,\,1)$.
Also, for $\pi=(\nu_{1},\,\dots,\,\nu_{s}\,;\,b_{0},\,\dots,\,b_{s}) 
\in \BB(\lambda)$, we set $\iota(\pi):=\nu_{1}$, and call it 
the initial direction of $\pi$; 
note that $\iota(\pi)=\pi(\varepsilon)/\varepsilon$ for 
sufficiently small $\varepsilon > 0$. 
We know from \cite[Proposition 3.1.3]{NS08} that 
for each $\eta \in \BB(\lambda)_{\cl}$, 
there exists a unique $\pi_{\eta} \in \BB_{0}(\lambda)$ 
satisfying the conditions that $\cl(\pi_{\eta})=\eta$ and 
$\iota(\pi_{\eta}) \in \lambda-Q^{+}$; recall that 
$Q^{+}=\sum_{j \in I} \BZ_{\ge 0}\alpha_{j}$.
Then it follows from \cite[Lemma 3.1.1]{NS08} that 
$\pi_{\eta}(1) \in X_{\af}^{0}$ is of the form: 
\begin{equation*}
\pi_{\eta}(1)=\lambda-\beta+K \delta
\end{equation*}
for some $\beta \in Q^{+}$ and $K \in \BZ_{\ge 0}$. 
We define the degree $\Deg(\eta) \in \BZ_{\le 0}$ of 
$\eta \in \BB(\lambda)_{\cl}$ by:
\begin{equation}
\Deg(\eta)=-K \in \BZ_{\le 0}.
\end{equation}

\begin{ex} \label{ex:deg}
Assume that $\Fg$ is of type $A_{2}^{(1)}$, and let $\lambda=\vpi_{1}+\vpi_{2}$. 
Note that $J=\bigl\{i \in I \mid \pair{\alpha_{i}^{\vee}}{\lambda}=0\bigr\}$ is 
the empty set, and hence $W^{J} = W$; for the quantum Bruhat graph 
$\QB(W)$,  see \cite[Fig.\,1]{LNSSS}. 
\begin{enumerate}
\item 
Set $\eta_{1}:=(e,\,w_{\circ} \,;\, 0,\,1/2,\,1)$. 
Since $w_{\circ}=r_{\theta} \edge{\theta} e$ is 
a quantum edge in $\QB(W)$, 
which is also an edge in $\QB_{(1/2)\lambda}(W)$, 
we see that $\eta_{1} \in \QLS(\lambda)=\BB(\lambda)_{\cl}$. 
We claim that $\pi_{\eta_{1}}=(\lambda,\,r_{0}\lambda\,;\,0,\,1/2,\,1)$. 
First we see that $e_{0}\pi_{\lambda}= (\lambda,\,r_{0}\lambda\,;\,0,\,1/2,\,1)$, 
which implies that $(\lambda,\,r_{0}\lambda\,;\,0,\,1/2,\,1)$ 
is contained in $\BB_{0}(\lambda)$. 
Since $\cl(\lambda)=e$ and $\cl(r_{0}\lambda)=w_{\circ}$, 
the image of $(\lambda,\,r_{0}\lambda\,;\,0,\,1/2,\,1)$ 
under the map $\cl$ is identical to $\eta$. 
Also, the initial direction of this element is equal to 
$\lambda \in \lambda-Q^{+}$. Therefore, we deduce that 
$\pi_{\eta_{1}}=(\lambda,\,r_{0}\lambda\,;\,0,\,1/2,\,1)$, as desired. 
Since $\pi_{\eta_{1}}(1)=\lambda+\alpha_{0} = \lambda-\theta+\delta$, 
we have $\Deg(\eta_{1})=-1$. 

\item Set $\eta_{2}:=(r_{1}r_{2},\,r_{2}\,;\,0,\,1/2,\,1)$. 
Since $r_{2} \edge{\theta} r_{1}r_{2}$ is a Bruhat edge in $\QB(W)$, 
which is also an edge in $\QB_{(1/2)\lambda}(W)$, we see that 
$\eta_{2} \in \QLS(\lambda)=\BB(\lambda)_{\cl}$. As above, 
we deduce that $\pi_{\eta_{2}}=(r_{1}r_{2},\,r_{2}\,;\,0,\,1/2,\,1)$. 
Since $\pi_{\eta_{2}}(1)=\lambda-(\alpha_{1}+\alpha_{2})$, 
we have $\Deg(\eta_{2})=0$. 
\end{enumerate}
\end{ex}

%
\begin{remark} \label{rem:degree}
It is known (see, e.g., \cite[Proposition 4.3.1]{NS08}) 
that for each $\eta \in \BB(\lambda)_{\cl}$, 
there exist $j_{1},\,j_{2},\,\dots,\,j_{k} \in I_{\af}$ 
such that $e_{j_{k}}e_{j_{2}} \cdots e_{j_{1}}\eta=\eta_{\lambda}$; 
recall from Remark~\ref{rem:straight2} that $\eta_{e}=\eta_{\lambda}=\cl(\pi_{\lambda})$.
Therefore, we deduce from \cite[Lemma 3.2.1]{NS08} that
$\Deg=\Deg_{\lambda} : \BB(\lambda)_{\cl} \rightarrow \BZ_{\le 0}$
is a unique function satisfying the following conditions: 

(i) $\Deg(\eta_{\lambda})=0$; 

(ii) for $\eta \in \BB(\lambda)_{\cl}$ and $j \in I_{\af}$
with $e_{j}\eta \ne \bzero$, 
%
%
\begin{equation} \label{eq:deg-ro}
\Deg(e_{j}\eta)=
 \begin{cases}
 \Deg(\eta)-1 
  & \text{\rm if $j=0$ and 
    $\iota(e_{0}\eta)=\iota(\eta)$}, \\[1.5mm]
 \Deg(\eta)-\pair{\ti{\alpha}_{0}^{\vee}}{\iota(\eta)}-1 
  & \text{\rm if $j=0$ and 
    $\iota(e_{0}\eta)=s_{0}(\iota(\eta))$}, \\[1.5mm]
 \Deg(\eta) & \text{\rm if $j \ne 0$},
 \end{cases}
\end{equation}
where $\iota(\eta):=\eta(\varepsilon)/\varepsilon$ for 
sufficiently small $\varepsilon > 0$. 
\end{remark}

%
\subsection{Relation between the degree function and the energy function}
\label{subsec:deg-ene}

Write $\lambda$ as $\lambda=
\vpi_{i_{1}} + \vpi_{i_{2}} + \cdots + \vpi_{i_{p}}$, 
with $i_{1},\,i_{2},\,\dots,\,i_{p} \in I$.
By Theorem~\ref{thm:LScl}\,(3), 
there exists an isomorphism 
\begin{equation}\label{defpsi}
\Psi:\BB(\lambda)_{\cl} \stackrel{\sim}{\rightarrow} 
 \BB(\vpi_{i_{1}})_{\cl} \otimes 
 \BB(\vpi_{i_{2}})_{\cl} \otimes \cdots \otimes 
 \BB(\vpi_{i_{p}})_{\cl}=:\BB
\end{equation}
of crystals. Here we should recall from Remark~\ref{rem:LScl} 
that $\BB(\vpi_{i})_{\cl}$ is isomorphic to 
the one-column KR crystal $B^{i,1}$. Also, recall from
Section~\ref{subsec:crystal} that we are using the Kashiwara convention for
tensor products in this paper.
So, following \cite[Section 3]{HKOTY} and \cite[Section 3.3]{HKOTT} 
(see also~\cite{SS} and~\cite[Section 4.1]{NS08}), 
we define the energy function $D=D_{\BB}:\BB \rightarrow \BZ_{\le 0}$ 
on $\BB$ as follows. First, for each $1 \le k,\,l \le p$, 
there exists a unique isomorphism 
(called a combinatorial $R$-matrix)
\begin{equation*}
R_{k,l}: 
 \BB(\vpi_{i_{k}})_{\cl} \otimes \BB(\vpi_{i_{l}})_{\cl}
 \stackrel{\sim}{\rightarrow}
 \BB(\vpi_{i_{l}})_{\cl} \otimes \BB(\vpi_{i_{k}})_{\cl}
\end{equation*}
of crystals. Also, there exists a unique $\BZ$-valued function 
(called a local energy function) 
$H_{k,l}:
\BB(\vpi_{i_{k}})_{\cl} \otimes \BB(\vpi_{i_{l}})_{\cl} \rightarrow \BZ$ 
satisfying the following conditions (H1) and (H2): 

(H1) 
For $\eta_{k} \otimes \eta_{l} \in 
\BB(\vpi_{i_{k}})_{\cl} \otimes \BB(\vpi_{i_{l}})_{\cl}$ and 
$j \in I_{\af}$ such that $e_{j}(\eta_{k} \otimes \eta_{l}) \ne \bzero$, 
\begin{align*}
& H_{k,l}(e_{j}(\eta_{k} \otimes \eta_{l})) = 
  \nonumber \\[3mm]
& \hspace*{10mm}
\begin{cases}
H_{k,l}(\eta_{k} \otimes \eta_{l})+1 & \\[1mm]
\hspace*{5mm} \text{\rm if $j=0$, and if
   $e_{0}(\eta_{k} \otimes \eta_{l})=
    e_{0}\eta_{k} \otimes \eta_{l}$,
   $e_{0}(\ti{\eta}_{l} \otimes \ti{\eta}_{k})=
    e_{0}\ti{\eta}_{l} \otimes \ti{\eta}_{k}$,} & \\[3mm]
H_{k,l}(\eta_{k} \otimes \eta_{l})-1 & \\[1mm]
\hspace*{5mm} \text{\rm if $j=0$, and if
   $e_{0}(\eta_{k} \otimes \eta_{l})=
    \eta_{k} \otimes e_{0}\eta_{l}$, 
   $e_{0}(\ti{\eta}_{l} \otimes \ti{\eta}_{k})=
    \ti{\eta}_{l} \otimes e_{0}\ti{\eta}_{k}$,} & \\[3mm]
H_{k,l}(\eta_{k} \otimes \eta_{l}) 
\hspace{10mm} \text{\rm otherwise}, & 
\end{cases}
\end{align*}
where we set 
$\ti{\eta}_{l} \otimes \ti{\eta}_{k}:=
 R_{k,l}(\eta_{k} \otimes \eta_{l}) \in 
 \BB(\vpi_{i_{l}})_{\cl} \otimes \BB(\vpi_{i_{k}})_{\cl}$. 

(H2) $H_{k,l}(\eta_{\vpi_{i_k}} \otimes \eta_{\vpi_{i_l}})=0$.

\noindent 
Now, for each $1 \le k < l \le p$, there exists a unique isomorphism
\begin{align*}
&
\BB(\vpi_{i_{k}})_{\cl} \otimes 
\BB(\vpi_{i_{k+1}})_{\cl} \otimes \cdots \otimes 
\BB(\vpi_{i_{l-1}})_{\cl} \otimes \BB(\vpi_{i_{l}})_{\cl} \\
& \hspace*{20mm} \stackrel{\sim}{\rightarrow} 
\BB(\vpi_{i_{l}})_{\cl} \otimes 
\BB(\vpi_{i_{k}})_{\cl} \otimes \cdots \otimes 
\BB(\vpi_{i_{l-2}})_{\cl} \otimes \BB(\vpi_{i_{l-1}})_{\cl}
\end{align*}
of crystals, which is given by composition of 
combinatorial $R$-matrices. 
Given $\eta_{k} \otimes \eta_{k+1} \otimes 
\cdots \otimes \eta_{l} \in 
\BB(\vpi_{i_{k}})_{\cl} \otimes 
\BB(\vpi_{i_{k+1}})_{\cl} \otimes \cdots \otimes 
\BB(\vpi_{i_{l}})_{\cl}$, 
we define $\eta_{l}^{(k)} \in \BB(\vpi_{i_{l}})_{\cl}$ 
to be the first factor of the image of 
$\eta_{k} \otimes \eta_{k+1} \otimes 
\cdots \otimes \eta_{l}$ under the above isomorphism of crystals.
For convenience, we set $\eta_{l}^{(l)}:=\eta_{l}$ 
for $\eta_{l} \in \BB(\vpi_{i_{l}})_{\cl}$, $1 \le l \le p$. 
In addition, for each $1 \le k \le p$, 
take (and fix) an arbitrary element 
$\eta_{k}^{\flat} \in \BB(\vpi_{i_{k}})_{\cl}$ 
such that $f_{j}\eta_{k}^{\flat}=\bzero$ 
for all $j \in I$.
Then we define the energy function $D=D_{\BB}: 
 \BB=
 \BB(\vpi_{i_{1}})_{\cl} \otimes 
 \BB(\vpi_{i_{2}})_{\cl} \otimes \cdots \otimes 
 \BB(\vpi_{i_{p}})_{\cl} \rightarrow \BZ$ by: 
%
%
\begin{align}
& 
D(\eta_{1} \otimes 
  \eta_{2} \otimes \cdots \otimes 
  \eta_{p}) = \nonumber \\
& \hspace*{25mm}
\sum_{1 \le k < l \le p} 
 H_{k,l}
 (\eta_{k} \otimes \eta_{l}^{(k+1)})
 + \sum_{k=1}^{p} 
 H_{k,k}
  (\eta^{\flat}_{k} \otimes \eta_{k}^{(1)}). \label{eq:Di}
\end{align}

We know the following theorem from \cite[Theorem 4.1.1]{NS08}. 
%
%
\begin{thm} \label{thm:ns-deg}
Using the same notation as above, we have 
%
%
\begin{equation} \label{eq:main}
\Deg(\eta)=D(\Psi(\eta))-D^{\ext} \qquad
\text{\rm for every $\eta \in \BB(\lambda)_{\cl}$}, 
\end{equation}
where $D^{\ext}=D_{\BB}^{\ext} \in \BZ$ is a constant defined by
\begin{equation*}
D^{\ext}=D_{\BB}^{\ext} : = 
\sum_{k=1}^{p} 
 H_{k,k}
  (\eta^{\flat}_{k} \otimes \eta_{\vpi_{i_{k}}}); 
\end{equation*}
here, $\Psi:\BB(\lambda)_{\cl} \stackrel{\sim}{\rightarrow} \BB$ is 
the isomorphism of crystals given in \eqref{defpsi}.
\end{thm}

%
\subsection{Formula for the degree function}
\label{subsec:degree}
Let $\eta \in \BB(\lambda)_{\cl}$. 
Because $\BB(\lambda)_{\cl}=\QLS(\lambda)$
by Theorem~\ref{thm:LS=QLS}, 
we can write $\eta$ as:
%
%
\begin{equation} \label{eq:eta-QLS}
\eta=
(x_{1},\,x_{2},\,\dots,\,x_{s}\,;\,
b_{0},\,b_{1},\,\dots,\,b_{s}) \in \QLS(\lambda)
\end{equation}
for some $x_{1},\,x_{2},\,\dots,\,x_{s} \in W^{J}$ and 
$0=b_{0} < b_{1} < \dots < b_{s}=1$; 
note that $\iota(\eta)=x_{1}\lambda$. 
%
%
\begin{thm} \label{thm:deg-QLS}
With the same notation as above, 
we have
%
%
\begin{equation} \label{eq:deg}
\Deg(\eta)=-\sum_{k=1}^{s-1} 
 (1-b_{k})\wts{x_{k+1}}{x_{k}}. 
\end{equation}
\end{thm}

\begin{proof}
For $\eta \in \QLS(\lambda) = \BB(\lambda)_{\cl}$, 
we define $F(\eta)$ to be the right-hand side of \eqref{eq:deg}. 
It suffices to show that 
$F$ satisfies conditions 
(i) and (ii) in Remark~\ref{rem:degree}, i.e., 

(i) $F(\eta_{\lambda})=0$; 

(ii) for $\eta \in \BB(\lambda)_{\cl}$ and $j \in I_{\af}$ 
with $e_{j}\eta \ne \bzero$, 
%
%
\begin{equation} \label{eq:deg-ro2}
F(e_{j}\eta)=
 \begin{cases}
 F(\eta)-1 
  & \text{\rm if $j=0$ and 
    $\iota(e_{0}\eta)=\iota(\eta)$}, \\[1.5mm]
 F(\eta)-\pair{\ti{\alpha}_{0}^{\vee}}{\iota(\eta)}-1 
  & \text{\rm if $j=0$ and 
    $\iota(e_{0}\eta)=s_{0}(\iota(\eta))$}, \\[1.5mm]
 F(\eta) & \text{\rm if $j \ne 0$}.
 \end{cases}
\end{equation}
It is obvious that $F$ satisfies condition (i). 
Let us show that $F$ satisfies condition (ii). 
Let $\eta \in \BB(\lambda)_{\cl}$ and $j \in I_{\af}$
be such that $e_{j}\eta \ne \bzero$. 
We see that the point 
$t_{1}=\min \bigl\{t \in [0,1] \mid 
       H^{\eta}_{j}(t)=m^{\eta}_{j}\bigr\}$
is equal to $b_{p}$ for some $0 < p \le s$. 
Let $0 < q \le p$ be such that 
$b_{q-1} \le t_{0} < b_{q}$; recall that 
$t_{0}=\max\bigl\{t \in [0,t_1] \mid
       H^{\eta}_{j}(t)=m^{\eta}_{j}+1\bigr\}$. 
It follows from the definition of the root operator $e_{j}$ 
that $e_{j}\eta \in \QLS(\lambda)$ can be written as follows: 
\begin{equation*}
\begin{split}
& x_1 \blarrl{b_1\lambda} \cdots \blarrl{b_{q-2}\lambda} x_{q-1} \blarrl{b_{q-1}\lambda}
\underbrace{ {x_{q}} \blarrl{t_0\lambda} }_{\text{(a)}} 
 {\mcr{s_jx_{q}}} \blarrl{b_q\lambda} {\mcr{s_jx_{q+1}}} \blarrl{b_{q+1}\lambda} \ldots \\
& \hspace*{50mm} \cdots \blarrl{b_{p-1}\lambda}
 \underbrace{{\mcr{s_jx_{p}}} \blarrl{b_p\lambda=t_1\lambda} {x_{p+1}}}_{\text{(b)}} 
 \blarrl{b_{p+1}\lambda} \cdots \blarrl{b_{s-1}\lambda} {x_{s}} \,.
\end{split}
\end{equation*}
Here, if $b_{q-1}=t_{0}$, then we drop (a) from the above; 
note that in this case, $x_{q-1} \ne \mcr{s_{j}x_{q}}$ 
since $\pair{\ti{\alpha}_{j}^{\vee}}{x_{q-1}\lambda} \le 0$ and 
$\pair{\ti{\alpha}_{j}^{\vee}}{s_{j}x_{q}\lambda}=
-\pair{\ti{\alpha}_{j}^{\vee}}{x_{q}\lambda} > 0$ by 
Remark~\ref{rem:ro}\,(1), (3). 
Also, if $\mcr{s_{j}x_{p}} = x_{p+1}$, then 
we replace (b) by  ${ x_{p+1} }$
(or ${ \mcr{s_{j}x_{p}} }$) in the above path. 
We see from the definition \eqref{defeqls} of the root operator $e_{j}$ 
that $\iota(e_{j}\eta)=s_{j}\iota(\eta)$ if and only if $t_{0}=b_{0}=0$; 
in this case, $m_j^\eta=-1$ (see \eqref{eq:H m}) 
since $H^{\eta}_{j}(t_{0})=H^{\eta}_{j}(0)=0$.

Now, by the definition of $F$, we have 
%
%
\begin{align} \label{eq:Fe}
\begin{split}
F(e_{j}\eta) & =-\Biggl\{
\underbrace{%
\sum_{k=1}^{q-2} 
 (1-b_{k})\wts{x_{k+1}}{x_{k}} + R}_{=:U_1} + 
\underbrace{%
 \sum_{k=q}^{p-1} 
 (1-b_{k})\wts{ \mcr{s_{j}x_{k+1}} }{ \mcr{s_{j}x_{k}} }
}_{=:U_2} \\[1.5mm]
& \qquad +
\underbrace{%
(1-b_{p})\wts{ x_{p+1} }{ \mcr{s_{j}x_{p}} }
}_{=:U_3} +
\sum_{k=p+1}^{s-1}
 (1-b_{k}) \wts{x_{k+1}}{x_{k}}
\Biggr\}, 
\end{split}
\end{align}
where
%
%
\begin{equation} \label{eq:R}
R:=
\begin{cases}
(1-b_{q-1})\wts{x_{q}}{x_{q-1}} & \\[1mm]
\hspace*{15mm}
 + (1-t_0)\wts{\mcr{s_{j}x_{q}}}{x_{q}} & 
 \text{if $t_{0} \ne b_{q-1}$}, \\[3mm]
 (1-b_{q-1})\wts{\mcr{s_{j}x_{q}}}{x_{q-1}} & 
 \text{if $q > 1$ and $t_{0} = b_{q-1}$}, \\[3mm]
0 &  \text{if $q=1$ and $t_{0} = b_{0} = 0$};
\end{cases}
\end{equation}
if $p=s$ (resp., $q=1$), then
$\wts{ x_{p+1} }{ \mcr{s_{j}x_{p}} }$ in $U_{3}$ 
(resp., $\wts{x_{q}}{x_{q-1}}$ in $R$) 
is understood to be $0$; 
notice that the equality \eqref{eq:Fe} is valid 
even when $\mcr{s_{j}x_{p}} = x_{p+1}$. 
Also, observe that in \eqref{eq:R}, 
%
%
\begin{equation} \label{eq:Xt0}
\wts{\mcr{s_{j}x_{q}}}{x_{q}}=
-\delta_{j,\,0}
\pair{\ti{\alpha}_{j}^{\vee}}{x_{q}\lambda}.
\end{equation}

Note that if $q > 1$ and $t_{0} = b_{q-1}$ 
(in the second case of \eqref{eq:R}), then
%
%
\begin{equation} \label{eq:X3}
R=(1-b_{q-1})\wts{x_{q}}{x_{q-1}}
 -(1-t_0)\delta_{j,\,0}
 \pair{\ti{\alpha}_{j}^{\vee}}{x_{q}\lambda}.
\end{equation}
Indeed, note that 
$\pair{\ti{\alpha}_{j}^{\vee}}{x_{q-1}\lambda} \le 0$ and 
$\pair{\ti{\alpha}_{j}^{\vee}}{s_{j}x_{q}\lambda}=
-\pair{\ti{\alpha}_{j}^{\vee}}{x_{q}\lambda} > 0$ 
by Remark~\ref{rem:ro}\,(1), (3). 
Therefore, applying Corollary~\ref{cor:weight}\,(1) to 
$w_{1}=\mcr{s_{j}x_{q}}$ and $w_{2}=x_{q-1}$, we see that
\begin{align*}
R & = (1-b_{q-1})\wts{\mcr{s_{j}x_{q}}}{x_{q-1}} \\
 & = (1-b_{q-1})\wts{x_{q}}{x_{q-1}}
 - (1-b_{q-1})\delta_{j,\,0}
 \pair{\ti{\alpha}_{j}^{\vee}}{x_{q}\lambda} \\
 & = (1-b_{q-1})\wts{x_{q}}{x_{q-1}}
 - (1-t_0)\delta_{j,\,0}
 \pair{\ti{\alpha}_{j}^{\vee}}{x_{q}\lambda},
\end{align*}
as desired. By combining \eqref{eq:Fe} and \eqref{eq:X3}, 
we obtain
%
%
\begin{equation} \label{eq:A1}
U_{1}=
\begin{cases}
{ \displaystyle \sum_{k=1}^{q-1} } 
 (1-b_{k})\wts{x_{k+1}}{x_{k}} -
(1-t_{0})\delta_{j,\,0}
 \pair{\ti{\alpha}_{j}^{\vee}}{x_{q}\lambda}
 & \text{if $t_{0} \ne 0$}, \\[5mm]
0 & \text{if $t_{0} = 0$}.
\end{cases}
\end{equation}

Recall that the function $H^{\eta}_{j}(t)$ 
is strictly decreasing on $[t_{0},\,t_{1}]$ (see Remark~\ref{rem:ro}\,(1)), 
which implies that $\pair{\ti{\alpha}_{j}^{\vee}}{x_{k}\lambda} < 0$ 
for all $q \le k \le p$. Hence, by Corollary~\ref{cor:weight}\,(2), 
we have
\begin{equation*}
\wts{ \mcr{s_{j}x_{k+1}} }{ \mcr{s_{j}x_{k}} } = 
\wts{x_{k+1}}{x_{k}}-
\delta_{j,\,0}\pair{\ti{\alpha}_{j}^{\vee}}{x_{k+1}\lambda}+ 
\delta_{j,\,0}\pair{\ti{\alpha}_{j}^{\vee}}{x_{k}\lambda}
\end{equation*}
for each $q \le k \le p-1$. From this, we see that 
%
%
\begin{align}
U_2 & = 
\sum_{k=q}^{p-1} (1-b_{k}) \wts{x_{k+1}}{x_{k}} - 
\delta_{j,\,0}
\underbrace{%
\sum_{k=q}^{p-1} 
 (1-b_{k}) \pair{\ti{\alpha}_{j}^{\vee}}{x_{k+1}\lambda}%
}_{%
= \sum_{k=q+1}^{p} 
 (1-b_{k-1}) \pair{\ti{\alpha}_{j}^{\vee}}{x_{k}\lambda}%
} + 
\delta_{j,0} 
\sum_{k=q}^{p-1} 
 (1-b_{k}) \pair{\ti{\alpha}_{j}^{\vee}}{x_{k}\lambda}
\nonumber \\[3mm]
& =
\sum_{k=q}^{p-1} (1-b_{k}) \wts{x_{k+1}}{x_{k}} 
+\delta_{j,\,0} 
 (1-b_{q})\pair{\ti{\alpha}_{j}^{\vee}}{x_{q}\lambda}
\nonumber \\
& \hspace*{20mm} -
\delta_{j,\,0}\sum_{k=q+1}^{p-1} 
 (b_{k}-b_{k-1}) 
 \pair{\ti{\alpha}_{j}^{\vee}}{x_{k}\lambda} - 
\delta_{j,\,0}(1-b_{p-1}) \pair{\ti{\alpha}_{j}^{\vee}}{x_{p}\lambda}. 
\label{eq:U2}
\end{align}

Finally, let us show that 
%
%
\begin{equation} \label{eq:U3}
U_{3}=
(1-b_{p})\wts{x_{p+1}}{x_{p}} +
\delta_{j,\,0} 
 (1-b_{p})\pair{\ti{\alpha}_{j}^{\vee}}{x_{p}\lambda},
\end{equation}
where if $p=s$, then $\wts{x_{p+1}}{x_{p}}$ is 
understood to be $0$. If $p=s$, then 
the equality obviously holds. Assume that $p < s$. 
Then, since $\pair{\ti{\alpha}_{j}^{\vee}}{x_{p}\lambda} < 0$ 
and $\pair{\ti{\alpha}_{j}^{\vee}}{x_{p+1}\lambda} \ge 0$ 
by Remark~\ref{rem:ro}\,(2), the equality \eqref{eq:U3} 
follows immediately from Corollary~\ref{cor:weight}\,(3)
(applied to $w_{1}=x_{p+1}$ and $w_{2}=x_{p}$).

Substituting \eqref{eq:A1}, \eqref{eq:U2}, \eqref{eq:U3} 
into \eqref{eq:Fe}, we conclude that
\begin{align*}
F(e_{j}\eta) & = 
\underbrace{%
-\sum_{k=1}^{ s-1 }
 (1-b_{k}) \wts{x_{k+1}}{x_{k}}}_{=F(\eta)} - T \\
& \quad +\delta_{j,\,0} \Biggl\{
\underbrace{%
(b_{q}-t_{0}) \pair{\ti{\alpha}_{j}^{\vee}}{x_{q}\lambda}+ 
\sum_{k=q+1}^{p} 
 (b_{k}-b_{k-1}) \pair{\ti{\alpha}_{j}^{\vee}}{x_{k}\lambda}%
}_{=:V}\Biggr\}, 
\end{align*}
where
\begin{equation*}
T : = 
\begin{cases}
0 & \text{if $t_{0} \ne 0$}, \\[1.5mm]
\delta_{j,\,0} \pair{\ti{\alpha}_{j}^{\vee}}{x_{1}\lambda}=
\delta_{j,\,0} \pair{\ti{\alpha}_{j}^{\vee}}{\iota(\eta)}
  & \text{if $t_{0} = 0$}.
\end{cases}
\end{equation*}
Here, observe that 
\begin{equation*}
V = H^{\eta}_{j}(b_{p})-H^{\eta}_{j}(t_{0}) = 
H^{\eta}_{j}(t_1)-H^{\eta}_{j}(t_{0})=
m^{\eta}_{j}-(m^{\eta}_{j}+1)=-1.
\end{equation*}
Thus we have shown that $F$ satisfies \eqref{eq:deg-ro2}, 
thereby completing the proof of Theorem~\ref{thm:deg-QLS}. 
\end{proof}
%
%
\subsection{Lusztig involution \texorpdfstring{$S$}{S} on 
  \texorpdfstring{$\BB(\lambda)_{\cl}=\QLS(\lambda)$}
    {B(lambda)cl=QLS(lambda)}}
\label{subsec:Lusztig}

The results in this subsection will be used in Section~\ref{section.bijection} to 
define the bijection between the quantum alcove model $\CA(\lambda)$ 
and the set $\QLS(\lambda)$ (not $\QLS(-w_{\circ}\lambda)$). 
Let $w_{\circ} \in W$ be the longest element in $W$, and 
let $\omega:I \rightarrow I$ be the Dynkin diagram automorphism for $\Fg$ 
induced by $w_{\circ}$, i.e., $w_{\circ}\alpha_{j}=-\alpha_{\omega(j)}$ 
for $j \in I$. Note that $\omega$ acts as $-w_{\circ}$ 
on the integral weight lattice $X$ and also on the Cartan subalgebra $\Fh$ of $\Fg$. 
Then, $\omega(\theta)=\theta$, and $\omega r_{j} = r_{\omega(j)} \omega$ 
on $X$, and on $\Fh$ for $j \in I$. There exists a group automorphism, 
denoted also by $\omega$, of the Weyl group $W$ such that 
$\omega(r_{j})=r_{\omega(j)}$ for all $j \in I$; 
notice that $\ell(\omega(v))=\ell(v)$ for $v \in W$, and 
$\omega(r_{\alpha})=r_{\omega(\alpha)}$ for $\alpha \in \Phi^{+}$. 
%
%
\begin{lem} \label{lem:da} \mbox{}
\begin{enumerate}
\item If $v \in W^{J}$, then $\omega(v) \in W^{\omega(J)}$.
\item
Let $b$ be a rational number. For $x_{1},\,x_{2} \in W^{J}$ and 
$\beta \in \Phi^{+} \setminus \Phi^{+}_{J}$, 
\begin{align*}
& x_{1} \xrightarrow{\beta} x_{2} 
  \quad \text{\rm in $\bQB{b}$} 
  \quad \iff \quad 
\mcr{w_{\circ}x_{1}}^{J} \xleftarrow{u_{\circ}\beta} \mcr{w_{\circ}x_{2}}^{J}
  \quad \text{\rm in $\bQB{b}$} \\
& \hspace*{20mm} \iff
\omega(x_{1}) \xrightarrow{\omega(\beta)} \omega(x_{2})
\quad \text{\rm in $\QB_{b\omega(\lambda)}(W^{\omega(J)})$} \\
& \hspace*{20mm} \iff
\mcr{x_1w_{\circ}}^{\omega(J)} 
\xleftarrow{v_{\circ}\omega(\beta)} \mcr{x_2w_{\circ}}^{\omega(J)}
\quad \text{\rm in $\QB_{b\omega(\lambda)}(W^{\omega(J)})$},
\end{align*}
where $u_{\circ}$ and $v_{\circ}$ are the longest elements in $W_{J}$ and 
$W_{\omega(J)}$, respectively. In addition, the types (i.e., Bruhat or quantum) of 
these four edges coincide. 
\item For every $x_{1},\,x_{2} \in W^{J}$, 
\begin{equation*}
\begin{split}
\wts{x_{1}}{x_{2}} & = \wt_{\lambda}\bigl( \mcr{w_{\circ}x_{2}}^{J} \Rightarrow \mcr{w_{\circ}x_{1}}^{J}\bigr) \\
 & = \wt_{\omega(\lambda)}\bigl( \omega(x_{1}) \Rightarrow \omega(x_{2}) \bigr) \\
 & = \wt_{\omega(\lambda)}\bigl( \mcr{x_2w_{\circ}}^{\omega(J)} \Rightarrow \mcr{x_1w_{\circ}}^{\omega(J)} \bigr). 
\end{split}
\end{equation*}
\end{enumerate}
\end{lem}

\begin{proof}
Part (1) is obvious. Part (2) follows from part (1) and \cite[Proposition~4.3]{LNSSS}, 
together with the fact that $vw_{\circ} = w_{\circ} \omega(v)$ for all $v \in W$. 
Part (3) follows from part (2) since $u_{\circ}\lambda=\lambda$ and 
$v_{\circ}\omega(\lambda) = \omega(\lambda)$. 
\end{proof}

Now, let 
\begin{equation} \label{eq:se1}
\eta=(x_{1},\,\dots,\,x_{s}\,;\,
b_{0},\,b_{1},\,\dots,\,b_{s}) \in 
\QLS(\lambda) = \BB(\lambda)_{\cl},
\end{equation}
with $x_{1},\,\dots,\,x_{s} \in W^{J}$ and rational numbers $0=b_{0} < \cdots < b_{s}=1$. 
Then we see from Lemma~\ref{lem:da} that 
\begin{equation} \label{eq:dual}
\eta^{\ast}:=
(\mcr{x_{s}w_\circ}^{\omega(J)},\,\dots,\,\mcr{x_{1}w_\circ}^{\omega(J)}\,;\,
 1-b_{s},\,1-b_{s-1},\,\dots,\,1-b_{0})
\end{equation}
is a QLS path of shape $\omega(\lambda)=-w_{\circ}\lambda$; 
note that the stabilizer of the dominant integral weight 
$\omega(\lambda)=-w_{\circ}\lambda$ in $W$ is identical to 
the parabolic subgroup $W_{\omega(J)}$. 
Because $\eta^{\ast}(t)=\eta(1-t)-\eta(1)$ for $t \in [0,1]$, 
it follows from \cite[Lemma 2.1\,e)]{Li} that 
\begin{equation} \label{dualpath2a}
\wt (\eta^{\ast})=-\wt (\eta), \qquad 
(e_{j}\eta)^{\ast}=f_{j}\eta^{\ast}, \qquad 
(f_{j}\eta)^{\ast}=e_{j}\eta^{\ast}
\end{equation}
for $\eta \in \BB(\lambda)_{\cl}$ and $j \in I_{\af}$.

Next, for $\eta \in \QLS(\lambda) = \BB(\lambda)_{\cl}$ of the form \eqref{eq:se1}, 
we define $\omega(\eta)$ by: 
%
%
\begin{equation} \label{eq:se2}
\omega(\eta)=
(\omega(x_{1}),\,\dots,\,\omega(x_{s})\,;\,
b_{0},\,b_{1},\,\dots,\,b_{s}). 
\end{equation}
We can easily check by using Lemma~\ref{lem:da} that
$\omega(\eta)$ is a QLS path of shape $\omega(\lambda)=-w_{\circ}\lambda$. 
Since $(\omega(\eta))(t)=\omega(\eta(t))$ for $t \in [0,1]$, 
we see that 
\begin{equation*}
\wt(\omega(\eta))=\omega(\wt(\eta)), \qquad 
\omega(e_{j}\eta)=e_{\omega(j)}\omega(\eta), \qquad 
\omega(f_{j}\eta)=f_{\omega(j)}\omega(\eta)
\end{equation*}
for $\eta \in \BB(\lambda)_{\cl}$ and $j \in I_{\af}$. 

Finally, we set $S(\eta):=\omega(\eta^{\ast})=(\omega(\eta))^{\ast}$ 
for each $\eta \in \QLS(\lambda)=\BB(\lambda)_{\cl}$; 
by the argument above, 
we see that $S(\eta) \in \BB(\lambda)_{\cl}$. Moreover, 
it is easily checked that $S$ is an involution on 
$\BB(\lambda)_{\cl}$, which we call the {\em Lusztig involution}
(see also~\cite{LeS} for the affine version of the Lusztig involution
in type $C$), such that 
%
%
\begin{equation} \label{eq:Lus1}
\begin{split}
& \wt(S(\eta))=-\omega(\wt(\eta)) = w_{\circ}(\wt(\eta)), \\
& S(e_{j}\eta)=f_{\omega(j)}S(\eta), \qquad 
  S(f_{j}\eta)=e_{\omega(j)}S(\eta)
\end{split}
\end{equation}
for $\eta \in \BB(\lambda)_{\cl}$ and $j \in I_{\af}$. 
We remark that if $\eta$ is of the form \eqref{eq:se1}, 
then 
%
%
\begin{equation}
\begin{split}
S(\eta) & =
(\omega\mcr{x_{s}w_{\circ}}^{\omega(J)},\,\dots,\,
 \omega\mcr{x_{1}w_{\circ}}^{\omega(J)}\,;\,
1-b_{s},\,1-b_{s-1},\,\dots,\,1-b_{0}) \\
& = 
(\mcr{w_{\circ}x_{s}}^{J},\,\dots,\,
 \mcr{w_{\circ}x_{1}}^{J}\,;\,
1-b_{s},\,1-b_{s-1},\,\dots,\,1-b_{0}). \label{eq:Se}
\end{split}
\end{equation}
%
%
\begin{cor} \label{cor:tail}
Let $\eta=(x_{1},\,x_{2},\,\dots,\,x_{s}\,;\,
b_{0},\,b_{1},\,\dots,\,b_{s}) \in \QLS(\lambda) = \BB(\lambda)_{\cl}$. 
Then, 
\begin{equation*}
\Deg_{\lambda}(S(\eta))=
-\sum_{k=1}^{s-1} b_{k}\wts{x_{k+1}}{x_{k}}.
\end{equation*}
\end{cor}

\begin{proof}
It follows from Theorem~\ref{thm:deg-QLS} and~\eqref{eq:Se} that
\begin{align*}
\Deg_{\lambda}(S(\eta)) & = 
-\sum_{k=1}^{s-1} 
 \bigl\{1-(1-b_{s-k})\bigr\}
 \wts{ \mcr{w_{\circ}x_{s-k}}^{J} }{ \mcr{w_{\circ}x_{s-k+1}}^{J} } \\
& = 
-\sum_{k=1}^{s-1} 
 b_{k}
 \wts{ \mcr{w_{\circ}x_{k}}^{J} }{ \mcr{w_{\circ}x_{k+1}}^{J} }. 
\end{align*}
Also, by Lemma~\ref{lem:da}\,(3), we have 
$\wts{ \mcr{w_{\circ}x_{k}}^{J} }{ \mcr{w_{\circ}x_{k+1}}^{J} }=\wts{x_{k+1}}{x_{k}}$ 
for all $1 \le k \le s-1$. This proves the corollary. 
\end{proof}

As in Section \ref{subsec:deg-ene}, we write $\lambda$ as $\lambda=
\vpi_{i_{1}} + \vpi_{i_{2}} + \cdots + \vpi_{i_{p}}$, 
with $i_{1},\,i_{2},\,\dots,\,i_{p} \in I$, and let 
\begin{equation*}
\Psi:\BB(\lambda)_{\cl} \stackrel{\sim}{\rightarrow} 
 \BB(\vpi_{i_{1}})_{\cl} \otimes 
 \BB(\vpi_{i_{2}})_{\cl} \otimes \cdots \otimes 
 \BB(\vpi_{i_{p}})_{\cl}=:\BB
\end{equation*}
be the isomorphism of crystals; recall again that 
$\BB(\vpi_{i})_{\cl} \cong B^{i,1}$ as crystals. 
Also, let 
\begin{equation*}
\Psi^{\rev}:\BB(\lambda)_{\cl} \stackrel{\sim}{\rightarrow} 
 \BB(\vpi_{i_{p}})_{\cl} \otimes 
 \BB(\vpi_{i_{p-1}})_{\cl} \otimes \cdots \otimes 
 \BB(\vpi_{i_{1}})_{\cl}=:\BB^{\rev}
\end{equation*}
be the isomorphism of crystals in Theorem~\ref{thm:LScl}.
Furthermore, we define $S:\BB \rightarrow \BB^{\rev}$ by:
%
%
\begin{equation} \label{eq:SB}
S(\eta_{1} \otimes \cdots \otimes \eta_{p})=
S(\eta_{p}) \otimes \cdots \otimes S(\eta_{1})
\end{equation}
for $\eta_{1} \otimes \cdots \otimes \eta_{p} \in \BB=
\BB(\vpi_{i_{1}})_{\cl} \otimes \cdots \otimes \BB(\vpi_{i_{p}})_{\cl}$. 
Then we deduce that 
%
%
\begin{equation} \label{eq:Lus2}
\begin{split}
& \wt(S(\bdeta))=-\omega(\wt(\bdeta))=w_{\circ}(\wt(\bdeta)), \\
& S(e_{j}\bdeta)=f_{\omega(j)}S(\bdeta), \qquad 
  S(f_{j}\bdeta)=e_{\omega(j)}S(\bdeta)
\end{split}
\end{equation}
for $\bdeta \in \BB$ and $j \in I_{\af}$. 
By the connectedness of the crystals (see Theorem~\ref{thm:LScl}\,(2), (3)) 
and \eqref{eq:Lus1}, \eqref{eq:Lus2}, 
we have the following commutative diagram:
%
%
\begin{equation} \label{eq:CD}
\begin{CD}
\BB(\lambda)_{\cl} @>{\Psi}>> \BB \\
@V{S}VV @VV{S}V \\
\BB(\lambda)_{\cl} @>{\Psi^{\rev}}>> \BB^{\rev}
\end{CD}
\end{equation}
The next corollary follows immediately from 
Theorem~\ref{thm:ns-deg} (applied to $\BB^{\rev}$) and 
the commutative diagram \eqref{eq:CD} above. 
%
%
\begin{cor} \label{cor:taile}
For each $\eta \in \BB(\lambda)_{\cl}$, 
\begin{equation*}
\Deg_{\lambda}(S(\eta)) = 
D_{\BB^{\rev}}\bigl( \Psi^{\rev}(S(\eta)) \bigr)-D_{\BB^{\rev}}^{\ext} = 
D_{\BB^{\rev}}\bigl( S(\Psi(\eta)) \bigr)-D_{\BB^{\rev}}^{\ext}. 
\end{equation*}
\end{cor}
%
%
\begin{remark} \label{rem:LR}
The energy function $D=D_{\BB}$ in Section~\ref{subsec:deg-ene} corresponds to 
the ``right'' energy function $D^{R}$ in \cite[Section 2.4]{LeS}; 
we should remark that the order of tensor products 
of crystals in \cite{LeS} is ``opposite'' to 
that in this paper. Hence the composite 
$D_{\BB^{\rev}} \circ S : \BB \rightarrow \BZ_{\le 0}$ corresponds 
to the ``left'' energy function $D^{L}$ in \cite[Section 2.4]{LeS}. 
\end{remark}

%
\section{The quantum alcove model}
\label{section.qalc} 
Now let us recall the {\em quantum alcove model}~\cite{LL}. 
Throughout this section, we refer to roots and weights in 
the corresponding finite lattices. Fix a dominant integral weight $\lambda\in X$.

%
\subsection{The objects of the model}
\label{objmod}

We say that two alcoves are {adjacent} if they are distinct and have a common wall. Given a pair
of adjacent alcoves $A$ and $B$, we write $A \stackrel{\beta}{\longrightarrow} B$  for $\beta\in\Phi$ if the
common wall is orthogonal to $\beta$ and $\beta$ points in the direction from $A$ to $B$. Recall that
alcoves are separated by hyperplanes of the form
\[
	H_{\beta,l}=\{\mu\in \Fh^{\ast}_{\mathbb{R}} \mid \langle \beta^\vee,\mu \rangle=l\}\,,
\]
where $\Fh^{\ast}_{\mathbb{R}} = \mathbb{R} \otimes X$.
We denote by $r_{\beta,l}$ the affine reflection in this hyperplane.

\begin{definition}[\cite{LP}]
	An  \emph{alcove path} is a sequence of alcoves $(A_0, A_1, \ldots, A_m)$ such that
	$A_{j-1}$ and $A_j$ are adjacent, for $j=1,\ldots, m.$ We say that $(A_0, A_1, \ldots, A_m)$  
	is \emph{reduced} if it has minimal length among all alcove paths from $A_0$ to $A_m$.
\end{definition}
	
Let $A_{\lambda}=A_{\circ}+\lambda$ be the translation of the fundamental alcove $A_{\circ}$ by the weight $\lambda$.
The fundamental alcove is defined as
\[
	A_{\circ} = \{ \mu \in \Fh_{\mathbb{R}}^{\ast} \mid 0< \langle \alpha^\vee, \mu \rangle < 1 \quad \text{for all $\alpha\in \Phi^+$}\}\;.
\]
	
\begin{definition}[\cite{LP}]
	The sequence of roots $(\beta_1, \beta_2, \dots, \beta_m)$ is called a
	\emph{$\lambda$-chain} if 
	\[	
		A_0=A_{\circ} \stackrel{-\beta_1}{\longrightarrow} A_1
		\stackrel{-\beta_2}{\longrightarrow}\dots 
		\stackrel{-\beta_m}{\longrightarrow} A_m=A_{-\lambda}
	\]
is a reduced alcove path.
\end{definition}

A reduced alcove path $(A_0=A_{\circ},A_1,\ldots,A_m=A_{-\lambda})$ can be identified 
with the corresponding total order on the hyperplanes $H_{\beta,-l}$, to be called 
$\lambda$-hyperplanes, which separate $A_\circ$ from $A_{-\lambda}$ (i.e., are subject 
to $\beta\in\Phi^+$ and $0\le l<\langle \beta^\vee, \lambda \rangle$); we refer here to the sequence 
$H_{\beta_i,-l_i}$ for $i=1,\ldots,m$, where $H_{\beta_i,-l_i}$ contains the common wall of 
$A_{i-1}$ and $A_i$. Note also that a $\lambda$-chain $(\beta_1, \ldots, \beta_m)$ determines 
the corresponding reduced alcove path. Indeed, we can recover the corresponding sequence 
$(l_1,\ldots,l_m)$, to be called the height sequence, by setting 
$l_i:=|\left\{ j<i \mid \beta_j = \beta_i \right\} |$. Therefore, we will sometimes refer 
to the sequence of $\lambda$-hyperplanes considered above as a $\lambda$-chain. 

\begin{remark} 
An alcove path corresponds to the choice of a reduced word for the affine Weyl group 
element sending $A_\circ$ to $A_{-\lambda}$ \cite[Lemma 5.3]{LP}. Another equivalent definition 
of an alcove path/$\lambda$-chain, based on a root interlacing condition which generalizes a similar 
condition characterizing reflection orderings, can be found in~\cite[Definition 4.1 and Proposition 10.2]{LP1}.
\end{remark}

We will work with a special choice of a $\lambda$-chain in~\cite[Section 4]{LP1}, which we now recall. 

\begin{prop}[\cite{LP1}]
\label{speciallch} 
Given a total order $I=\{1<2<\dotsm<r\}$ on the set of Dynkin nodes, one may express 
a coroot $\beta^\vee=\sum_{i=1}^r c_i \alpha_i^\vee$
in the ${\mathbb Z}$-basis of simple coroots. Consider the
total order on the set of $\lambda$-hyperplanes defined by 
the lexicographic order on their images in ${\mathbb Q}^{r+1}$ under the map
\begin{equation}
\label{E:stdvec}
	H_{\beta,-l}\mapsto \frac{1}{\langle \beta^\vee, \lambda \rangle} (l,c_1,\ldots,c_r).
\end{equation}
This map is injective, thereby endowing 
the set of $\lambda$-hyperplanes with a total order, which is a $\lambda$-chain. We call it the
{\em lexicographic (lex) $\lambda$-chain}. 
\end{prop}

\begin{ex} \label{ex:qa1}
As in Example~\ref{ex:deg}, 
assume that $\Fg$ is of type $A_{2}^{(1)}$, 
and let $\lambda=\vpi_{1}+\vpi_{2}$. 
Then the $\lambda$-hyperplanes are 
$H_{\alpha_1,0}$, $H_{\alpha_2,0}$, $H_{\theta,0}$, and 
$H_{\theta,-1}$, where $\theta=\alpha_{1}+\alpha_{2}$. 
Since 
\begin{align*}
& H_{\alpha_1,0} \mapsto (0,\,1,\,0), & 
& H_{\alpha_2,0} \mapsto (0,\,0,\,1), \\
& H_{\theta,0} \mapsto (0,\,1/2,\,1/2), & 
& H_{\theta,-1} \mapsto (1/2,\,1/2,\,1/2)
\end{align*}
under the map \eqref{E:stdvec}, the corresponding total order $\prec$
on the set of $\lambda$-hyperplanes is 
$H_{\alpha_2,0} \prec H_{\theta,0} \prec H_{\alpha_1,0} \prec H_{\theta,-1}$.
\end{ex}

The objects of the quantum alcove model are defined next.

\begin{definition}[\cite{LL}]
\label{def:admissible}
	Given a $\lambda$-chain $\Gamma=(\beta_1,\,\ldots,\,\beta_m)$, a subset 
	$A=\left\{ j_1 < j_2 < \cdots < j_s \right\}$ of $[m]:=\{1,\ldots,m\}$ (possibly empty)
 	is an \emph{admissible subset} if
	we have the following path in the quantum Bruhat graph $\QB(W)$ on $W$:
	\begin{equation}
	\label{eqn:admissible}
	e \stackrel{\beta_{j_1}}{\longrightarrow} r_{\beta_{j_1}} 
	\stackrel{\beta_{j_2}}{\longrightarrow} r_{\beta_{j_1}}r_{\beta_{j_2}} 
	\stackrel{\beta_{j_3}}{\longrightarrow} \cdots 
	\stackrel{\beta_{j_s}}{\longrightarrow} r_{\beta_{j_1}}r_{\beta_{j_2}} \cdots r_{\beta_{j_s}}\,.
	\end{equation}
	The weight of $A$ (not necessarily admissible) is defined by 
	\begin{equation}
	\label{defwta}
	\wt(A):=-r_{\beta_{j_1},-l_{j_1}} \cdots 
	         r_{\beta_{j_s},-l_{j_s}}(-\lambda)\,.
	\end{equation}
 	We let ${\mathcal A}(\Gamma)$ be the collection of all admissible subsets of $[m]$.
\end{definition}

\begin{ex} \label{ex:qa2}
Keep the notation and setting of Example~\ref{ex:qa1}; 
recall that $\Gamma=(\beta_1=\alpha_2,\,\beta_2=\theta,\,\beta_3=\alpha_1,\,\beta_4=\theta)$ 
is the lex $\lambda$-chain. We see that the admissible subsets of $\bigl\{1,\,2,\,3,\,4\bigr\}$ are
\begin{align*}
& \emptyset, \quad \bigl\{1\bigr\}, \quad \bigl\{3\bigr\}, \quad
   \bigl\{1,\,2\bigr\}, \quad \bigl\{1,\,3\bigr\}, \quad 
   \bigl\{1,\,4\bigr\}, \quad \bigl\{3,\,4\bigr\}, \\
& \bigl\{1,\,2,\,3\bigr\}, \quad \bigl\{1,\,2,\,3,\,4\bigr\}.
\end{align*}
\end{ex}

\begin{remark}\label{spec}
If we restrict to admissible subsets for which the path \eqref{eqn:admissible} 
has no quantum edges, we recover the classical alcove model in \cite{LP,LP1}.
\end{remark}

%
\subsection{Root operators in the quantum alcove model}\label{rootopsqalc}

We continue to use the notation in Subsection~\ref{objmod}. 
Fix a $\lambda$-chain $\Gamma=(\beta_1,\ldots,\beta_m)$ 
and the corresponding reduced alcove path $(A_0, A_1, \ldots, A_m)$. 
In this subsection, 
we recall from \cite{LL} the construction of (combinatorial) root operators 
in the quantum alcove model, namely on the collection $\A(\Gamma)$ of admissible 
subsets of $[m]$. 

Let $A=\left\{ j_1 < j_2 < \cdots < j_s \right\}$  be an arbitrary subset of $[m]$. 
The elements of $A$ are called \emph{folding positions}. 
We ``fold'' the reduced alcove path $(A_0, A_1, \ldots, A_m)$
in the hyperplanes corresponding to these positions and obtain 
a ``folded'' alcove path; this can be recorded by a sequence of roots, namely 
$\Gamma(A)=\left( \gamma_1,\gamma_2, \dots, \gamma_m \right)$; 
	 here 
	 \begin{equation}\label{defw}
	 \gamma_i:=r_{\beta_{j_1}}r_{\beta_{j_2}} \cdots r_{\beta_{j_k}}(\beta_i)\,,
	 \end{equation}
	 with $j_k$ the largest folding position less than $i$. 
	 We define $\gamma_{\infty} := r_{\beta_{j_1}}r_{\beta_{j_2}} \cdots r_{\beta_{j_s}}(\rho)$, 
	 where $\rho=(1/2)\sum_{\alpha \in \Phi^{+}} \alpha$. 
	 Upon folding, the hyperplane separating the alcoves $A_{i-1}$ and $A_i$ is mapped to 
\begin{equation}\label{defheighthyp}
H_{|\gamma_i|,-\l{i}}=
r_{\beta_{j_1}, -l_{j_1}}
r_{\beta_{j_2}, -l_{j_2}} \cdots 
r_{\beta_{j_k}, -l_{j_k}}(H_{\beta_i, -l_i})\,,
\end{equation}
for some $\l{i}$, which is defined by this relation; 
here we write $|\alpha|:=\sgn(\alpha)\alpha$, where $\sgn(\alpha)$ 
is the sign of the root $\alpha$. 

Given $A\subseteq [m]$ and
$\alpha\in \Phi$, we will use the following notation:
\[ 
	 I_\alpha = I_{\alpha}(A):= \left\{ i \in [m] \, | \, \gamma_i = \pm \alpha \right\}\,, \qquad 
	\widehat{I}_\alpha = \widehat{I}_{\alpha}(A):= I_{\alpha} \cup \{\infty\}\,, 
\]
and $l_{\alpha}^{\infty}:=\inner{\sgn(\alpha)\alpha^{\vee}}{\wt(A)}$.

Let $A$ now be an admissible subset, so $A\in\A(\Gamma)$.
Fix $p \in I_{\af}$, so $\ti{\alpha}_p$ is a simple root 
if $p \ne 0$, or $-\theta$ if $p=0$, see \eqref{tialpha}.
We set
\begin{equation} \label{eq:M}
M:=\max \Bigl(
\big\{ \sgn(\ti{\alpha}_{p}) l_{i}^{A} \mid i \in I_{\ti{\alpha}_{p}} \bigr\} \cup
\bigl\{ \sgn(\ti{\alpha}_{p}) l_{\ti{\alpha}_{p}}^{\infty} \bigr\}\Bigr).
\end{equation}
Let $\ell$ be the minimum index $i$ 
in $\widehat{I}_{\ti{\alpha}_p}$ for which we have $\sgn(\ti{\alpha}_p)\l{i}=M$. 
It was proved in \cite{LL} that, if $M\ge\delta_{p,0}$, then either $\ell \in A$ or $\ell=\infty$; furthermore, 
if $M>\delta_{p,0}$, then $\ell$ has a predecessor $k$ in 
$\widehat{I}_{\ti{\alpha}_p}$, and we have $k\not\in A$. 
We define
\begin{equation}
	\label{eqn:rootF} 
f_p(A):= 
	\begin{cases}
		(A \backslash \left\{ \ell \right\}) \cup \{ k \} & \text{ if $M>\delta_{p,0} $ } \\
				\Bzero & \text{ otherwise. }
	\end{cases}
\end{equation}
Now we define $e_p$. Let $M$ be as in \eqref{eq:M}.
Assuming that $M > \inner{\ti{\alpha}_p^{\vee}}{\wt(A)}$, 
let $k$ be the maximum index 
$i$  in $I_{\ti{\alpha}_p}$ for which we have $\sgn(\ti{\alpha}_p)\l{i}=M$,
and let $\ell$ be the successor of $k$ in $\widehat{I}_{\ti{\alpha}_p}$. 
Assuming also that $M\ge\delta_{p,0}$, it was proved in \cite{LL} 
that $k\in A$, and either $\ell \not\in A$ or $\ell=\infty$. Define 
\begin{equation}
	\label{eqn:rootE}
e_p(A):= 
	\begin{cases}
		(A \backslash \left\{ k \right\}) \cup \{ \ell \} & \text{ if }
		M>\inner{\ti{\alpha}_p^{\vee}}{\wt(A)} \text{ and } M \geq \delta_{p,0}  \\
				\Bzero & \text{ otherwise. }
	\end{cases}
\end{equation}
In the above definitions, we use the
convention that $A\backslash \left\{ \infty \right\}= A \cup \left\{ \infty \right\} = A$. 

\begin{ex} \label{ex:ro}
Keep the notation and setting of Examples~\ref{ex:qa1} and \ref{ex:qa2}.
We can verify that
\begin{equation*}
\emptyset \stackrel{f_{2}}{\mapsto} \{1\} 
\stackrel{f_{1}}{\mapsto} \{1,\,4\} 
\stackrel{f_{1}}{\mapsto} \{1,\,2\} 
\stackrel{f_{2}}{\mapsto} \{1,\,2,\,3\}
\stackrel{f_{0}}{\mapsto} \{1,\,2,\,3,\,4\} 
\stackrel{f_{0}}{\mapsto} \bzero.
\end{equation*}
As an example, let us check that 
$f_{0}\{1,\,2,\,3\}=\{1,\,2,\,3,\,4\}$. 
We set $A=\bigl\{1,\,2,\,3\}$. Then, 
\begin{align*}
\wt (A) & 
 = -r_{\beta_{1},-l_{1}}r_{\beta_{2},-l_{2}}r_{\beta_{3},-l_{3}}(-\lambda)
 = -r_{\alpha_2,0}r_{\theta,0}r_{\alpha_{1},0}(-\lambda) \\
&  = -r_{2}r_{\theta}r_{1}(-\lambda) = w_{\circ}\lambda. 
\end{align*}
If we write $\Gamma(A)=
(\gamma_{1},\,\gamma_{2},\,\gamma_{3},\,\gamma_{4})$, then
\begin{align*}
& \gamma_{1} = \beta_{1} = \alpha_{2}, & 
& \gamma_{2} = r_{\beta_{1}}(\beta_{2}) = r_{2}(\theta)=\alpha_{1}, & \\
& \gamma_{3} = r_{\beta_{1}}r_{\beta_{2}}(\beta_{3}) = r_{2}r_{\theta}(\alpha_{1}) = \alpha_{2}, & 
& \gamma_{4} = r_{\beta_{1}}r_{\beta_{2}}r_{\beta_{3}}(\beta_{4}) = r_{2}r_{\theta}r_{1}(\theta)=-\theta,
\end{align*}
and 
\begin{align*}
& l_{1}^{A}=0 \quad \text{since $H_{\beta_{1},-l_{1}} = H_{\alpha_{2},0}$}, \\
& l_{2}^{A}=0 \quad \text{since $r_{\beta_{1},-l_{1}}(H_{\beta_{2},-l_{2}}) = 
r_{\alpha_{2},0}(H_{\theta,0}) = H_{\alpha_{1},0}$}, \\
& l_{3}^{A}=0 \quad \text{since 
$r_{\beta_{1},-l_{1}}r_{\beta_{2},-l_{2}}(H_{\beta_{3},-l_{3}}) = 
r_{\alpha_{2},0}r_{\theta,0}(H_{\alpha_{1},0}) = H_{\alpha_{2},0}$}, \\
& l_{4}^{A}=-1 \quad \text{since 
$r_{\beta_{1},-l_{1}}r_{\beta_{2},-l_{2}}r_{\beta_{3},-l_{3}}(H_{\beta_{4},-l_{4}}) = 
r_{\alpha_{2},0}r_{\theta,0}r_{\alpha_{1},0}(H_{\theta,-1}) = H_{\theta,1}$}.
\end{align*}
Also, since $\ti{\alpha}_{0}=-\theta$, 
we have $I_{\ti{\alpha}_{0}}(A) = \{4\}$. 
Because 
\begin{equation*}
\sgn(\ti{\alpha}_{0}) l_{\ti{\alpha}_{0}}^{\infty} = (-1) \times 
\Bpair{(-1) \times (-\theta)^{\vee}}{\wt (A)} = 2 > 1 = 
\sgn(\ti{\alpha}_{0})l_{4}^{A},
\end{equation*}
we obtain $M=2$ (which implies that $f_{0}A \ne \bzero)$, and hence 
$\ell = \infty$, $k=4$. Therefore, we conclude that
\begin{equation*}
f_{0}A = (A \setminus \bigl\{\infty\bigr\}) \cup \{4\} = 
\{1,\,2,\,3,\,4\}.
\end{equation*}

Next, let us check that $f_{0}B=\bzero$, with $B=\{1,\,2,\,3,\,4\}$. 
We have
\begin{equation*}
\wt (B)  
 = -r_{\beta_{1},-l_{1}}r_{\beta_{2},-l_{2}}r_{\beta_{3},-l_{3}}r_{\beta_{4},-l_{4}}(-\lambda)
 = -r_{\alpha_2,0}r_{\theta,0}r_{\alpha_{1},0}r_{\theta,-1}(-\lambda) = 0,
\end{equation*}
and hence $\sgn(\ti{\alpha}_{0}) l_{\ti{\alpha}_{0}}^{\infty} = 0$. 
Furthermore, by a computation similar to the one above for $A$,
we deduce that $\Gamma(B)=
(\alpha_{2},\,\alpha_{1},\,\alpha_{2},\,-\theta)$, and that
$l_{i}^{B}=l_{i}^{A}$ for all $1 \le i \le 4$. 
Also, since $\ti{\alpha}_{0}=-\theta$, 
we have $I_{\ti{\alpha}_{0}}(B) = \{4\}$. 
Since $\sgn(\ti{\alpha}_{0})l_{4}^{B} = 1$, 
it follows that $M=1$, which implies that $f_{0}B=\bzero$ by the definition.
\end{ex}

For more examples, we refer the reader to \cite[Examples~3.6--3.7]{LL}.

The following theorem about root operators on 
$\CA(\Gamma)$ was proved in \cite{LL}.

\begin{thm}[{\cite[Theorem 3.8]{LL}}] \label{theorem:admissible}
\mbox{}
    \begin{enumerate}
    \item[{\rm (1)}]  If $A$ is an admissible subset and if $f_p(A) \ne \Bzero$, then $f_p(A)$ is also an admissible subset. 
    Similarly for $e_p(A)$. Moreover, $f_p(A)=A'$ if and only if  $e_p(A')=A$.
    \item[{\rm (2)}] We have $ \wt(f_p(A)) = \wt(A) - \ti{\alpha}_p $. Moreover, if $M\ge\delta_{p,0}$, then
    \[
    	\varphi_p(A)=M-\delta_{p,0}\,,\;\;\;\;\varepsilon_p(A)=M-\langle\ti{\alpha}_p^\vee,\wt(A)\rangle\,,
    \]
    while otherwise $\varphi_p(A)=\varepsilon_p(A)=0$. 
    \end{enumerate}
\end{thm}

\begin{remark}\label{changepath}
Let $A=\{j_1 < \dots < j_s \}$ be an admissible subset, 
and $w_i := r_{\beta_{j_1}}r_{\beta_{j_2}} \cdots r_{\beta_{j_i}}$. 
Let $M,\,\ell,\,k$ be as in the above definition of $f_p(A)$, 
assuming $M > \delta_{p,0}$. Now assume that $\ell \ne \infty$, 
and let $a<b$ be such that 
\[
	 j_a < k < j_{a+1} < \dots < j_b = \ell < j_{b+1} \;; 
\] 
if $a=0$ or $b+1>s$, then the corresponding indices $j_a$, 
respectively $j_{b+1}$, are missing. In the proof of 
Theorem~\ref{theorem:admissible} in \cite{LL}, 
it was shown that $f_p$ has the effect of changing 
the path in the quantum Bruhat graph 
\[
e = w_{0} \rightarrow \cdots \rightarrow w_a \rightarrow w_{a+1} 
 \rightarrow \cdots \rightarrow w_{b-1} \rightarrow w_b 
 \rightarrow \cdots \rightarrow w_s
\]
corresponding to $A$ into the following path corresponding to $f_p(A)$:
\[
e = w_{0} \rightarrow \cdots \rightarrow w_a \rightarrow s_p w_a \rightarrow 
 s_p w_{a+1} \rightarrow \cdots \rightarrow s_p w_{b-1} = w_b \rightarrow
 \cdots \rightarrow w_s\,,
\]
where $s_{p}$ is the simple reflection $r_{p}$ if $p \ne 0$, or 
$r_{\theta}$ if $p=0$, see \eqref{tialpha}. The case $\ell=\infty$ is similar.
\end{remark}

%
\section{The bijection between the quantum LS path model and the quantum alcove model}
\label{section.bijection}

The main result of this section is the crystal isomorphism between the QLS paths of Section~\ref{sec:qLS}
and the quantum alcove model of Section~\ref{section.qalc} as stated in Theorem~\ref{qalc2qls}.

%
\subsection{The forgetful map}
\label{S:forget} 
Fix a dominant integral weight $\lambda$, and set 
$J=\bigl\{i \in I \mid \pair{\alpha_{i}^{\vee}}{\lambda}=0\bigr\}$. 
Recall from Sections \ref{subsec:qLS} and \ref{objmod}
the notation related to QLS paths and the quantum alcove model, 
respectively. 

We will now define a forgetful map from the quantum 
alcove model based on a lex $\lambda$-chain 
$\Gamma_{\rm lex}=(\beta_1,\ldots,\beta_m)$ (see Proposition~\ref{speciallch}), 
namely from ${\mathcal A}(\lambda):={\mathcal A}(\Gamma_{\rm lex})$, 
to the set $\QLS(-w_{\circ}\lambda)$ of QLS paths of shape $-w_{\circ}\lambda$. 
Given an index $i\in[m]$, we let $t_i:=l_i/\langle \beta_i^\vee,\lambda \rangle$, where $l_i$ 
is the height defined in Section~\ref{section.qalc}. Note that $0\le t_1\le t_2\le\cdots\le t_m$, 
by the definition of $\Gamma_{\mathrm{lex}}$. 
Consider an admissible subset $A=\{ j_1 < j_2 < \cdots < j_s\}$, and let
\begin{equation} \label{eq:bq}
	\{0=b_0<b_1 < \cdots < b_p\}:=
	\{t_{j_1}\le t_{j_2}\le \cdots\le t_{j_s}\}\cup\{0\}\,.
\end{equation}
Let $0=n_{0}\le n_1<\cdots<n_{p+1}=s$ be such that $t_{j_h}=b_k$ if and only if $n_{k}<h\le n_{k+1}$, 
for $k=0,\ldots,p$. Define Weyl group elements $u_h$ for $h=0,\ldots,s$ and $w_k$ for $k=0,\ldots,p$ 
by $u_0:=e$, $u_h:=r_{\beta_{j_1}} \cdots r_{\beta_{j_h}}$, and 
$w_k:=u_{n_{k+1}}$. 
For any $k=1,\ldots,p$, we have the following directed path 
in the quantum Bruhat graph $\QB(W)$:
\begin{equation} \label{pathk}
	w_{k-1}=u_{n_{k}} 
	\xrightarrow{\beta_{n_k+1}} u_{n_{k}+1} 
	\xrightarrow{\beta_{n_k+2}} \cdots 
	\xrightarrow{\beta_{n_{k+1}}} u_{n_{k+1}}=w_k\,.
\end{equation}
We claim that this is a directed path in $\QB_{b_k\lambda}(W)$. 
Indeed, for $n_{k}< h \le n_{k+1}$, we have
\[  
	b_k\langle \beta_{j_h}^\vee, \lambda \rangle=l_{j_h} \in \BZ_{\ge 0}\,,
\]
by the definition of $b_k=t_{j_h}$.  By Lemma~\ref{q2ls} below, 
for each edge $u_i \rightarrow u_{i+1}$ in the path~\eqref{pathk}, 
there is a directed path from $\mcr{u_i}^{J}$ to $\mcr{u_{i+1}}^{J}$ 
in $\QB_{b_k\lambda}(W^J)$. 
Concatenating these directed paths in $\QB_{b_k\lambda}(W^J)$, 
we obtain a directed path from $\mcr{w_{k-1}}^{J}$ to $\mcr{w_k}^{J}$ 
in $\QB_{b_k\lambda}(W^J)$. Hence it follows from Lemma~\ref{lem:da} 
that there is a directed path from 
$\mcr{w_{k}w_{\circ}}^{\omega(J)}$ to 
$\mcr{w_{k-1}w_{\circ}}^{\omega(J)}$ in 
$\QB_{b_k \omega(\lambda)}(W^{\omega(J)})$; 
note that $\omega(\lambda)=-w_{\circ}\lambda$. 
Thus we conclude that 
\begin{equation} \label{e:lspath}
    \mcr{w_{0}w_{\circ}}^{\omega(J)} \blarrl{-b_1w_{\circ}\lambda} 
    \mcr{w_{1}w_{\circ}}^{\omega(J)} \blarrl{-b_2w_{\circ}\lambda} 
    \cdots \blarrl{-b_pw_{\circ}\lambda} 
    \mcr{w_{p}w_{\circ}}^{\omega(J)}
\end{equation}
is a QLS path of shape $-w_{\circ}\lambda$.
We denote this QLS path of shape $-w_{\circ}\lambda$ 
by $\Pi(A)$, and the dual QLS path of shape $\lambda$
defined in Section~\ref{subsec:Lusztig} by $\Pi^{*}(A)$; 
that is, $\Pi^{\ast}(A) \in \QLS(\lambda)$ is of the form:
%
%
\begin{equation} \label{eq:PiastA}
\lfloor w_{p} \rfloor^J
\blarrl{(1-b_p)\lambda} \lfloor w_{p-1} \rfloor^J
\blarrl{(1-b_{p-1})\lambda} \cdots 
\blarrl{(1-b_2)\lambda} \lfloor w_{1} \rfloor^J
\blarrl{(1-b_1)\lambda} \lfloor w_{0} \rfloor^J, 
\end{equation}
and we have the following commutative diagram: 
\begin{equation} 
\begin{diagram}
\node{\CA(\lambda)} \arrow{e,t}{\Pi} \arrow{se,b}{\Pi^{\ast}} 
\node{\QLS(-w_{\circ}\lambda)} \arrow{s,r}{\ast} \\
\node{} \node{\QLS(\lambda).}
\end{diagram}
\end{equation}

\begin{ex} \label{ex:qa3}
Keep the notation and setting of Example~\ref{ex:qa2}.

\begin{enumerate}
\item Let us compute the image of 
$A_{1}=\bigl\{1,\,2,\,3,\,4\bigr\} \in \CA(\lambda)$ 
under the map $\Pi^{\ast}:\CA(\lambda) \rightarrow \QLS(\lambda)$. 
Recall that $t_{i}=l_{i}/\pair{\beta_{i}^{\vee}}{\lambda}$ for $1 \le i \le 4$. 
Since $l_{1}=l_{2}=l_{3}=0$ and $l_{4}=1$, and since 
$\pair{\beta_{1}^{\vee}}{\lambda}=\pair{\beta_{3}^{\vee}}{\lambda}=1$ 
and $\pair{\beta_{2}^{\vee}}{\lambda} = \pair{\beta_{4}^{\vee}}{\lambda} =2$, 
we see that $t_{1}=t_{2}=t_{3}=0 < 1/2 = t_{4}$. 
Hence, by \eqref{eq:bq}, we have $b_{0}=0$ and $b_{1}=1/2$ (with $p=1$). 
Also, we see that $n_{1}=3$ and $n_{2}=4$, and hence that
\begin{equation*}
w_{0}=u_{n_{1}} = r_{2}r_{\theta}r_{1}=r_{\theta}=w_{\circ}, \qquad 
w_{1}=u_{n_{2}} = r_{2}r_{\theta}r_{1}r_{\theta}=e.
\end{equation*}
Therefore, by definition \eqref{eq:PiastA}, we obtain
\begin{equation*}
\Pi^{\ast}(A_{1})=(e,\,w_{\circ}\,;\,0,\,1/2,\,1),
\end{equation*}
which is equal to $\eta_{1} \in \QLS(\lambda)$ in Example~\ref{ex:deg}. 

\item Let us compute the image of 
$A_{2}=\bigl\{1,\,4\bigr\} \in \CA(\lambda)$ 
under the map $\Pi^{\ast}:\CA(\lambda) \rightarrow \QLS(\lambda)$. 
As seen above, $t_{1}=0 < 1/2 = t_{4}$. 
Hence, by \eqref{eq:bq}, we have $b_{0}=0$ and $b_{1}=1/2$ 
(with $p=1$). Also, we see that $n_{1}=1$ and $n_{2}=2$, 
and hence that
\begin{equation*}
w_{0}=u_{n_{1}} = r_{2}, \qquad 
w_{1}=u_{n_{2}} = r_{2}r_{\theta}=r_{1}r_{2}.
\end{equation*}
Therefore, by definition \eqref{eq:PiastA}, we obtain
\begin{equation*}
\Pi^{\ast}(A_{2})=(r_{1}r_{2},\,r_{2}\,;\,0,\,1/2,\,1),
\end{equation*}
which is equal to $\eta_{2} \in \QLS(\lambda)$ in Example~\ref{ex:deg}. 

\item 
By computations similar to the ones above, we can verify that
\begin{align*}
& \emptyset \mapsto (e\,;\,0,\,1), & 
& \bigl\{1\bigr\} \mapsto (r_{2} \,;\,0,\,1), \\
& \bigl\{1,\,2\bigr\} \mapsto (r_{1}r_{2}\,;\,0,\,1), & 
& \bigl\{1,\,2,\,3\bigr\} \mapsto (w_{\circ} \,;\,0,\,1)
\end{align*}
under the map $\Pi^{\ast}:\CA(\lambda) \rightarrow \QLS(\lambda)$. 
\end{enumerate}
\end{ex}

We now state Lemma~\ref{q2ls}, 
which is the main ingredient in the above construction. 
This lemma will be proved in Section~\ref{sq2ls} below.

\begin{lem} \label{q2ls} 
Let $w\xrightarrow{\gamma} wr_\gamma$ be an edge in $\QB_{b\lambda}(W)$ 
for some rational number $b$, which is viewed as a path $\bq$. 
Then there exists a path $\bp$ from $\lfloor w\rfloor$ to 
$\lfloor w r_\gamma\rfloor$ in $\QB_{b\lambda}(W^J)$ (possibly of length $0$), such that 
$\wt(\bp)\equiv\wt(\bq) \mod Q_J^\vee$. 
\end{lem}

\begin{remarks} \mbox{}
\begin{enumerate}
\item The special case of the lemma corresponding to the $b$-Bruhat order on $W$ and $W^J$
(i.e., the subgraphs of $\QB_{b\lambda}(W)$ and $\QB_{b\lambda}(W^J)$ with no quantum edges) 
was proved in \cite[Lemma 4.16]{LeSh}. 
For $b=0$, i.e., the usual Bruhat order, the latter result is well-known; 
see, e.g., \cite[Proposition 2.5.1]{BB}. 
\item Based on the strong connectivity of the quantum Bruhat graph (cf. Theorem \ref{thm:shell} below), 
the lemma implies the same property for the parabolic quantum Bruhat graph. Note that this was proved 
by different methods (and, in fact, in a slightly stronger form) in~\cite[Lemma 6.12]{LNSSS}.
\end{enumerate}
\end{remarks}

%
\subsection{The inverse map}\label{S:invmap} 
Next we prove that the forgetful map in Section \ref{S:forget}, from the quantum alcove model to quantum 
LS paths, is a bijection, by exhibiting the inverse map. 
We will use the {\em shellability} of the quantum Bruhat graph $\QB(W)$ 
with respect to a reflection ordering on the positive 
roots~\cite{Dyer}, which we now recall. 

\begin{thm}[\cite{BFP}] \label{thm:shell} 
Fix a reflection ordering on $\Phi^+$.
\begin{enumerate}
\item For any pair of elements $v,w\in W$, there is a unique path from $v$ to $w$ in the quantum 
Bruhat graph $\QB(W)$ such that its sequence of edge labels is strictly increasing (resp., decreasing) 
with respect to the reflection ordering.
\item The path in {\rm (1)} has the smallest possible length $\ell(v\to w)$ and is lexicographically minimal 
(resp., maximal) among all shortest paths from $v$ to $w$.
\end{enumerate}
\end{thm}

In \cite[Section 4.3]{LeSh}, we constructed a reflection ordering $<_\lambda$ on $\Phi^+$ which depends on $\lambda$. 
The bottom of the order $<_\lambda$ consists of the roots in $\Phi^+\setminus\Phi_J^+$. For two such roots 
$\alpha$ and $\beta$, define $\alpha<\beta$ whenever 
the hyperplane $H_{(\alpha,0)}$ precedes $H_{(\beta,0)}$ in the lex $\lambda$-chain (see Proposition~\ref{speciallch}).
This forms an \emph{initial section} \cite{Dyer} of $<_\lambda$.
The top of the order $<_\lambda$ consists of the positive roots for the Weyl group $W_J$,
and we fix any reflection ordering for them. We refer to the reflection ordering $<_\lambda$ throughout this section.

\begin{remark}
Given a $\lambda$-hyperplane $H_{\beta,-l}$, we call the first component of the vector associated with it 
in~\eqref{E:stdvec}, namely $l/\langle\beta^\vee,\lambda\rangle$, the {\em relative height} of $H_{\beta,-l}$. It is 
not hard to see that, in the lex $\lambda$-chain, the order on the $\lambda$-hyperplanes $H_{\beta,-l}$ with the 
same relative height is given by the order $<_\lambda$ on the corresponding roots $\beta$. We will use this fact implicitly below. 
\end{remark}

Recall from~\cite[Proposition 7.2]{LNSSS} that
there exists a unique element $x \in wW_{J}$ such that 
$\ell(v \to x)$ attains its minimum value as a function 
of $x\in wW_J$, for fixed $v,w\in W$.
We refer also to~\cite[Theorem 7.1]{LNSSS}, stating 
that the mentioned minimum is, in fact, attained by
the minimum of the coset $wW_J$ with respect to 
the {\em $v$-tilted Bruhat order} $\preceq_v$ on $W$~\cite{BFP}; 
therefore, it makes sense to denote it by $\min(wW_J,\preceq_v)$,
 although we will not use this stronger result. 

\begin{lem} \label{L:mainb} 
Consider $\sigma,\tau\in W^J$ and $w_J\in W_J$. 
Write $\min(\tau W_J,\preceq_{\sigma w_J}) \in \tau W_J$ as{\rm:}
$\tau w_J'=\min(\tau W_J,\preceq_{\sigma w_J})$, with $w_J' \in W_{J}$.
\begin{enumerate}
\item There is a unique path in $\QB(W)$ from $\sigma w_J$ to some $x\in\tau W_J$ whose edge labels 
are increasing and lie in $\Phi^+\setminus\Phi_J^+$. This path ends at $\tau w_J'$.
\item Assume that there is a path from $\sigma$ to $\tau$ in $\QB_{b\lambda}(W^J)$ for some $b\in{\mathbb Q}$. 
Then the path in {\rm (1)} from $\sigma w_J$ to $\tau w_J'$ is in $\QB_{b\lambda}(W)$.
\end{enumerate}
\end{lem}

\begin{proof} 
The first part is just the content of~\cite[Lemmas 7.4 and 7.5]{LNSSS}, based on the results 
recalled just above the lemma. For the second part, we start by considering a second path in $\QB(W)$ from $\sigma w_J$ 
to $\tau w_J'$, beside the one given by (1). This path is formed by concatenating the following:
\begin{itemize}
\item 
a path from $\sigma w_J$ to $\sigma$ with only quantum edges and all edge labels in $\Phi_J^+$ 
(for instance simple roots in $\Phi_J$);
\item 
a path from $\sigma$ to $\tau$ constructed from the path in $\QB(W^J)$ between the same two elements 
by replacing each edge $u\overset{\alpha}\longrightarrow \lfloor ur_\alpha\rfloor$ with 
$u\overset{\alpha}\longrightarrow ur_\alpha$ (cf.~\cite[Condition ($2'$) in Section 4.2]{LNSSS}) 
followed by a path from $ur_\alpha$ to $\lfloor ur_\alpha\rfloor$ with only quantum edges and all edge labels 
in $\Phi_J^+$ (for instance simple roots in $\Phi_J$);
\item 
a path from $\tau$ to $\tau w_J'$ with only Bruhat edges and all edge labels in $\Phi_J^+$.
\end{itemize}

By Theorem \ref{thm:shell}~(2), we know that the first of the two paths above is a shortest one 
(from $\sigma w_J$ to $\tau w_J'$). Furthermore, by the hypothesis, the second path is in 
$\QB_{b\lambda}(W)$ (any edge in $\QB(W)$ labeled by a root in $\Phi_J^+$ is by default in $\QB_{b\lambda}(W)$). 
So we can apply Lemma \ref{bqbpaths} below and deduce that the first path is also in $\QB_{b\lambda}(W)$.
\end{proof}

Let us now state Lemma \ref{bqbpaths}, which will be proved in Section \ref{sbqbpaths} below.

\begin{lem}
\label{bqbpaths} 
Consider two paths in $\QB(W)$ between some $v$ and $w$. Assume that the first one is a shortest 
path, while the second one is in $\QB_{b\lambda}(W)$, for some rational number $b$. 
Then the first path is in $\QB_{b\lambda}(W)$ as well.
\end{lem}

We now construct the inverse of the forgetful map in Section \ref{S:forget}. 
We begin with a QLS path in $\QLS(-w_{\circ}\lambda)$, 
which is written in the form 
\begin{equation} \label{qls2map1}
\sigma_0' 
\blarrl{-b_1w_{\circ}\lambda} \sigma_1' 
\blarrl{-b_2w_{\circ}\lambda} \cdots 
\blarrl{-b_pw_{\circ}\lambda} \sigma_p'\,,
\end{equation}
where $\sigma_i' \in W^{\omega(J)}$, and 
$0=b_{0} < b_{1} < \cdots < b_{p} < 1$. 
Set $\sigma_{i}:=\mcr{\sigma_{i}'w_{\circ}}^{J} \in W^{J}$; 
we see from Lemma~\ref{lem:da} that
\begin{equation} \label{qls2map}
\sigma_0 \barrl{b_1\lambda} \sigma_1 \barrl{b_2\lambda} \cdots \barrl{b_p\lambda}\sigma_p\,.
\end{equation}
We will now associate with it an admissible subset (see Definition \ref{def:admissible}), i.e., a lex increasing 
sequence of $\lambda$-hyperplanes, and the corresponding path in $\QB(W)$ defined in \eqref{eqn:admissible}. 

We start by defining the sequence $w_{-1},w_{0},\ldots,w_p$ in $W$ recursively by $w_{-1}=e$, and by 
$w_i=\min(\sigma_i W_J,\preceq_{w_{i-1}})$ for $i=0,\ldots,p$. Note that $w_{0}=\sigma_0$. For each 
$i=0,\ldots,p$, consider the unique path in $\QB(W)$ with increasing edge labels (with respect to 
$<_\lambda$, cf. Theorem \ref{thm:shell}) from $w_{i-1}$ to $w_i$. Note that the path corresponding to $i=0$ 
is, in fact, a saturated chain in the Bruhat order on $W^J$. By \eqref{qls2map} and Lemma \ref{L:mainb}, for 
any $i$, the edges of the corresponding path are in $\QB_{b_i\lambda}(W)$ and no edge label is in $\Phi_J^+$. We 
define the path in the quantum alcove model corresponding to the QLS path \eqref{qls2map} by 
concatenating the paths constructed above, for $i=0,\ldots,p$. The corresponding sequence of 
$\lambda$-hyperplanes is defined by associating with a label $\beta$ in the path from $w_{i-1}$ to $w_i$ 
the $\lambda$-hyperplane $(\beta,-b_i\langle\beta^\vee,\lambda\rangle)$; this is indeed a $\lambda$-hyperplane:  
$b_i\langle\beta^\vee,\lambda\rangle$ is an integer, and we have $0\le b_i\langle \beta^\vee,\lambda\rangle<
\langle\beta^\vee,\lambda\rangle$, as $0\le b_i<1$ and $\beta\in\Phi^+\setminus\Phi_J^+$, so 
$\langle\beta^\vee,\lambda\rangle>0$. The constructed sequence of $\lambda$-hyperplanes is lex increasing 
because the sequence $(b_i)$ is increasing and the edge labels in the path from $w_{i-1}$ to $w_i$ increase 
(with respect to $<_\lambda$). So we constructed an admissible sequence, which is now associated 
with the QLS path in~\eqref{qls2map}. 

\begin{prop}
\label{proposition.bijqls}
The forgetful map $A\mapsto \Pi^*(A)$ is a weight-preserving bijection from ${\mathcal A}(\lambda)$ to $\bcllt{\lambda}$.
\end{prop}

\begin{proof} 
We need to show that the maps in Sections~\ref{S:forget} and~\ref{S:invmap} are mutually inverse. The 
crucial fact to check is that the map $\Pi$ followed by the backward one is the identity. This follows 
from the uniqueness part in Lemma \ref{L:mainb}~(1). 

It remains to prove that the map $\Pi^{\ast}$ preserves weights. 
We proceed by induction on the cardinality of $A \in \CA(\lambda)$. 
Note that the assertion is obvious if $A=\emptyset$. 
Now, let us take $A=\bigl\{j_{1} < \cdots < j_{s-1} < j_{s} \bigr\} \in \CA(\lambda)$, 
and set $A':=\bigl\{j_{1} < \cdots < j_{s-1} \bigr\}$. 
Then, by \eqref{eq:PiastA}, the QLS path $\Pi^{\ast}(A') \in \QLS(\lambda)$ is of the form (where we dropped the superscript $J$):
\begin{equation*}
\lfloor w_{p} \rfloor 
\blarrl{(1-b_p)\lambda} \lfloor w_{p-1} \rfloor 
\blarrl{(1-b_{p-1})\lambda} \cdots 
\blarrl{(1-b_2)\lambda} \lfloor w_{1} \rfloor
\blarrl{(1-b_1)\lambda} \lfloor w_{0} \rfloor 
\end{equation*}
in the notation of Section~\ref{S:forget}, with $s-1$ instead of $s$; 
note that $w_{p}=r_{\beta_{j_1}} \cdots r_{\beta_{j_{s-1}}}$ and $b_{p}=t_{j_{s-1}}$.
Moreover, noting that $A=A' \cup \bigl\{j_{s}\bigr\}$, we deduce that the QLS path
$\Pi^{\ast}(A) \in \QLS(\lambda)$ is of the form:
\begin{equation*}
\begin{cases}
\lfloor w_{p}r_{\beta_{j_s}} \rfloor 
\blarrl{(1-b_p)\lambda} \lfloor w_{p-1} \rfloor 
\blarrl{(1-b_{p-1})\lambda} \cdots 
\blarrl{(1-b_1)\lambda} \lfloor w_{0} \rfloor
 & \text{if $t_{j_s} = b_{p}$}, \\[3mm]
\lfloor w_{p}r_{\beta_{j_s}} \rfloor  
\blarrl{(1-t_{j_s})\lambda} \lfloor w_{p} \rfloor 
\blarrl{(1-b_p)\lambda} \lfloor w_{p-1} \rfloor  
\blarrl{(1-b_{p-1})\lambda} \cdots 
\blarrl{(1-b_1)\lambda} \lfloor w_{0} \rfloor
 & \text{if $t_{j_s} > b_{p}$}. 
\end{cases}
\end{equation*}
From this, by direct calculation using \eqref{eq:eta}, 
we can show that in both cases above, 
\begin{align*}
\wt \bigl( \Pi^{\ast}(A) \bigr) 
  & = \bigl(\Pi^{\ast}(A)\bigr)(1) 
    = \bigl(\Pi^{\ast}(A')\bigr)(1)-
      \pair{\beta_{j_s}^{\vee}}{\lambda}(1-t_{j_s})w_{p}\beta_{j_s} \\
  & = \wt \bigl( \Pi^{\ast}(A') \bigr) -
      \pair{\beta_{j_s}^{\vee}}{\lambda}(1-t_{j_s})
      r_{\beta_{j_1}} \cdots r_{ \beta_{j_{s-1}} }\beta_{j_s}, 
\end{align*}
thus obtaining a relation 
between $\wt \bigl( \Pi^{\ast}(A) \bigr)$ and $\wt \bigl( \Pi^{\ast}(A') \bigr)$. 

If we set
\begin{equation*}
z_{A'}:=r_{ \beta_{j_1},-l_{j_1} } \cdots r_{ \beta_{j_{s-1}}, -l_{j_{s-1}} } \in W_{\af},
\end{equation*}
then we have $z_{A'}\mu=\gamma_{A'}+w_{A'}\mu$ for all $\mu \in X$, 
where $w_{A'}:=r_{\beta_{j_1}} \cdots r_{ \beta_{j_{s-1}} } \in W$ and 
$\gamma_{A'}$ is an element of $Q$. It follows that by \eqref{defwta}, 
\begin{align*}
\wt (A') & 
 = - r_{ \beta_{j_1},-l_{j_1} } \cdots r_{ \beta_{j_{s-1}}, -l_{j_{s-1}} } (-\lambda)
 = -z_{A'}(-\lambda) \\ 
& = -\gamma_{A'}-w_{A'}(-\lambda). 
\end{align*}
Also, we have
\begin{equation*}
r_{ \beta_{j_{s}}, -l_{j_{s}} }(-\lambda) = 
 -\lambda + \bigl( \pair{\beta_{j_{s}}^{\vee}}{\lambda} - l_{j_{s}} \bigr) \beta_{j_{s}}.
\end{equation*}
Therefore, again by \eqref{defwta}, we see that 
\begin{align*}
\wt(A) & = 
 - r_{ \beta_{j_1},-l_{j_1} } \cdots r_{ \beta_{j_{s-1}}, -l_{j_{s-1}} }
 r_{ \beta_{j_s}, -l_{j_s} } (-\lambda) \\
& = -z_{A'}r_{ \beta_{j_s}, -l_{j_s} } (-\lambda) 
  = - z_{A'}\bigl(-\lambda + 
    (\pair{\beta_{j_s}^{\vee}}{\lambda}-l_{j_s})\beta_{j_s}\bigr) \\
& = -\gamma_{A'}-w_{A'}(-\lambda) - 
    (\pair{\beta_{j_s}^{\vee}}{\lambda}-l_{j_s})w_{A'}\beta_{j_s} \\
& = \wt(A')-(\pair{\beta_{j_s}^{\vee}}{\lambda}-l_{j_s})
  r_{\beta_{j_1}} \cdots r_{ \beta_{j_{s-1}} }\beta_{j_s}.
\end{align*}
Since $l_{j_s}=\pair{\beta_{j_s}^{\vee}}{\lambda}t_{j_s}$ 
by the definition of $t_{j_s}$, we conclude that 
\begin{equation*}
\wt (A) = \wt (A') - 
      \pair{\beta_{j_s}^{\vee}}{\lambda}(1-t_{j_s})
      r_{\beta_{j_1}} \cdots r_{ \beta_{j_{s-1}} }\beta_{j_s}, 
\end{equation*}
thus obtaining the same relation between $\wt(A)$ and $\wt(A')$ 
as the one between $\wt (\Pi^{\ast}(A))$ and $\wt (\Pi^{\ast}(A'))$. 
This proves that $\wt (\Pi^{\ast}(A)) = \wt (A)$ by induction on $s$. 
\end{proof}

%
\subsection{The crystal isomorphism between
  \texorpdfstring{${\mathcal A}(\lambda)$}{A(lambda)} and
  \texorpdfstring{$\mathbb B$}{B}}
\label{section.isom}

We will now prove that, up to the $f_0$ arrows at the end of a string, 
we can view ${\mathcal A}(\lambda)$ as a model for the tensor product of 
KR crystals $\mathbb B$ via the bijection $\widetilde{\Psi}:=\Psi\circ\Pi^*$, see \eqref{defpsi};
\begin{equation}
\begin{diagram}
\node{\CA(\lambda)} \arrow{e,t}{\Pi} \arrow{se,b}{\Pi^{\ast}} 
\node{\QLS(-w_{\circ}\lambda)} \arrow{s,r}{\ast} \\
\node{} \node{\QLS(\lambda)} \arrow{e,t}{\Psi} \node{\BB.}
\end{diagram}
\end{equation}
\begin{definition}
Let $b\rightarrow f_i(b)$ be an arrow in $\mathbb B$. It is called a \emph{Demazure arrow} if $i \ne 0$, or $i=0$ and 
$\varepsilon_0(b)\ge 1$. It is called a \emph{dual Demazure arrow} if $i \ne 0$, or $i=0$ and 
$\varphi_0(b)\geq 2$. 
\end{definition}

\begin{remark}\label{perfect}
In the case when all of the tensor factors of $\mathbb B$ are perfect
crystals (see Definition \ref{def:perfect}), the subgraph of $\mathbb B$ 
consisting of the dual Demazure arrows is connected.
See the discussion in Section~\ref{section.perfectness} below 
about which column shape KR crystals are perfect. 
\end{remark}

We now state the main result of this section, 
relating the crystal structures in the QLS path model and 
the quantum alcove model.

\begin{thm}\label{qalc2qls} 
Consider the root operator $e_p$ 
(and the corresponding map $\varepsilon_p$) for QLS paths, 
as defined in Section~{\rm \ref{subsec:crystal}},
and the root operator $f_p$ in the quantum alcove model defined 
in Section~{\rm \ref{rootopsqalc}}. 
Given $A$ in ${\mathcal A}(\lambda)$, 
we have $f_p(A)\ne\Bzero$ if and only if 
$\varepsilon_p(\Pi(A))>\delta_{p,0}$; in this case, we have
\[
	e_p(\Pi(A))=\Pi(f_p(A))\,.
\]
\end{thm}

\begin{proof}
The proof of the similar result for the classical alcove model, namely \cite[Theorem 9.4]{LP1}, carries 
through (cf. Remark \ref{spec}). The main fact underlying this proof is the similarity between the 
definition~\eqref{defeqls} of $e_p$ for QLS paths, and the change under $f_p$ of the relevant path in 
the quantum Bruhat graph, which is explained in Remark \ref{changepath}; note that in both cases the 
reflection $s_p$ is applied to a segment of the corresponding path. 

To be more precise, the proof is based on deforming the path $\Pi(A)$ to a path $\widehat{\Pi}_\varepsilon(A)$ 
between the same endpoints (where $\varepsilon$ is a sufficiently small positive real number), such that the 
latter does not pass through the intersection of two or more $\lambda$-hyperplanes (here we exclude the endpoints). 
The path $\widehat{\Pi}_\varepsilon(A)$ encodes the same information 
as the ``folded'' alcove path corresponding to $A$ or, equivalently, 
the sequence of roots $\Gamma(A)$; see Section \ref{rootopsqalc}. 
Therefore, the actions of $e_p$ on $\widehat{\Pi}_\varepsilon(A)$ 
and of $f_p$ on $A$ (where the latter is based on $\Gamma(A)$) are 
equivalent. The proof concludes by taking the limit $\varepsilon\to 0$, 
under which $\widehat{\Pi}_\varepsilon(A)$ goes to $\Pi(A)$. 

We will now point out the additional elements in the proof. 
First, some results invoked in the proof of \cite[Theorem 9.4]{LP1}
need to be replaced, as follows: \cite[Corollary 6.11]{LP1} 
with \cite[Propositions 3.15 and 3.18]{LL}, \cite[Corollary 6.12]{LP1}
with \cite[Propositions 3.16 and 3.19]{LL}, and 
\cite[Proposition 7.3]{LP1} with Remark \ref{changepath}. Other than this, 
there is just one notable addition to the proof, 
which has to do with the case $p=0$. 
Consider the number $M$ in the definition of $f_0(A)$, 
and assume for the moment that $M\ge 1$. By the same reasoning as in \cite{LP1}, 
we can see that the minimum of the function 
$t\mapsto\langle\ti{\alpha}_0^\vee,\widehat{\Pi}_\varepsilon(A)(t)\rangle$ is 
$-M$. Therefore, as discussed in \cite[Section 2.2]{LNSSS2}, cf. 
also \cite[Lemma 2.1 (c)]{Li}, the maximum number of times $e_0$ 
can be applied to $\widehat{\Pi}_\varepsilon(A)$ is $M$. 
Meanwhile, the maximum number of times $f_0$ can be applied to $A$ is $M-1$, 
by Theorem \ref{theorem:admissible} (2). In the remaining case, namely 
$M<1$, we have $f_0(A)=\Bzero$, and the minimum of 
the function mentioned above is $0$, so $e_0$ is not 
defined on  $\widehat{\Pi}_\varepsilon(A)$. 
We conclude that $f_0(A)\ne\Bzero$ if and only if 
$\varepsilon_0(\widehat{\Pi}_\varepsilon(A))\ge 2$. 
The rest of the argument is identical to the one in \cite{LP1}.
\end{proof}

\begin{remarks}\label{ndd} \mbox{}

\noindent
(1) The forgetful map $\Pi$ from the quantum alcove model to the QLS path model is 
a very natural map. Therefore, we think of the former model as a mirror image of 
the latter, via this bijection. If we use the mentioned identification 
to construct the non-dual Demazure arrows in the quantum alcove model, 
we quickly realize that, in general, the constructions are 
considerably more involved than \eqref{eqn:rootF} and \eqref{eqn:rootE},
see~\cite[Example 4.9]{LL}. 

\noindent
(2) Although the quantum alcove model so far misses 
the non-dual Demazure arrows, it has the advantage of being a discrete model. 
Therefore, combinatorial methods are applicable, for instance 
in proving the independence of the model from the choice of 
an initial alcove path (or $\lambda$-chain of roots), see below, 
including the application in Remark \ref{perfectcase}\,(2). 
This should be compared with the subtle continuous arguments used for the similar 
purpose in the Littelmann path model \cite{Li}. 
\end{remarks}

Based on \eqref{dualpath2a}, 
we immediately obtain the following corollary of 
Theorems \ref{thm:LScl}, \ref{thm:LS=QLS}, \ref{qalc2qls}, 
and Proposition \ref{proposition.bijqls}.

\begin{cor}\label{mainthm}
The bijection $\widetilde{\Psi}=\Psi\circ\Pi^*$ is 
a weight-preserving affine crystal isomorphism 
from ${\mathcal A}(\lambda)$ to 
the subgraph of $\mathbb B$ consisting 
of the dual Demazure arrows. 
\end{cor}

Recall that the set ${\mathcal A}(\lambda)={\mathcal A}(\Gamma_{\rm lex})$ in Corollary \ref{mainthm} is based 
on a lex $\lambda$-chain $\Gamma_{\rm lex}$. In fact, the following stronger version of 
Corollary \ref{mainthm} is proved in \cite{LL2}.

\begin{thm}\label{mainconj}\cite{LL2}
 Given any $\lambda$-chain $\Gamma$, there is a weight-preserving affine crystal isomorphism between ${\mathcal A}(\Gamma)$ 
 and the subgraph of $\mathbb B$ consisting of the dual Demazure arrows. 
\end{thm}

The proof uses Corollary \ref{mainthm} as the starting point. Then, given 
two $\lambda$-chains $\Gamma$ and $\Gamma'$, we construct a bijection between ${\mathcal A}(\Gamma)$ 
and ${\mathcal A}(\Gamma')$ preserving the dual Demazure arrows, as well as the weights and heights of the vertices 
(see Definition~\ref{defht} below); this means that the quantum alcove model 
does not depend on the choice of a $\lambda$-chain. The mentioned construction is based on generalizing to the quantum 
alcove model the so-called Yang-Baxter moves in \cite{Le1}. As a result, we obtain a collection of a priori different 
bijections between $\mathbb B$ and ${\mathcal A}(\Gamma)$.

\begin{remarks}\label{perfectcase} \mbox{}

\noindent
(1) We believe that the bijections mentioned above are identical. In fact, this is clearly the case if all 
the tensor factors of $\mathbb B$ are perfect crystals. Indeed, since the subgraph of $\mathbb B$ consisting of the 
dual Demazure arrows is connected, there is no more than one isomorphism between it and ${\mathcal A}(\Gamma)$.

\noindent
(2) In the case when all the tensor factors of $\mathbb B$ are perfect crystals, a corollary of the work in \cite{LL2} 
is the following application of the quantum alcove model, cf. Remark \ref{perfectcase} (1). By making specific 
choices for the $\lambda$-chains $\Gamma$ and $\Gamma'$, the bijection between $\A(\Gamma)$ and $\A(\Gamma')$ 
mentioned above gives a uniform realization of the combinatorial $R$-matrix (i.e., the unique affine crystal 
isomorphism commuting factors in a tensor product of KR crystals). In fact, we believe that this statement would hold 
in full generality, rather than just the perfect case.

\noindent
(3) In \cite{LL} we proved Theorem \ref{mainconj} in types $A$ and $C$, for certain $\lambda$-chains 
different from the lex ones, via certain bijections constructed in \cite{Le}. Here we used the realization of the 
corresponding crystal $\mathbb B$ in terms of Kashiwara-Nakashima (KN) columns \cite{KN}.
\end{remarks}

%
\section{The energy function in the quantum alcove model and \texorpdfstring{$P=X$}{P=X}}
\label{S:energy}

We use the notation in Section~\ref{section.qalc}. 
Given the lex $\lambda$-chain $\Gamma=(\beta_1,\ldots,\beta_m)$ 
with height sequence $(l_1,\ldots,l_m)$, we define 
the complementary height sequence $(\widetilde{l}_1,\ldots,\widetilde{l}_m)$ 
by $\widetilde{l}_i:=\langle\beta_i^\vee,\lambda\rangle-l_i$. 
In other words, $\widetilde{l}_i=|\{j\ge i \mid \beta_j=\beta_i\}|$.

\begin{definition} \label{defht}
Given $A=\{j_1<\cdots<j_s\}\in{\mathcal A}(\Gamma)$, we let 
\begin{equation*}
A^{-}:=\bigl\{j_i \in A \mid 
  r_{\beta_{j_1}} \cdots r_{\beta_{j_{i-1}}} > 
  r_{\beta_{j_1}} \cdots r_{\beta_{j_{i-1}}}r_{\beta_{j_{i}}}
\bigr\}
\end{equation*}
{\rm(}in other words, we record the quantum steps 
in the path \eqref{eqn:admissible}{\rm)}. 
We also define
\begin{equation}
\label{defheight}
	\Ht(A):=\sum_{j\in A^-}\widetilde{l}_j\,.\end{equation}
\end{definition}

For examples, we refer to \cite{Le}.

Our goal is to show that the bijection in Section~\ref{S:forget} translates 
the height statistic in the quantum alcove model to 
the energy statistic in the quantum LS path model given 
in Section~\ref{subsec:deg-ene}. For this, 
we need the following lemma, whose proof is 
given in Section~\ref{comppaths}.

\begin{lem} \label{relwt} 
Let $\sigma,\,\tau \in W^J$, and $v\in \sigma W_J$, $w\in\tau W_J$. 
Consider a shortest path $\bp$ from $\sigma$ to $\tau$ in 
$\QB(W^J)$, as well as a shortest path $\bq$ from $v$ to 
$w$ in $\QB(W)$. Then, $\pair{\wt(\bp)}{\lambda}=
\pair{\wt(\bq)}{\lambda}$. 
\end{lem}

We can now state one of our main results.

\begin{thm} \label{energy-transl} 
Consider an admissible subset $A$ in ${\mathcal A}(\lambda)$, and 
the corresponding QLS path $\Pi(A) \in \QLS(-w_{\circ}\lambda)$. 
Write $\Pi(A)$ as{\rm:}
\begin{equation*}
\sigma_0' 
\blarrl{-b_1w_{\circ}\lambda} \sigma_1' 
\blarrl{-b_2w_{\circ}\lambda} \cdots 
\blarrl{-b_pw_{\circ}\lambda} \sigma_p',
\end{equation*}
with $\sigma_i' \in W^{\omega(J)}$ and $0=b_{0} < b_{1} < \cdots < b_{p} < 1$.
Set $\sigma_{i}:=\mcr{\sigma_{i}'w_{\circ}}^{J} \in W^{J}${\rm;} note that 
$\sigma_0 \barrl{b_1\lambda} \sigma_1 \barrl{b_2\lambda} \cdots \barrl{b_p\lambda}\sigma_p$ 
{\rm(}cf. \eqref{qls2map1} and \eqref{qls2map}{\rm)}. Then, we have
\[
	{\rm height}(A)=\sum_{i=1}^p (1-b_i)\wt_\lambda(\sigma_{i-1}\Rightarrow\sigma_i)\,.
\]
\end{thm}
 
\begin{proof}
Recall from Section \ref{S:invmap} that $A$ (in fact, the corresponding path \eqref{eqn:admissible} 
in $\QB(W)$) can be reconstructed from the quantum LS path by first defining recursively a 
sequence $w_i\in \sigma_i W_J$, $i=0,\ldots,p$ (and $w_{-1}=1$), and then by concatenating 
the unique paths $\bq_i$ with increasing edge labels (with respect to $<_\lambda$, cf. 
Section~\ref{S:invmap}) between $w_{i-1}$ and $w_i$, for $i=0,\ldots,p$. By 
Theorem~\ref{thm:shell} (2), the paths $\bq_i$ are shortest ones. 
Therefore, by Lemma~\ref{relwt}, we have 
\begin{equation} \label{zeq1}
	\wt_\lambda(\sigma_{i-1} \Rightarrow \sigma_i)=
	\pair{\wt(\bq_i)}{\lambda} \qquad \text{for $i=1,\ldots,p$}\,;
\end{equation}
we also have $\wt(\bq_0)=0$. 
 
Consider a quantum edge in some path $\bq_i$, and let $\beta_j$ be the root in the lex 
$\lambda$-chain labeling it (so $j\in A^-$). As discussed in Section \ref{S:invmap}, we have 
\begin{equation} \label{zeq2}
	b_i \pair{\beta_j^\vee}{\lambda} =l_j\,,
	\qquad \text{so $(1-b_i) \pair{\beta_j^\vee}{\lambda}=\widetilde{l}_j$}\,.
\end{equation}
By noting that $\wt(\bq_i)$ is the sum of $\beta_j^\vee$ for all such $\beta_j$, and by 
combining \eqref{defheight}, \eqref{zeq1}, and \eqref{zeq2}, the statement of the theorem follows.
 \end{proof}

\begin{cor} \label{dllev} 
Keep the notation of Sections~{\rm \ref{subsec:Lusztig}} 
and {\rm \ref{section.isom}}. 
For each $A \in {\mathcal A}(\lambda)$, we have
\begin{align*}
-\mathrm{height}(A) & 
  = \Deg_{-w_{\circ}\lambda}\bigl( \Pi(A) \bigr) 
  = \Deg_{\lambda} \bigl( \omega(\Pi(A)) \bigr) \\
& = \Deg_{\lambda}\bigl( S (\Pi^{\ast}(A)) \bigr) 
  = D_{\BB^{\rev}}\bigl(S(\ti{\Psi}(A))\bigr) - D^{\ext}_{\BB^{\rev}}. 
\end{align*}
Namely, the following diagram commutes:
\begin{equation}
\begin{diagram}
\node{\BZ_{\le 0}} \arrow[3]{s,l}{\mathrm{id}}
\node{\CA(\lambda)} 
  \arrow{w,t}{-\mathrm{height} \, \eqref{defheight}}
  \arrow{e,t}{\Pi \, \text{\rm (\S\ref{S:forget})} } 
  \arrow{se,t}{\Pi^{\ast}}
  \arrow{s,r}{\ti{\Psi}}
\node{\QLS(-w_{\circ}\lambda)} 
  \arrow{s,r}{\ast \, \text{\rm (\S\ref{subsec:Lusztig})} } 
  \arrow{e,t}{\mathrm{id}} 
\node{\QLS(-w_{\circ}\lambda)} 
  \arrow[2]{s,r}{\omega \, \text{\rm (\S\ref{subsec:Lusztig})}} \\
\node{}
\node{\BB} \arrow{s,r}{S \, \text{\rm (\S\ref{subsec:Lusztig})}} 
\node{\QLS(\lambda)} \arrow{w,b}{\Psi \, \eqref{defpsi}} 
\arrow{s,r}{S \, \text{\rm (\S\ref{subsec:Lusztig})}} \\
\node{}
\node{\BB^{\rev}}
\arrow{sw,b}{
   D_{\BB^{\rev}}(\,\cdot\,)-D_{\BB^{\rev}}^{\ext} \, 
   \text{\rm (\S\ref{subsec:deg-ene})} }
\node{\QLS(\lambda)} 
\arrow{s,r}{\Deg_{\lambda} \, \text{\rm (\S\ref{subsec:dfn-deg})} }
\arrow{w,b}{\Psi^{\rev} \, \text{\rm (\S\ref{subsec:Lusztig})}}
\node{\QLS(\lambda)} 
\arrow{w,b}{\mathrm{id}} \\
\node{\BZ_{\le 0}} \arrow[2]{e,b}{\mathrm{id}} \node{} \node{\BZ_{\le 0}}
\end{diagram}
\end{equation}
\end{cor}

\begin{proof} 
As in Theorem~\ref{energy-transl}, 
write $\Pi(A) \in \QLS(-w_{\circ}\lambda)$ as: 
\begin{equation*}
\sigma_0' 
\blarrl{-b_1w_{\circ}\lambda} \sigma_1' 
\blarrl{-b_2w_{\circ}\lambda} \cdots 
\blarrl{-b_pw_{\circ}\lambda} \sigma_p',
\end{equation*}
with $\sigma_i' \in W^{\omega(J)}$ and 
$0=b_{0} < b_{1} < \cdots < b_{p} < 1$, and set 
$\sigma_{i}:=\mcr{\sigma_{i}'w_{\circ}}^{J} \in W^{J}$. 
It follows from Theorem~\ref{thm:deg-QLS} that 
\begin{equation*}
\Deg_{-w_{\circ}\lambda}(\Pi(A)) = -
 \sum_{i=1}^{p} (1-b_{i}) 
   \wt_{-w_{\circ}\lambda}(\sigma_{i}' \Rightarrow \sigma_{i-1}').
\end{equation*}
Also, we see from Lemma~\ref{lem:da}\,(3) that 
\begin{equation*}
\wt_{-w_{\circ}\lambda}(\sigma_{i}' \Rightarrow \sigma_{i-1}') = 
\wt_{\lambda}(\sigma_{i-1} \Rightarrow \sigma_{i}) \quad 
 \text{for all $1 \le i \le p$}. 
\end{equation*}
Therefore, we obtain
\begin{equation*}
\Deg_{-w_{\circ}\lambda}(\Pi(A)) = -
 \sum_{i=1}^{p} (1-b_{i}) 
   \wt_{\lambda}(\sigma_{i-1} \Rightarrow \sigma_{i}) = 
 -\mathrm{height}(A) \quad \text{by Theorem~\ref{energy-transl}}, 
\end{equation*}
which proves the first equality. For the second equality, observe
\begin{equation*}
\wt_{-w_{\circ}\lambda}(\sigma_{i}' \Rightarrow \sigma_{i-1}') = 
\wt_{\lambda}\bigl(\omega(\sigma_{i}') \Rightarrow \omega(\sigma_{i-1}')\bigr) \quad 
 \text{for all $1 \le i \le p$ by Lemma~\ref{lem:da}\,(3)}. 
\end{equation*}
Using this, we deduce that
\begin{align*}
\Deg_{-w_{\circ}\lambda}(\Pi(A)) 
& = - \sum_{i=1}^{p} (1-b_{i}) 
   \wt_{\lambda}\bigl(\omega(\sigma_{i}') \Rightarrow \omega(\sigma_{i-1}')\bigr) \\
& = \Deg_{\lambda}(\omega(\Pi(A))) \quad 
 \text{by Theorem~\ref{thm:deg-QLS} and \eqref{eq:se2}},
\end{align*}
as desired. Since $S \circ \Pi^{\ast} = \omega \circ \Pi$ by the definitions of these maps, 
the third equality follows. The last equality follows from Corollary~\ref{cor:taile} 
since $\ti{\Psi}=\Psi \circ \Pi^{\ast}$. 
\end{proof}

Based on Theorem \ref{mainconj}, cf. also the discussion following it, we have the strengthening of Corollary \ref{dllev} stated below. 

\begin{thm}\label{conjen}
Corollary {\rm \ref{dllev}} holds for ${\mathcal A}(\Gamma)$, where $\Gamma$ is an arbitrary $\lambda$-chain, 
with $\widetilde{\Psi}$ replaced with one of the isomorphisms in Theorem {\rm \ref{mainconj}}. 
\end{thm}

\begin{remark}
In \cite{LeS} the energy function in types $A$ and $C$ was realized in terms of a statistic in the model based on 
KN columns, which is known as charge. Furthermore, in \cite{Le} it was shown that this statistic is the translation of 
the height statistic via the bijections constructed there (also mentioned in Remark~\ref{perfectcase}~(3)), between 
the corresponding quantum alcove model and models based on KN columns. This should be compared with 
Corollary~\ref{dllev} and Theorem~\ref{conjen}, 
where the constant $D^{\ext}$ is $0$ in these cases.
\end{remark}

The following is due to Ion \cite[Theorem~4.2]{Ion} 
for the dual of an untwisted affine root system.

\begin{lem} \label{L:NSMac=Mac} 
For $\la$ dominant,
\begin{align}\label{E:E=Patt=0}
  P_\la(x;q,0) = E_{w_{\circ}\la}(x;q,0)
\end{align}
where $E_\mu$ is the nonsymmetric Macdonald polynomial \cite{machecke}.
\end{lem}

\begin{proof} 
Applying~\cite[(5.7.8)]{machecke} and its notation, at $t = 0$ we have $\xi_{\mu} \rightarrow 0$ if $\mu$ is not the unique antidominant 
element $w_{\circ} \lambda$ in the finite Weyl group orbit of $\lambda$. Indeed, letting $v(\mu) = r_{i_1} r_{i_2} \cdots r_{i_p}$ 
be a reduced expression of the shortest element $v(\mu)$ in the finite Weyl group such that $v(\mu) \mu$ is antidominant, we obtain
\begin{equation*}
	\xi_{\mu} = \prod_{k=1}^{p} \frac{t q^{-\langle \beta_{k}^{\vee}, \mu \rangle} - t^{-\langle v(\mu) \beta_{k}^{\vee}, \rho\rangle}}
	{q^{-\langle \beta_{k}^{\vee}, \mu \rangle} - t^{-\langle v(\mu) \beta_{k}^{\vee}, \rho\rangle}},
\end{equation*}
where $\beta_{k} := r_{i_p} \cdots r_{i_{k+1}} \alpha_{i_k}$ for $1 \leq k \leq p$; here we note that $\langle \beta_{k}^{\vee}, \mu \rangle > 0$ 
and $\langle v(\mu) \beta_{k}^{\vee}, \rho \rangle < 0$ for all $1 \leq k \leq p$ since the elements $\beta_{k}$, $1 \leq k \leq p$, comprise 
the inversion set for $v(\mu)$. 
\end{proof}

We now recall the specialization of the Ram-Yip formula \cite{RY} 
for the nonsymmetric Macdonald polynomial $E_{w_{\circ}\lambda}(x;q,t)$ at $t=0$, 
which was worked out by Orr-Shimozono \cite{OS}. 
Let us consider a reduced alcove path 
\begin{equation*}
\Gamma:=\bigl(
A_{\circ}=A_{0} \edge{-\beta_1} A_{1} \edge{-\beta_2} \cdots 
 \edge{-\beta_m} A_{m} = A_{\circ}+w_{\circ}\lambda \bigr), 
\end{equation*}
where $(\beta_1,\,\beta_2,\,\dots,\,\beta_m)$ is the corresponding 
$(-w_{\circ}\lambda)$-chain of roots, and let $H_{\beta_i,-l_{i}}$ denote 
the hyperplane separating $A_{i-1}$ and $A_{i}$ for $1 \le i \le m$. 
Then, an admissible subset $A=\bigl\{j_{1} < \cdots < j_{s}\bigr\} \in \CA(\Gamma)$
can be interpreted as a ``folding'' of the alcove path $\Gamma$ along the hyperplanes 
$H_{\beta_i,-l_{i}}$, where $i$ ranges over $A$. 
With this notation, the Orr-Shimozono formula for the specialization 
$E_{w_{\circ}\lambda}(x;q,0)$ can be stated as follows. 
%
%
\begin{prop}[{\cite[Corollary~4.4]{OS}}] \label{prop:OS44}
For $\la$ dominant,
\begin{equation} \label{ryt0}
  E_{w_{\circ}\lambda}(x;q,0)=
  \sum_{A \in \CA(\Gamma)} q^{\Ht(A)}x^{-\wt(A)};
\end{equation}
here, for an admissible subset
$A=\bigl\{j_1 < \cdots < j_s \bigr\} \in \CA(\Gamma)$, 
\begin{align*}
\Ht(A) & = \sum_{j \in A^{-}} 
  \bigl( -l_{j} + \pair{\beta_{j}^{\vee}}{-w_{\circ}\lambda} \bigr), \\[1.5mm]
-\wt(A) & = 
 r_{\beta_{j_1},-l_{j_1}} \cdots r_{\beta_{j_s},-l_{j_s}}(w_{\circ}\lambda).
\end{align*}
\end{prop}

\begin{proof}[Sketch of proof]
Let $t_{w_{\circ}\lambda}=r_{i_{1}} \cdots r_{i_{m}}\pi$ be the reduced expression 
(in the extended affine Weyl group) corresponding to $\Gamma$, where $\pi$ is an 
element of length zero. Rewriting $t_{w_{\circ}\lambda}$ as 
$\pi r_{i_{1}'} \cdots r_{i_{m}'}$, we set
\begin{equation*}
\ha{\beta}_{j}:= r_{i_{m}'} \cdots r_{i_{j+1}'} \alpha_{i_j'}, \quad 
1 \le j \le m,
\end{equation*}
and define a sequence of alcoves
\begin{equation*}
\Gamma':=(A_{0}' \rightarrow A_{1}' \rightarrow \cdots \rightarrow A_{m}'=A_{\circ})
\end{equation*}
in such a way that the $\ha{\beta}_{j}$ corresponds to the hyperplane
separating $A_{j-1}'$ and $A_{j}'$ for $1 \le j \le m$. Then, 
by applying the identity
\begin{equation*}
r_{i_1} \cdots r_{i_j} 
 = t_{w_{\circ}\lambda}\pi^{-1}r_{i_m} \cdots r_{i_{j+1}}
 = t_{w_{\circ}\lambda} r_{i_m'} \cdots r_{i_{j+1}'} \pi^{-1}
\end{equation*}
to $A_{\circ}$, we obtain 
\begin{equation*}
A_{j} = A_{j}' + w_{\circ}\lambda \quad 
 \text{for all $1 \le j \le m$}; 
\end{equation*}
this explains why the exponent of $q$ in 
$E_{w_{\circ}\lambda}(x;q,0)$, as given in \cite[Corollary~4.4]{OS}, 
can be written as desired, Also, it is not hard to 
see that the exponent of $x$ in $E_{w_{\circ}\lambda}(x;q,0)$ 
can be written as $r_{\beta_{j_1},-l_{j_1}} \cdots r_{\beta_{j_s},-l_{j_s}}(w_{\circ}\lambda)$.
\end{proof}
%
%
\begin{thm} \label{P:Macalcove}
For $\la$ dominant,
\begin{equation} \label{ryt0}
  P_\lambda(x;q,0)=
  \sum_{\eta \in \QLS(\lambda)} q^{-\Deg(\eta)} x^{\wt(\eta)} = 
  \sum_{A\in{\mathcal A}(\lambda)} q^{\Ht(A)}x^{\wt(A)}\,.
\end{equation}
\end{thm}

\begin{proof} 
For simplicity of notation, we set $\mu:=-w_{\circ}\lambda=\omega(\lambda)$, 
where $\omega$ is the Dynkin diagram automorphism given by 
$-w_{\circ}\alpha_{j} = \alpha_{\omega(j)}$ for $j \in I$.
By Lemma \ref{L:NSMac=Mac} and Proposition~\ref{prop:OS44}, 
we obtain
\begin{equation*}
P_\la(x;q,0) = 
 \sum_{A \in \A(\mu)} 
 q^{\Ht(A)}x^{-\wt(A)}.
\end{equation*}
Moreover, by Proposition~\ref{proposition.bijqls} applied to $\mu=-w_{\circ}\lambda$ 
and the first equality of Corollary~\ref{dllev}, we deduce that 
\begin{equation*}
\sum_{A \in \A(\mu)} 
 q^{\Ht(A)}x^{-\wt(A)} = 
\sum_{\eta \in \QLS(\lambda)} 
 q^{-\Deg_{\lambda}(\eta)}x^{\wt(\eta)},
\end{equation*}
which proves the first equality of \eqref{ryt0}. 

Now, we have
\begin{align*}
P_\la(x;q,0) & = 
\sum_{\eta \in \QLS(\lambda)} 
 q^{-\Deg_{\lambda}(\eta)}x^{\wt(\eta)} =
\sum_{\eta \in \QLS(\lambda)} 
 q^{-\Deg_{\lambda}(S(\eta))}x^{\wt(S(\eta))} \\
& = \sum_{\eta \in \QLS(\lambda)} 
 q^{-\Deg_{\lambda}(S(\eta))}x^{w_{0}\wt(\eta)} 
  = \sum_{\eta \in \QLS(\lambda)} 
 q^{-\Deg_{\lambda}(S(\eta))}x^{\wt(\eta)};
\end{align*}
the last equality follows from the fact 
that $P_\la(x;q,t)$ is symmetric, i.e., 
invariant under the action of the Weyl group $W$, in the variable $x$. 
Because $\Pi^{\ast}:\A(\lambda) \rightarrow \QLS(\lambda)$ 
is a weight-preserving bijection, it follows that 
\begin{equation*}
\sum_{\eta \in \QLS(\lambda)} 
 q^{-\Deg_{\lambda}(S(\eta))}x^{\wt(\eta)} = 
\sum_{A \in \CA(\lambda)} 
 q^{-\Deg_{\lambda}(S(\Pi^{\ast}(A)))}x^{\wt(A)}.
\end{equation*}
Also, by Corollary~\ref{dllev}, we get
\begin{equation*}
\sum_{A \in \CA(\lambda)} 
 q^{-\Deg_{\lambda}(S(\Pi^{\ast}(A)))}x^{\wt(A)} = 
\sum_{A \in \CA(\lambda)} 
 q^{\Ht(A)}x^{\wt(A)},
\end{equation*}
which proves the second equality of \eqref{ryt0}. 
\end{proof}

\begin{remark} \label{R:anychain}\mbox{}
The formula \eqref{ryt0} holds for any $\lambda$-chain.
\end{remark}

Now define the graded character corresponding to 
the KR crystal $\BB$ (see for example~\cite{HKOTT,HKOTY}) by
\begin{equation} \label{defx}
	X_\lambda(x;q):=\sum_{b \in \BB}q^{D_{\BB}(b)-D^{\ext}_{\BB}}x^{\wt(b)},
\end{equation}
where $\wt(b)$ is the weight of the crystal element $b$.
From Theorems~\ref{thm:ns-deg} and \ref{P:Macalcove}, 
we immediately derive our main result.

\begin{cor} \label{peqx} 
We have
\[
	P_\lambda(x;q^{-1},0)=X_\lambda(x;q)\,.
\]
\end{cor}

\begin{remark} 
In type $A$, the Macdonald polynomial at $t=0$ can be expanded in terms of Schur functions
with Kostka-Foulkes polynomials as the transition matrix~\cite[Chapter III.6]{macsfh}. These in turn can be expressed
as one-dimensional configuration sums $X$~\cite{NY}, which implies the $P=X$ result in type $A$. In all simply-laced types it was 
known by combining the results in \cite{Ion} and \cite{FL1}, which equate a certain affine Demazure 
character with $P$ and $X$, respectively. It was also known in type $C$ by~\cite{Le,LeS}.
\end{remark}

%
\section{Proofs of the lemmas in Sections \ref{S:forget}, \ref{S:invmap}, and \ref{S:energy}}
\label{sec:proofs of lemmas}

%
\subsection{Proof of Lemma {\rm \ref{q2ls}}}\label{sq2ls}

In the proof of this lemma, a dotted (resp., plain) edge represents a quantum (resp., Bruhat) edge
in $\QB(W)$ or $\QB(W^J)$, while a dashed edge can be of both types. Define $\beta\in\Phiafp$ by
\begin{equation}\label{ahbh}
	\beta:=\casetwo{w\gamma}{w\gamma\in\Phi^+}
	{\delta+w\gamma}{w\gamma\in\Phi^-}
\end{equation}
As in the proof of one of the main results in \cite{LNSSS}, namely Theorem 6.5 (more precisely, the 
converse statement), we proceed by induction on the height of $\beta$ (i.e., the sum of the coefficients 
in its expansion in the basis of affine simple roots). The base case, when $\beta$ is an affine simple 
root, is treated in the following lemma.

\begin{lem}
\label{bc16} 
In $\QB_{b\lambda}(W)$ we have an edge $w\xrightarrow{w^{-1}\alpha}r_\alpha w$ for a finite simple root 
$\alpha$ with $w^{-1}\alpha\not\in \Phi_J$ 
(resp. $\begin{diagram}\dgARROWLENGTH=3em\node{w} \arrow{e,t,..}{-w^{-1} \theta}\node{r_\theta w,}\end{diagram}$
where $w^{-1}\theta\not\in \Phi_J$) if and only if in $\QB_{b\lambda}(W^J)$ we have 
$\lfloor w\rfloor\xrightarrow{\lfloor w\rfloor^{-1}\alpha}r_\alpha \lfloor w\rfloor$ 
(resp. 
$\begin{diagram}\dgARROWLENGTH=3em\node{\lfloor w\rfloor} \arrow{e,t,..}{-\lfloor w\rfloor^{-1}\theta}
\node{\lfloor r_\theta w\rfloor}\end{diagram}$). 
\end{lem}

\begin{proof}
Let us first ignore the parameter $b$ (or just assume $b=0$). By the trichotomy of cosets 
in~\cite[Propositions 5.10 and 5.11]{LNSSS}, there is a simple way to test whether we have the 
mentioned edges in $\QB(W)$, namely $w^{-1} \ti{\alpha} \in \Phi^\pm \setminus \Phi_J^\pm$, where 
$\ti{\alpha}$ is the simple root $\alpha$ or $-\theta$, respectively; similarly for the mentioned edges 
in $\QB(W^J)$, with $w$ replaced by $\lfloor w \rfloor$. The proof is completed by noting that
\[
	w^{-1} \ti{\alpha}\in\Phi^\pm\setminus\Phi_J^\pm\;\;\;\Leftrightarrow\;\;\;\lfloor w\rfloor^{-1} \ti{\alpha}\in\Phi^\pm\setminus\Phi_J^\pm\,,
\]
where $\ti{\alpha}$ can be any root, in fact; indeed, writing $w=\lfloor w\rfloor w_J$, we have 
$\lfloor w\rfloor^{-1}=w_J w^{-1}$, and we know that the elements of $W_J$ permute 
$\Phi^\pm\setminus\Phi_J^\pm$. For an arbitrary $b$ (and $\ti{\alpha}$), we observe that 
\[
	b\langle w^{-1}\ti{\alpha}^\vee, \lambda \rangle=b\langle \ti{\alpha}^\vee, w\lambda \rangle
	=b\langle \ti{\alpha}^\vee, \lfloor w\rfloor \lambda \rangle=b\langle \lfloor w\rfloor^{-1}\ti{\alpha}^\vee, \lambda\rangle\,.
\]
\end{proof}

We need the following result from \cite{LNSSS}, which we recall. 

\begin{lem}[{\cite[Lemma 6.10]{LNSSS}}]
\label{lem.covers.PQBG}
Let $w\in W$, and let $\gamma\in\Phi^+\setminus\Phi^+_J$. Define $\beta\in\Phiafp$ as in \eqref{ahbh}.
There exists an affine simple root $\alpha$ (in fact, $\alpha\ne\alpha_0$ if $w\gamma\in\Phi^+$)  
such that $\langle \alpha^\vee, \beta \rangle>0$, 
and we have the edge in $\QB(W^J)$ indicated either in case {\rm (1)} or {\rm (2)} below, where 
$z$ is defined by 
$r_\theta \lfloor w r_\gamma\rfloor 
	= \lfloor r_\theta \lfloor w r_\gamma \rfloor \rfloor z
	= \lfloor r_\theta w r_\gamma \rfloor z$:
\[
 	\text{\rm{(1)}} \;\; \casetwoc{\begin{diagram}\node{\lfloor w\rfloor} \arrow{e,t}{\lfloor w\rfloor^{-1}\alpha}\node{r_\alpha \lfloor w\rfloor}\end{diagram}}{\alpha\ne\alpha_0}
	{\begin{diagram}\dgARROWLENGTH=3em\node{\lfloor w\rfloor} \arrow{e,t,..}{-\lfloor w\rfloor^{-1}\theta}\node{\lfloor r_\theta w \rfloor}\end{diagram}}{\alpha=\alpha_0} \;\;\; 
 	\text{\rm{(2)}} \;\; \casetwo{\begin{diagram}\dgARROWLENGTH=4.5em\node{r_\alpha \lfloor w r_\gamma\rfloor} \arrow{e,t}{-\lfloor w r_\gamma\rfloor^{-1}\alpha}
	\node{\lfloor w r_\gamma\rfloor}\end{diagram}}{\alpha\ne\alpha_0}{\begin{diagram}\dgARROWLENGTH=5em\node{\lfloor r_\theta w r_\gamma\rfloor} \arrow{e,t,..}
	{z\lfloor w r_\gamma\rfloor^{-1}\theta}\node{\lfloor w r_\gamma\rfloor}\end{diagram}}{\alpha=\alpha_0}
\]
\end{lem}

We also need the following lemma.

\begin{lem}
\label{bbchains} 
Consider any one of the diamonds in the parabolic quantum Bruhat graph $\QB(W^J)$ listed in 
{\rm \cite[Lemma 5.14]{LNSSS}}. If one of the two paths 
{\rm(}of length $2${\rm)} is in $\QB_{b\lambda}(W^J)$, for some 
fixed $b$, then the other one is too. 
\end{lem}

\begin{proof}
By \cite[Lemma 5.14]{LNSSS}, we know that, up to sign and left multiplication by elements 
of $W_J$, the pairs of labels on the two paths are $\{w^{-1}\ti{\alpha},\gamma\}$ and 
$\{\gamma,r_\gamma w^{-1}\ti{\alpha}\}$, for some $\gamma \in \Phi^{+} \setminus \Phi_J^+$ and 
$w\in W^J$, while $\ti{\alpha}$ is a finite simple root or $-\theta$. The equivalence of the 
integrality conditions with respect to $b$ for the two pairs follows from the simple 
calculation
\[
	b\langle r_\gamma w^{-1}\ti{\alpha}^\vee, \lambda \rangle
	=b \pair{w^{-1}\ti{\alpha}^\vee+l\gamma^\vee}{\lambda}=
	 b \pair{w^{-1}\ti{\alpha}^\vee}{\lambda} + 
	 l \left (b\langle \gamma^\vee, \lambda \rangle\right)\,,
\]
where $l$ is an integer. On another hand, 
it is clear that mapping roots via elements of 
$W_J$ preserves the integrality condition. 
\end{proof}

\begin{proof}[Proof of Lemma {\rm \ref{q2ls}}] 
We can assume that $\gamma\not\in\Phi_J$, as otherwise the statement is obvious. 
As stated above, we proceed by induction on the height of the affine root $\beta$. 
If $\beta$ is an affine simple root, the conclusion follows directly from Lemma~\ref{bc16}. 
Otherwise, we apply Lemma~\ref{lem.covers.PQBG} for $\QB(W^J)$; 
this gives an affine simple root $\alpha$ satisfying
\begin{equation}\label{conda}
	\alpha\ne\beta\,, \qquad  \langle \alpha^\vee, \beta \rangle>0\,,
\end{equation}
and either condition (1) or (2) in the mentioned 
lemma. Assume that condition (1) holds, as the reasoning is completely similar if condition (2) holds. 
By Lemma~\ref{bc16}, we have 
\begin{align}
\label{rte}
&w\xrightarrow{w^{-1}\alpha}r_\alpha w \;\;\;\;\;\;\mbox{if $\alpha\ne\alpha_0$, where $w^{-1}\alpha\not\in \Phi_J$, and}\\ 
&\!\!\!\!\begin{diagram}\dgARROWLENGTH=3em\node{w} \arrow{e,t,..}{-w^{-1} \theta}\node{r_\theta w}\end{diagram} 
\;\;\;\;\mbox{if $\alpha=\alpha_0$, where $w^{-1}\theta\not\in \Phi_J$}\,.\nonumber
\end{align}

By Lemma \ref{lem.covers.PQBG}, we have one of the following three cases:
%
%
\begin{equation} \label{c31h}
	(\beta\in\Phi^+,\alpha\ne\alpha_0)\,,\qquad (\beta\in\delta-\Phi^+,\alpha\ne\alpha_0)\,,\qquad 
	(\beta\in\delta-\Phi^+,\alpha=\alpha_0)\,.
\end{equation}
By \cite[Lemma 5.14]{LNSSS}, known as the ``diamond lemma'', we have the left diamonds in 
\cite[Eqs.\,(5.3), (5.4), and (5.7)]{LNSSS}, respectively.
Note that all the necessary conditions for applying the diamond lemma are checked 
as in the proof of the converse statement of~\cite[Theorem 6.5]{LNSSS}. 
We can represent the diamonds in the three cases \eqref{c31h} using the single diagram below, 
where $\ti{\alpha}:=\alpha$ if $\alpha\ne\alpha_0$, and $\ti{\alpha}:=-\theta$, otherwise. 
\begin{equation}\label{qd1}
 \begin{diagram}
  \node{r_{\ti{\alpha}} w} \arrow{e,t,--}{\gamma} \node{r_{\ti{\alpha}} w r_\gamma} \\
  \node{w} \arrow{e,t,--}{\gamma} \arrow{n,l,--}{|w^{-1}\ti{\alpha}|} 
  \node{w r_\gamma,} \arrow{n,r,--}{|r_\gamma w^{-1}\ti{\alpha}|}
 \end{diagram}
\end{equation}
where $|\xi|:=\xi$ (resp., $|\xi|:=-\xi$) for $\xi \in \Phi^{+}$ 
(resp., $\xi \in \Phi^{-}$).
Recall that the bottom edge is viewed as a path $\bq$; similarly, we view the top edge as a path 
$\bq'$, and we clearly have $\wt(\bq)=\wt(\bq')$ if $\ti{\alpha} \ne -\theta$. 

Define $\beta'$ for the top edge of the diamond \eqref{qd1} in the same way as $\beta$ was 
defined for the bottom one in \eqref{ahbh}. As in the proof of the converse statement 
of~\cite[Theorem 6.5]{LNSSS}, we can check in all three cases \eqref{c31h} that 
$\beta'=r_\alpha\beta$. Since $\langle \alpha^\vee, \beta \rangle>0$, this implies that the 
height of $\beta'$ is strictly smaller than that of $\beta$. Therefore, by applying the induction 
hypothesis to the top edge of the diamond \eqref{qd1}, which is clearly in $\QB_{b\lambda}(W)$, 
we obtain a path in $\QB_{b\lambda}(W^J)$:
\begin{equation}
\label{qd2}
	\bp'\,:\;\;\;\begin{diagram}\dgARROWLENGTH=2em\node{\lfloor r_{\ti{\alpha}} w\rfloor=\!\!\!\!\!\!\!\!\!\!\!\!\!\!\!\!}\node{y_0} 
	\arrow{r,--} \node{y_1}\arrow{r,--}\node{\cdots}\arrow{r,--}\node{y_n}\node{\!\!\!\!\!\!\!\!\!=\lfloor r_{\ti{\alpha}} w r_\gamma\rfloor\,.}\end{diagram}
\end{equation}
By induction, we have $\wt(\bp')\equiv\wt(\bq') \mod Q_J^\vee$.

{\bf Case 1.} Assume for the moment that $r_\gamma w^{-1}\ti{\alpha}\not\in\Phi_J$. 
By Lemma \ref{bc16}, the right edge in \eqref{qd1} implies that in $\QB(W^J)$ we have 
an edge 
\begin{equation}
\label{qd3}
\begin{diagram}\dgARROWLENGTH=2em\node{\lfloor w r_\gamma\rfloor}\arrow{e,--}\node{\lfloor r_{\ti{\alpha}} w r_\gamma\rfloor}\end{diagram}.
\end{equation}
Assuming that the diamond lemma can be successively applied based on \eqref{qd2} and \eqref{qd3}, we 
exhibit the diamonds as in~\eqref{alld} below in $\QB(W^J)$, from right to left, where $y_i=\lfloor r_{\ti{\alpha}} x_i\rfloor$.
\begin{equation}
\label{alld}
\begin{diagram}\dgARROWLENGTH=2em\node{\lfloor r_{\ti{\alpha}} w\rfloor=\!\!\!\!\!\!\!\!\!\!\!\!\!\!\!\!}\node{y_0} \arrow{r,--} 
\node{y_1}\arrow{r,--}\node{\cdots}\arrow{r,--}\node{y_n}\node{\!\!\!\!\!\!\!\!\!=\lfloor r_{\ti{\alpha}} w r_\gamma\rfloor}\\
\node{\lfloor w\rfloor=\!\!\!\!\!\!\!\!\!\!\!\!\!\!\!\!\!\!\!\!}\node{x_0} \arrow{r,--} \arrow{n,--}
\node{x_1}\arrow{r,--}\arrow{n,--}\node{\cdots}\arrow{r,--}\node{x_n}\arrow{n,--}\node{\!\!\!\!\!\!\!\!\!\!\!\!\!=\lfloor w r_\gamma\rfloor}
\end{diagram}
\end{equation}
Note that the labels on the top edges are the same as those on the corresponding edges on the bottom, 
or at most differ from those by elements of $W_J$; so all the bottom edges are in $\QB_{b\lambda}(W^J)$ too, and 
we can define $\bp$ to be the path formed by them.

Now let us prove that $\wt(\bp)\equiv\wt(\bq) \mod Q_J^\vee$. By \cite[Lemma 5.14]{LNSSS}, the 
weights of all paths from $\lfloor w\rfloor$ to $\lfloor r_{\ti{\alpha}} w r_\gamma\rfloor$ in \eqref{alld} are 
congruent mod $Q_J^\vee$. If $\ti{\alpha} \ne -\theta$, 
then all the vertical edges in \eqref{alld} are Bruhat edges, 
so $\wt(\bp)\equiv\wt(\bp') \mod Q_J^\vee$. Applying the induction hypothesis and the fact 
that $\wt(\bq)=\wt(\bq')$ concludes the induction step in this case. 
If $\ti{\alpha}=-\theta$, then we are in the third case in \eqref{c31h}, 
and so diagram \eqref{qd1} is the left one in \cite[Eq. (5.7)]{LNSSS}; 
thus its top edge is a Bruhat edge, which implies $\wt(\bp')=\wt(\bq')=0$, and its bottom edge 
is a quantum one, in particular $w\gamma\in\Phi^-$. Moreover, all the vertical edges 
in~\eqref{alld} are quantum ones; in particular, the leftmost and the rightmost ones have weights
\begin{equation*}
-\mcr{w}^{-1} \theta^\vee \equiv -w^{-1}\theta^\vee \mod Q_J^\vee, \quad \text{and} \quad
-\mcr{wr_{\gamma}}^{-1}\theta^\vee \equiv - r_\gamma w^{-1}\theta^\vee \mod Q_J^\vee, 
\end{equation*}
respectively. Then, by the above observation about the paths from $\lfloor w\rfloor$ to 
$\lfloor r_{\ti{\alpha}} w r_\gamma\rfloor$ in \eqref{alld}, we have mod $Q_J^\vee$:
\[
\wt(\bp)\equiv\wt(\bp')-w^{-1}\theta^\vee+r_\gamma w^{-1}\theta^\vee
=-\langle w^{-1}\theta^\vee, \gamma \rangle\gamma^\vee
=\langle \underbrace{\theta^\vee, -w\gamma \rangle}_{=1}\gamma^\vee=\gamma^\vee=\wt(\bq)\,.
\]
Here we see that $\langle \theta^\vee, -w\gamma \rangle=1$ by using the well-known fact that if 
$\phi\ne\theta$ is a positive root, then $\langle \theta^\vee, \phi\rangle$ is $0$ or $1$; indeed, in 
our case we saw that $-w\gamma\in\Phi^+$, while we have $-w\gamma\ne\theta$ and 
$\langle \theta^\vee, -w\gamma \rangle\ne 0$ by \eqref{conda}. 

{\bf Case 2.} The reasoning in Case 1 fails, i.e., we cannot apply the diamond lemma at some 
point, if we have the following situation for some $i\le n$.
\[
\begin{diagram}\dgARROWLENGTH=2em\node{y_{i-1}} \arrow{r,--}\arrow{se,=} 
\node{y_i}\arrow{r,--}\node{y_{i+1}}\arrow{r,--}\node{\cdots}\arrow{r,--}\node{y_n}
\node{\!\!\!\!\!\!\!\!\!\!\!\!\!=\lfloor r_{\ti{\alpha}} w r_\gamma\rfloor}\\
\node{\;\;}\node{x_i} \arrow{r,--} \arrow{n,--}
\node{x_{i+1}}\arrow{r,--}\arrow{n,--}\node{\cdots}\arrow{r,--}\node{x_n}\arrow{n,--}
\node{\!\!\!\!\!\!\!\!\!\!\!\!\!\!\!\!\!=\lfloor w r_\gamma\rfloor}
\end{diagram}
\]
In this case, the edge 
$\begin{diagram}
 \dgARROWLENGTH=2em \node{x_i} \arrow{e,--} \node{y_i} 
 \end{diagram}$
is in $\QB_{b\lambda}(W^J)$, since it coincides with the edge 
$\begin{diagram} 
 \dgARROWLENGTH=2em \node{y_{i-1}} \arrow{e,--} \node{y_i,} \end{diagram}$
which has this property by the induction hypothesis, cf.\,\eqref{qd2}. 
By Lemma \ref{bbchains}, all vertical edges are also in $\QB_{b\lambda}(W^J)$, 
in particular the rightmost one. By Lemma \ref{bc16}, the edge 
$\begin{diagram}
 \dgARROWLENGTH=2em
 \node{w r_\gamma}\arrow{e,--}\node{r_{\ti{\alpha}} w r_\gamma}
 \end{diagram}$
in \eqref{qd1} is in $\QB_{b\lambda}(W)$. 
By applying Lemma \ref{bbchains} to \eqref{qd1} this time, 
we conclude that the edge 
$\begin{diagram}
 \dgARROWLENGTH=2em 
 \node{w} \arrow{e,--} \node{r_{\ti{\alpha}} w} 
 \end{diagram}$
is in $\QB_{b\lambda}(W)$. But we showed in \eqref{rte} that 
$w^{-1}\ti{\alpha}\not\in \Phi_J$, so by Lemma \ref{bc16} the edge 
$\begin{diagram}
 \dgARROWLENGTH=2em
 \node{\mcr{w}} \arrow{e,--} 
 \node{\mcr{r_{\ti{\alpha}} w}}
 \end{diagram}$ 
is in $\QB_{b\lambda}(W^J)$. We now define $\bp$ to 
be the following path in $\QB_{b\lambda}(W^J)$:
\[
\bp:\!\!\!\!\!\!\!\!\!\!\!\!\!\!\!\begin{diagram}\dgARROWLENGTH=2em
\node{x_0=\!\!\!\!\!\!\!\!\!\!\!\!\!\!\!\!\!\!\!\!\!\!\!\!\!\!\!\!\!\!\!\!\!\!\!\!\!\!\!\!\!\!\!\!\!\!\!\!\!\!}\node{\lfloor w\rfloor} 
\arrow{r,--} \node{\lfloor r_{\ti{\alpha}} w\rfloor=y_0}\arrow{r,--}\node{\cdots}\arrow{r,--}\node{y_{i-1}=x_i}
\arrow{r,--}\node{\cdots}\arrow{r,--}\node{x_n=\lfloor w r_\gamma\rfloor}\end{diagram}.
\]
We then prove that $\wt(\bp)\equiv\wt(\bq) \mod Q_J^\vee$ in a way completely similar
to Case 1, which concludes the induction step.

{\bf Case 3.} The last case to consider is the one when $r_\gamma w^{-1}\ti{\alpha}\in\Phi_J$. We still have the edge 
\[
\begin{diagram}\dgARROWLENGTH=4.3em\node{w r_\gamma}\arrow{e,t,--}{|r_\gamma w^{-1}\ti{\alpha}|}
\node{r_{\ti{\alpha}} w r_\gamma}\end{diagram}
\]
in $\QB_{b\lambda}(W)$, because $\langle r_\gamma w^{-1}\ti{\alpha}, \lambda \rangle=0$. So we can reason 
as in the previous paragraph in order to prove that the edge 
$\begin{diagram}\dgARROWLENGTH=2em\node{\lfloor w\rfloor}\arrow{e,--}\node{\lfloor r_{\ti{\alpha}} w\rfloor}\end{diagram}$ 
is in $\QB_{b\lambda}(W^J)$. We now define $\bp$ to be the following path in $\QB_{b\lambda}(W^J)$:
\[\bp:\;\begin{diagram}\dgARROWLENGTH=1.9em\node{\lfloor w\rfloor} \arrow{r,--} 
\node{\lfloor r_{\ti{\alpha}} w\rfloor=\!\!\!\!\!\!\!\!\!\!\!\!\!\!\!\!}\node{y_0}\arrow{r,--}\node{y_1}\arrow{r,--}
\node{\cdots}\arrow{r,--}\node{y_n}\node{\!\!=x_n=\lfloor w r_\gamma\rfloor}\end{diagram}.\]
Note that this is the only case when the induction step produces a path of a different length 
(more precisely, longer by 1) based on the path in the induction hypothesis.

Now let us prove that $\wt(\bp)\equiv\wt(\bq) \mod Q_J^\vee$. 
If $\ti{\alpha} \ne -\theta$, then the first edge of 
$\bp$ is a Bruhat edge, so $\wt(\bp)=\wt(\bp')$. Applying the induction hypothesis 
and the fact that $\wt(\bq)=\wt(\bq')$ concludes the induction step in this case. 
If $\ti{\alpha} = -\theta$, then by the same reasoning as in Case 1, we deduce 
\begin{equation*}
\wt(\bp')=\wt(\bq')=0, \qquad 
\wt(\bq)=\gamma^\vee, \qquad 
w\gamma\in\Phi^-.
\end{equation*}
We conclude that $\wt(\bp)\equiv-w^{-1}\theta^\vee \mod Q_J^\vee$ (cf. Case 1), so we 
need to prove that $-w^{-1}\theta^\vee\equiv\gamma^\vee \mod Q_J^\vee$. This follows from
\[
\Phi_J^\vee\ni r_{\gamma}w^{-1}\theta^{\vee}=w^{-1}\theta^{\vee}-
 \underbrace{\langle w^{-1}\theta^{\vee}, \gamma \rangle}_{=-1}\gamma^{\vee} 
 = w^{-1}\theta^{\vee}+\gamma^{\vee}. 
\]
Here we have used the fact that $\langle w^{-1}\theta^{\vee}, \gamma \rangle=-1$, 
as in Case 1.
\end{proof}

%
\subsection{Proof of Lemma~{\rm \ref{bqbpaths}}}
\label{sbqbpaths}

We require some notation and results from \cite{LS}. 
Let $W_\af^-$ denote the set of minimum coset representatives in $W_\af/W$. 
Write $y \ulc{\af} x$ for the covering relation 
in the (strong) Bruhat order on $W_\af$. 
For $M \in\BZ_{>0}$, say that $\xi \in Q^\vee$ is $M$-superantidominant 
if $\pair{\xi}{\alpha} \le -M$ for every positive root $\alpha \in \Phi^{+}$. 
We fix once and for all a sufficiently large $M\in\BZ_{>0}$ ($M =2|W |+2 $ is sufficient).

\begin{lem}[{\cite[Lemma~3.3]{LS}}] \label{L:affgrass} 
Let $w\in W$ and $\xi \in Q^\vee$. Then
$wt_\xi \in W_\af^-$ if and only if $\xi$ is antidominant
(that is, $\pair{\xi}{\alpha_i} \le 0$ for all $i\in I$) and $w \in W^{L}$, 
where $L:=\bigl\{i \in I \mid \pair{\xi}{\alpha_{i}}=0\bigr\}$.
\end{lem}

\begin{prop}[{\cite[Proposition~4.4]{LS}}] \label{P:affinelift} 
Let $\xi \in Q^\vee$ be $M$-superantidominant
and let $x=w t_{v\xi}$ with $v,w\in W$. Then 
$y=xr_{v\alpha+n\delta} \ulc{\af} x$ 
if and only if one of the following conditions holds:
\begin{enumerate}
\item[(i)] $\ell(wv)=\ell(wvr_\alpha)-1$ and $n=\pair{\xi}{\alpha}$, giving $y=w r_{v\alpha}t_{v\xi}$;
\item[(ii)] $\ell(wv)=\ell(wvr_\alpha)+\pair{\alpha^\vee}{2\rho}-1$ and $n=\pair{\xi}{\alpha}+1$, giving 
$y=w r_{v\alpha}t_{v(\xi+\alpha^\vee)}$;
\item[(iii)] $\ell(v)=\ell(vr_\alpha)+1$ and $n=0$, giving $y=wr_{v\alpha} t_{v r_\alpha \xi}$;
\item[(iv)] $\ell(v)=\ell(vr_\alpha)-\pair{\alpha^\vee}{2\rho}+1$ and $n=-1$, 
giving $y=wr_{v\alpha} t_{v r_\alpha (\xi+\alpha^\vee)}$.
\end{enumerate}
\end{prop}

We start with the following lemma.
We need the $b$-Bruhat order on $\Waf$, denoted $\ul{\af,b}$, 
which is defined by a condition completely similar to \eqref{bbord} 
applied to the covers in $\Waf$.

\begin{lem} \label{twoch} 
Assume that in $\Waf^{-}$ we have
\begin{equation*}
vt_\xi \ug{\af} w t_h, \qquad 
vt_\xi \ug{\af,b} w t_{h'}, \qquad
\text{where} \quad h'-h \in Q^{\vee +},
\end{equation*}
and $\xi,h,h'\in Q^\vee$ are $M$-superantidominant. 
Then $vt_\xi \ug{\af,b} w t_h$, and in 
fact any saturated chain between these elements is a chain in $b$-Bruhat order.
\end{lem}

\begin{proof}
We claim that $wt_h \ug{\af} wt_{h'}$ using a downward chain in $W_\af^-$. 
It suffices to prove this when $h'-h=\alpha_i^\vee$ for some $i\in I$. 
Suppose this is the case. Suppose first that $wr_i \ul{\af} w$. 
By Proposition \ref{P:affinelift} we have 
$w t_h \ugc{\af} 
  w r_i t_{h+\alpha_i^\vee} 
  \ugc{\af} w t_{h+\alpha_i^\vee}$ as required. 
Otherwise we have $wr_i \ug{\af} w$. 
Then by Proposition \ref{P:affinelift} we have 
$wt_h \ugc{\af} 
  wr_i t_h \ugc{\af} 
  w t_{h+\alpha_i^\vee}$ as required.

Knowing this, using Proposition \ref{P:affinelift} we pick a downward saturated
 chain from $vt_\xi$ to $w t_h$, followed by one from $w t_h$ to $w t_{h'}$, all in $W_\af^-$.
 By the hypothesis, there is a downward saturated chain in $b$-Bruhat order 
 from $vt_\xi$ to $wt_{h'}$. By \cite[Lemma 4.15]{LeSh}, we know that the 
 first chain is in $b$-Bruhat order too, which concludes the proof.
\end{proof}

\begin{proof}[Proof of Lemma {\rm \ref{bqbpaths}}] 
By Proposition \ref{P:affinelift} we can lift both paths to downward saturated chains 
in $\Waf^-$ starting at $v t_\xi$, where $\xi$ is a fixed $M$-superantidominant element in $Q^{\vee}$. 
Denote the endpoints of the two chains by $w t_{\xi+h}$ and $w t_{\xi+h'}$, 
respectively. Recall that $h$ and $h'$ are the sums of the coroots corresponding 
to (the labels of) the quantum edges in the paths in $\QB(W)$ which are lifted. Since the first 
path in $\QB(W)$ is a shortest one, by \cite[Lemma 1]{Po}, we have $h'-h\in Q^{\vee +}$. 
Furthermore, by the hypothesis, the second chain in $\Waf^-$ is in $b$-Bruhat order. Thus the 
hypotheses of Lemma \ref{twoch} are all satisfied, so we conclude that the first chain in 
$\Waf^-$ is also in $b$-Bruhat order, and therefore the first path in $\QB(W)$ is in $\QB_{b\lambda}(W)$.
\end{proof}

%
\subsection{Proof of Lemma {\rm \ref{relwt}}}
\label{comppaths} 

Let us first recall Proposition~\ref{prop:weight},
which is the parabolic generalization of a lemma due to Postnikov~\cite{Po}.

\begin{proof}[Proof of Lemma {\rm \ref{relwt}}] 
By Lemma \ref{q2ls}, we can construct a path from $\bp'$ from $\sigma$ to $\tau$ in $\QB(W^J)$ with 
\begin{equation}\label{xeq0}\wt(\bp')\equiv\wt(\bq) \mod Q_J^\vee\,;\end{equation}
namely, we simply concatenate the paths in $\QB(W^J)$ that correspond, by the 
mentioned lemma, to each edge of $\bq$, cf. the construction of the forgetful map in 
Section \ref{S:forget}. By Proposition~\ref{prop:weight}, we have 
\begin{equation}
\label{xeq1}
	\langle \wt(\bp'), \lambda \rangle\ge \langle \wt(\bp), \lambda \rangle\,.
\end{equation}

We then exhibit a path $\bq'$ from $v$ to $w$ as in the proof of Lemma \ref{L:mainb} 
(on which the construction of the inverse map in Section \ref{S:invmap} is based); we 
refer to this proof for the details. Namely, we concatenate the following:
\begin{itemize}
\item a path from $v$ to $\sigma$ with only quantum edges and all edge labels in $\Phi_J^+$;
\item a path from $\sigma$ to $\tau$ constructed based on $\bp$;
\item a path from $\tau$ to $w$ with only Bruhat edges and all edge labels in $\Phi_J^+$.
\end{itemize}
Note that 
\begin{equation}
\label{xeq2}
\langle \wt(\bq'),\lambda \rangle= \langle \wt(\bp), \lambda \rangle\,,
\end{equation}
since all the edges in the first segment of $\bq'$, as well as the extra edges 
introduced in the second segment, have labels orthogonal to $\lambda$. On 
another hand, by Proposition~\ref{prop:weight} (in fact, we only need here the original 
version \cite[Lemma 1\,(3)]{Po}), we have 
\begin{equation}
\label{xeq3}
	\langle \wt(\bq'), \lambda \rangle\ge \langle \wt(\bq), \lambda \rangle\,.
\end{equation}
The proof is concluded by combining \eqref{xeq0}, \eqref{xeq1}, \eqref{xeq2}, and \eqref{xeq3}.
\end{proof}

\appendix

%
\section{Perfectness and classical decomposition}
\label{section.perfectness} 

The notion of perfectness plays an important role for level-zero crystals. It ensures
for example that the Kyoto path model is applicable, which gives a model for highest weight
affine crystals as a semi-infinite tensor product of Kirillov--Reshetikhin crystals.
Let us define perfect crystals, see for example~\cite{hkqgcb}. Given a crystal $\mathcal{B}$ and $b\in \mathcal{B}$,
we need the definition
\[
	\varepsilon(b) = \sum_{i\in I_{\af}} \varepsilon_i(b) \Lambda_i \quad \text{and} \quad
	\varphi(b) = \sum_{i\in I_{\af}} \varphi_i(b) \Lambda_i
\]
with $\varepsilon_i(b)$ and $\varphi_i(b)$ as defined in~\eqref{eq:phi eps}. Furthermore, denote by
$\overline{X}_{\af}^{+\ell} = \{\lambda \in \overline{X}_\af^+ \mid \lev(\lambda) = \ell\}$ the set of dominant weights of level $\ell$, 
where $\overline{X}_\af^+ :=\bigoplus_{i\in I_{\af}} \mathbb{Z}_{\ge 0} \Lambda_i$.

\begin{definition} \label{def:perfect}
For a positive integer $\ell > 0$, a crystal $\mathcal{B}$ is called a perfect crystal of level $\ell$, 
if the following conditions are satisfied:
\begin{enumerate}
\item $\mathcal{B}$ is isomorphic to the crystal graph of a finite-dimensional $U_q^\prime(\mathfrak{g}_{\af})$-module.
\item $\mathcal{B}\otimes \mathcal{B}$ is connected.
\item There exists a $\lambda\in \bigoplus_{i\in I_{\af}} \mathbb{Z} \Lambda_i $, such that 
$\wt(\mathcal{B}) \subset \lambda + \sum_{i\in I} \mathbb{Z}_{\le 0} \alpha_i$ 
and there is a unique element in $\mathcal{B}$ of classical weight $\la$.
\item $\forall \; b \in \mathcal{B}, \;\; \lev(\varepsilon (b)) \geq \ell$.
\item $\forall \; \Lambda \in \overline{X}_{\af}^{+\ell}$, there exist unique elements $b_{\Lambda}, b^{\Lambda} \in \mathcal{B}$, 
such that
$$ \ve ( b_{\Lambda}) = \Lambda = \vp( b^{\Lambda}). $$
\end{enumerate}
\end{definition}

We denote by $\mathcal{B}_{\min}$ the set of minimal elements in $\mathcal{B}$, namely
\begin{equation*}	
	\mathcal{B}_{\min} = \{ b\in \mathcal{B} \mid \lev(\ve(b))=\ell \}.
\end{equation*}
Note that condition (5) of Definition~\ref{def:perfect} ensures that $\ve,\vp:\mathcal{B}_{\min}\to \overline{X}_{\af}^{+\ell}$ are bijections. 

Recall from Section~\ref{subsec:notation} that $\delta=\sum_{j \in I_{\af}} a_{j}\alpha_{j} \in \Fh_\af^{\ast}$ and 
$c=\sum_{j \in I_{\af}} a^{\vee}_{j} \alpha_{j}^{\vee} \in \Fh_\af$. Define 
$c_r = \max\{ \frac{a_r}{a_r^\vee}, a_0^\vee\}$.

\begin{conj} \cite[Conjecture 2.1]{HKOTT} \label{conj:perfectness}
The Kirillov-Reshetikhin crystal $B^{r,s}$ is perfect if and only if $\frac{s}{c_r}$ is an integer. 
If $B^{r,s}$ is perfect, its level is $\frac{s}{c_r}$.\end{conj}

For all nonexceptional types this conjecture was proven in~\cite{FOS-perfect}. Given the explicit models for
$B^{r,1}$ for all untwisted types in this paper and their implementation into {\sc Sage}~\cite{sage,sagecombinat},
we have verified Conjecture~\ref{conj:perfectness} also for untwisted exceptional types when $s=1$.
For type $G_2^{(1)}$, perfectness was also treated in~\cite{Y}.

\begin{thm}
Conjecture{\rm ~\ref{conj:perfectness}} holds for $B^{r,1}$ for types $G_2^{(1)}, F_4^{(1)}, E_6^{(1)}, E_7^{(1)}$
for all Dynkin nodes, and type $E_8^{(1)}$ for all nodes (except possibly $5,8$ in the labeling of~\cite{HKOTT}).
In addition, the graded classical decompositions of~\cite[Appendix A]{HKOTY} were verified (except for type $E_8^{(1)}$).
\end{thm}

For the other nodes in type $E_8^{(1)}$ the program is currently too slow to test it.

\begin{proof}
Point (1) of Definition~\ref{def:perfect} 
follows from Remark~\ref{rem:LScl}. 
Point (2) can be deduced from~\cite{Kas-OnL}. 
Points (3)-(5) were checked explicitly on the computer 
using the implementation of level-zero LS paths 
in {\sc Sage}~\cite{sage,sagecombinat} 
(version sage-7.1 or higher), see for example
\begin{verbatim}
	sage: C = CartanType(['E',6,1])
	sage: R = RootSystem(C)
	sage: La = R.weight_space().basis()
	sage: LS = crystals.ProjectedLevelZeroLSPaths(La[1])
	sage: LS.is_perfect()
	True
\end{verbatim}
This showed that $B^{r,1}$ is perfect
\begin{itemize}
\item for all nodes of type $E_{6,7}^{(1)}$ and the nodes specified in the theorem for type $E_8^{(1)}$;
\item the first 2 nodes of $F_4^{(1)}$ (long roots);
\item the second node of $G_2^{(1)}$ (long root).
\end{itemize}
This confirms the perfectness claim of the theorem. The graded classical decompositions of~\cite[Appendix A]{HKOTY} 
were also confirmed by computer.
\end{proof}



\begin{thebibliography}{Sage-comb}

\bibitem[BB]{BB}
\newblock A. Bj\"{o}rner and F. Brenti.
\newblock {\em Combinatorics of {C}oxeter groups}.
\newblock Graduate Texts in Mathematics Vol.~231, Springer, New York, 2005.

\bibitem[BF1]{BF1} 
\newblock A. Braverman and M. Finkelberg.
\newblock Semi-infinite Schubert varieties and quantum $K$-theory of flag manifolds.
\newblock {\em J. Amer. Math. Soc.}, 27:1147-1168, 2014.

\bibitem[BF2]{BF2} 
\newblock A. Braverman and M. Finkelberg.
\newblock Weyl modules and $q$-Whittaker functions.
\newblock {\em Math. Ann.}, 359:45-59, 2014. 

\bibitem[BFP]{BFP}
\newblock F.~Brenti, S.~Fomin, and A.~Postnikov.
\newblock Mixed Bruhat operators and Yang-Baxter equations for Weyl groups.
\newblock {\em Int. Math. Res. Notices}, no. 8, 419--441, 1999.

\bibitem[CI]{CI}
\newblock V.~Chari and B.~Ion.
\newblock BGG reciprocity for current algebras.
\newblock  {\em Compos. Math.}, 151:1265--1287, 2015.

\bibitem[CSSW]{CSSW}
\newblock V.~Chari, L.~Schneider, P.~Shereen, and J.~Wand.
\newblock Modules with Demazure flags and character formulae.
\newblock {\em SIGMA} 10, 032, 16pp, 2014.

\bibitem[Dy]{Dyer}
\newblock M. J. Dyer.
\newblock Hecke algebras and shellings of {B}ruhat intervals.
\newblock {\em Compositio Math.}, 89:91--115, 1993.

\bibitem[E]{E} 
\newblock P. Etingof.
\newblock Whittaker functions on quantum groups and $q$-deformed Toda operators. 
\newblock Differential topology, infinite-dimensional Lie algebras, and applications, pp. 9--25, 
Amer. Math. Soc. Transl. Ser. 2, 194, Amer. Math. Soc., Providence, RI, 1999. 

\bibitem[FGP]{FGP}
\newblock S.~Fomin, S.~Gelfand, and A.~Postnikov.
\newblock Quantum Schubert polynomials.
\newblock {\em J. Amer. Math. Soc.}, 10:565--596, 1997. 

\bibitem[FL]{FL1}
\newblock G.~Fourier and P.~Littelmann.
\newblock Tensor product structure of affine Demazure modules and limit constructions.
\newblock {\em Nagoya Math. J.}, 182:171--198, 2006.

\bibitem[FOS]{FOS-combinatorics}
\newblock G.~Fourier, M.~Okado, and A.~Schilling.
\newblock Kirillov-Reshetikhin crystals for nonexceptional types.
\newblock {\em Adv. Math.}, 222:1080--1116, 2009.

\bibitem[FOS1]{FOS-perfect}
\newblock G.~Fourier, M.~Okado, and A.~Schilling.
\newblock Perfectness of Kirillov-Reshetikhin crystals for nonexceptional types.
\newblock {\em Contemp. Math.}, 506:127--143, 2010.

\bibitem[FW]{FW}
\newblock W. Fulton and C.~Woodward.
\newblock On the quantum product of Schubert classes.
\newblock \textit{J. Algebraic Geom.}, 13:641--661, 2004.

\bibitem[HKOTT]{HKOTT}
\newblock G. Hatayama, A. Kuniba, M. Okado, T. Takagi, and Z. Tsuboi.
\newblock Paths, crystals and fermionic formulae.
\newblock In {\em MathPhys Odyssey 2001. Integrable Models and Beyond}
(M. Kashiwara and T. Miwa, Eds.),
Prog. Math. Phys. Vol.~23, pp. 205--272,
Birkh\"{a}user, Boston, 2002.

\bibitem[HKOTY]{HKOTY}
\newblock G. Hatayama, A. Kuniba, M. Okado, T. Takagi, and Y. Yamada.
\newblock Remarks on fermionic formula.
\newblock In {\em Recent Developments in Quantum Affine Algebras and Related Topics}
(N. Jing and K.C. Misra, Eds.), Contemp. Math. Vol.~248,
pp. 243--291, Amer. Math. Soc., Providence, RI, 1999.

\bibitem[HK]{hkqgcb}
\newblock J.~Hong and S.J. Kang.
\newblock {\em Introduction to {Q}uantum {G}roups and {C}rystal {B}ases}, volume~42 of
  {\em Graduate Studies in Mathematics}.
\newblock Amer. Math. Soc., 2000.

\bibitem[Ion]{Ion}
\newblock B.~Ion.
\newblock Nonsymmetric {M}acdonald polynomials and {D}emazure characters.
\newblock {\em Duke Math. J.}, 116:299--318, 2003.

\bibitem[INS]{ns1}
\newblock M. Ishii, S. Naito, and D. Sagaki.
\newblock Semi-infinite {L}akshmibai-{S}eshadri path model for level-zero extremal weight modules over quantum affine algebras.
\newblock {\em Adv. Math.} 290:967--1009, 2016.

\bibitem[Kac]{Kac}
\newblock V. G. Kac. 
\newblock {\em Infinite Dimensional Lie Algebras}, 3rd Edition, 
\newblock Cambridge University Press, Cambridge, UK, 1990.

\bibitem[Kas]{Kas-OnL}
\newblock M. Kashiwara.
\newblock On level-zero representations of quantized affine algebras.
\newblock {\it Duke Math. J.}, 112:117--175, 2002.

\bibitem[KN]{KN}
\newblock M.~Kashiwara and T.~Nakashima.
\newblock Crystal graphs for representations of the {$q$}-analogue of classical
  {L}ie algebras.
\newblock {\em J. Algebra}, 165:295--345, 1994.

\bibitem[Kho]{Kho}
\newblock A.~Khoroshkin.
\newblock Highest weight categories and Macdonald polynomials.
\newblock Preprint {\tt arXiv:1312.7053}.

\bibitem[KR]{karrym}
\newblock A.~Kirillov and N.~Reshetikhin.
\newblock {R}epresentations of {Y}angians and multiplicities of the inclusion of
  the irreducible components of the tensor product of representations of simple
  {L}ie algebras.
\newblock {\em J. Sov. Math.}, 52:3156--3164, 1990.

\bibitem[LS]{LS} 
\newblock T.~Lam and M.~Shimozono.
\newblock Quantum cohomology of $G/P$ and homology of affine Grassmannian.
\newblock {\em Acta Math.}, 204:49--90, 2010.

\bibitem[LOS]{LOS}
\newblock C. Lecouvey, M. Okado, and M. Shimozono.
\newblock Affine crystals, one-dimensional sums and parabolic Lusztig $q$-analogues. 
\newblock {\em Math. Z.}, 271:819--865, 2012. 

\bibitem[Le1]{Le1}
\newblock C.~Lenart.
\newblock On the combinatorics of crystal graphs, {I}. {L}usztig's involution.
\newblock {\em Adv. Math.}, 211:324--340, 2007.

\bibitem[Le2]{Le}
\newblock C.~Lenart.
\newblock From Macdonald polynomials to a charge statistic beyond type $A$.
\newblock {\em J. Combin. Theory Ser. A}, 119:683--712, 2012.

\bibitem[LL1]{LL}
\newblock C.~Lenart and A.~Lubovsky.
\newblock A generalization of the alcove model and its applications.
\newblock {\em J. Algebraic Combin.}, 41(3):751--783, 2015.

\bibitem[LL2]{LL2}
\newblock  C.~Lenart and A. Lubovsky.
\newblock A uniform realization of the combinatorial $R$-matrix.
\newblock Preprint arXiv:1503.01765.

\bibitem[LNSSS1]{LNSSS}
\newblock C.~Lenart, S.~Naito, D.~Sagaki, A.~Schilling, and M.~Shimozono.
\newblock A uniform model for Kirillov-Reshetikhin crystals I. Lifting the parabolic quantum Bruhat graph. 
\newblock {\em Int. Math. Res. Not.},  2015, no. 7, 1848--1901.

\bibitem[LNSSS2]{LNSSS2}
\newblock C.~Lenart, S.~Naito, D.~Sagaki, A.~Schilling, and M.~Shimozono.
\newblock Quantum {L}akshmibai-{S}eshadri paths and root operators.
\newblock Preprint {\tt arXiv:1308.3529}.
\newblock To appear in {\em Proceedings of the 5th Mathematical Society of
  Japan Seasonal Institute: Schubert Calculus}, Osaka, Japan, 2014.

\bibitem[LP]{LP}
C.~Lenart and A.~Postnikov.
\newblock Affine {W}eyl groups in {$K$}-theory and representation theory.
\newblock {\em Int. Math. Res. Not.}, 1--65, 2007.
\newblock Art. ID rnm038.

\bibitem[LP1]{LP1}
C.~Lenart and A.~Postnikov.
\newblock A combinatorial model for crystals of {K}ac-{M}oody algebras.
\newblock {\em Trans. Amer. Math. Soc.}, 360:4349--4381, 2008.

\bibitem[LeS]{LeS}
\newblock C.~Lenart and A.~Schilling.
\newblock Crystal energy via the charge in types $A$ and $C$.
\newblock  {\em Math. Zeitschrift}, 273:401--426, 2013.

\bibitem[LeSh]{LeSh} 
\newblock C.~Lenart and M.~Shimozono.
\newblock Equivariant $K$-Chevalley rules for Kac-Moody flag manifolds.
\newblock {\em Amer. J. Math.}, 136:1--39, 2014.

\bibitem[Li]{Li}
\newblock P. Littelmann.
\newblock Paths and root operators in representation theory.
\newblock {\it Ann. of Math.} (2), 142:499--525, 1995.

\bibitem[Ma]{macsfh}
\newblock I.~G. Macdonald.
\newblock {\em Symmetric Functions and \mbox{H}all Polynomials}.
\newblock Oxford Mathematical Monographs. Oxford University Press, Oxford,
  second edition, 1995.

\bibitem[Ma2]{machecke}
\newblock I.~G. Macdonald.
\newblock {\em Affine Hecke algebras and orthogonal polynomials}.
\newblock Cambridge Tracts in Mathematics, 157. Cambridge University Press, Cambridge, 2003.

\bibitem[NS1]{NS03}
\newblock S. Naito and D. Sagaki. 
\newblock Path model for a level-zero extremal weight module over a quantum affine algebra.
\newblock {\it Int. Math. Res. Not.}, no.~32, 1731--1754, 2003.

\bibitem[NS2]{NS05}
\newblock S. Naito and D. Sagaki. 
\newblock Crystal of Lakshmibai-Seshadri paths associated to an integral weight of level zero
for an affine Lie algebra.
\newblock {\it Int. Math. Res. Not.}, no.~14, 815--840, 2005.

\bibitem[NS3]{NS06}
\newblock S. Naito and D. Sagaki.
\newblock Path model for a level-zero extremal weight module over a quantum 
affine algebra. I\hspace{-1pt}I. 
\newblock {\it Adv. Math.}, 200:102--124, 2006. 

\bibitem[NS4]{NS-CMP}
\newblock S. Naito and D. Sagaki. 
\newblock Construction of perfect crystals conjecturally 
corresponding to Kirillov-Reshetikhin modules 
over twisted quantum affine algebras.
\newblock {\it Commun. Math. Phys.}, 263:749--787, 2006. 

\bibitem[NS5]{NS-London}
\newblock S. Naito and D. Sagaki. 
\newblock Crystal structure on the set of Lakshmibai-Seshadri paths 
of an arbitrary level-zero shape. 
\newblock {\it Proc. London Math. Soc.}, 96:582--622, 2008. 

\bibitem[NS6]{NS08}
\newblock S. Naito and D. Sagaki. 
\newblock Lakshmibai-{S}eshadri paths of level-zero weight shape and 
one-dimensional sums associated to level-zero fundamental representations.
\newblock {\it Compos. Math.}, 144:1525--1556, 2008.

\bibitem[NS7]{ns2}
\newblock S. Naito and D. Sagaki.
\newblock Demazure submodules of level-zero extremal weight modules 
and specializations of Macdonald polynomials.
\newblock {\em Math. Zeit.}, to appear ({\tt arXiv:1404.2436}).

\bibitem[N]{N}
\newblock H. Nakajima. 
\newblock Extremal weight modules of quantum affine algebras.
\newblock In {\em Representation Theory of Algebraic Groups and Quantum Groups} (T. Shoji et al., Eds.), 
Adv. Stud. Pure Math. Vol.~40, pp. 343--369, Math. Soc. Japan, 2004. 

\bibitem[NY]{NY}
\newblock A.~Nakayashiki and Y.~Yamada.
\newblock Kostka polynomials and energy functions in solvable lattice models.
\newblock{\it Selecta Math. (N.S.)}, 3:547--599, 1997.

\bibitem[Na]{Na}
\newblock K. Naoi.
\newblock Weyl modules, Demazure modules and finite crystals for non-simply laced type.
\newblock {\em Adv. Math.} 229:875--934, 2012.

\bibitem[OS]{OS}
\newblock D. Orr and M. Shimozono.
\newblock Specializations of nonsymmetric Macdonald-Koornwinder polynomials. 
\newblock Preprint {\tt arXiv:1310.0279}

\bibitem[Po]{Po}
\newblock A. Postnikov.
\newblock Quantum {B}ruhat graph and {S}chubert polynomials.
\newblock {\em Proc. Amer. Math. Soc.}, 133:699--709, 2004.

\bibitem[RY]{RY}
A.~Ram and M.~Yip.
\newblock A combinatorial formula for {M}acdonald polynomials.
\newblock  {\em Adv. Math.},  226:309--331, 2011.

\bibitem[Sage]{sage}
\newblock Stein, W. A., and others.
\newblock \textit{Sage Mathematics Software ({V}ersion 5.4).}
\newblock The Sage~Development Team, 2012.
\newblock {\tt http://www.sagemath.org}.

\bibitem[Sage-comb]{sagecombinat}
\newblock The Sage-Combinat community.
\newblock \textit{Sage-Combinat: enhancing Sage as a toolbox for computer exploration in algebraic combinatorics, 2008-2012.}
\newblock {\tt http://combinat.sagemath.org}.

\bibitem[Sa]{Sa} 
\newblock Y. Sanderson.
\newblock On the connection between Macdonald polynomials and Demazure characters. 
\newblock {\em J. Algebraic Combin.}, 11:269--275, 2000.

\bibitem[Se]{Se} 
\newblock A. Sevostyanov.
\newblock Regular nilpotent elements and quantum groups. 
\newblock {\em Comm. Math. Phys.}, 204:1--16, 1999.

\bibitem[SS]{SS}
\newblock A.~Schilling and M.~Shimozono.
\newblock $X=M$ for symmetric powers.
\newblock {\em J. Algebra}, 295:562--610, 2006. 

\bibitem[ST]{ST}
\newblock A.~Schilling and P.~Tingley.
\newblock Demazure crystals, {K}irillov-{R}eshetikhin crystals, and the energy function.
\newblock {\em Electronic J. Combin.}, 19(2) P4, 2012.

\bibitem[Y]{Y}
\newblock S.~Yamane.
\newblock Perfect crystals of $U_q(G^{(1)}_2)$.
\newblock {\em J. Algebra} 210:440--486, 1998. 

\end{thebibliography}
\end{document}